%% file: EmStEm.tex
\documentclass[a4paper]{article}
\usepackage[dvipsnames]{xcolor}
\usepackage{graphicx} %
\usepackage[textwidth=16cm]{geometry}

\usepackage{enumitem}
\usepackage{seqsplit} %

\newcommand{\msf}{\mathsf{f}}
\newcommand{\msB}{\mathsf{B}}
\newcommand{\Om}{\Omega}
\newcommand{\vp}{\varphi}
\newcommand{\weakstar}{\overset{*}{\rightharpoonup}}

\newcommand{\divv}{\operatorname{div}} 
\newcommand{\wt}{\widetilde}
\newcommand{\la}{\langle}
\newcommand{\ra}{\rangle}
\newcommand{\pa}{\partial}
\newcommand{\loc}{\mathrm{loc}}
\newcommand{\ve}{\varepsilon}

\newcommand{\ol}{\overline}

\newcommand{\wh}{\widehat}
\newcommand{\id}{\mathrm{id}}

\newcommand{\R}{\mathbb{R}}

\newcommand{\rel}{\mathrm{rel}}

\newcommand{\dft}{\mathrm{def}}
\newcommand{\dd}{\mathrm{d}}

\newcommand{\weakly}{\rightharpoonup}
\newcommand{\liftpi}{\wt\pi}
\newcommand{\GEO}{\operatorname{Geo}}
\newcommand{\err}{\operatorname{err}}

\newcommand{\TAN}{{\mathrm{tan}}}
\newcommand{\PARAM}{\mathsf{y}}
\newcommand{\embeds}{\hookrightarrow}
\newcommand{\cvx}{\mathsf{cvx}}
\newcommand{\ccv}{\mathsf{ccv}}
\newcommand{\sfF}{\mathsf{F}}
\def\Xint#1{\mathchoice
   {\XXint\displaystyle\textstyle{#1}}%
   {\XXint\textstyle\scriptstyle{#1}}%
   {\XXint\scriptstyle\scriptscriptstyle{#1}}%
   {\XXint\scriptscriptstyle\scriptscriptstyle{#1}}%
   \!\int}
\def\XXint#1#2#3{{\setbox0=\hbox{$#1{#2#3}{\int}$}
     \vcenter{\hbox{$#2#3$}}\kern-.5\wd0}}

\def\dashint{\Xint-}

\title{Dissipative solutions for fibred Wasserstein gradient flows of Cahn--Hilliard-type energies\thanks{Funded by the Deutsche Forschungsgemeinschaft (DFG, German Research Foundation) under Germany’s Excellence Strategy – The Berlin Mathematics Research Center MATH+ (EXC-2046/2, project ID: 390685689) (all authors). Funded by the Deutsche Forschungsgemeinschaft (DFG, German Research Foundation) – Project number 591879973 (K.H.).}}
\author{Georg Heinze, Katharina Hopf, Matthias Liero}

\date{}%

\input{header.tex}
\allowdisplaybreaks
\numberwithin{equation}{section}

\theoremstyle{remark}
\newtheorem*{rem*}{Remark}

    \usepackage[backend=biber, style=alphabetic, giveninits, maxnames=6,alldates=year, maxalphanames=4,url=false,isbn=false]{biblatex}
\begin{document}

\maketitle

\begin{abstract}	
We study fourth-order evolution equations with a fibred Wasserstein gradient-flow structure, describing spatial transport of finitely many or a continuum of types. The driving Cahn--Hilliard-like energy combines a diagonal gradient contribution with nonlinear interactions between types. We prove subsequential convergence of the fibred Jordan--Kinderlehrer--Otto scheme to solutions satisfying an energy-dissipation inequality. 
For general type distributions, the absence of compactness in the type variable is addressed through defect-based solution concepts. 
On $\mathbb R^d, d\le 3$, we construct dissipative solutions satisfying an additional entropy-dissipation inequality. 
On smooth bounded domains in any space dimension, a coercivity estimate for the metric slope allows us to construct energy-dissipating solutions without relying on flow interchange, thereby covering non-convex domains endowed with the intrinsic length distance.
In both cases, a variational comparison of the piecewise constant and De Giorgi interpolants yields a common limiting energy and strong convergence of the difference to zero without requiring strong compactness.
For finitely many types, the defects vanish and the construction yields weak solutions. 
We further establish weak--strong uniqueness of dissipative solutions in every space dimension and of energy-dissipating solutions in dimensions $d\le5$. Combined with the respective existence results, this yields strong convergence of the variational approximations whenever a sufficiently regular, strictly positive solution exists.
\end{abstract}

{\small
\noindent
\emph{Keywords: }Fibered Wasserstein distance, intrinsic cost, fourth-order PDEs, Cahn–Hilliard systems, minimizing movements, energy-dissipating solutions, weak--strong uniqueness, continuous type distributions.
\\[1em]
\emph{MSC[2020]:}
35K41, %
35K65, %
49Q22, %
35A15 %
}

\section{Introduction}

We study fourth-order evolution equations describing the spatial transport of interacting populations with a fixed distribution of types. The dynamics are driven by a common Cahn--Hilliard-like energy and formally have the structure of a gradient flow with respect to a fibred Wasserstein distance. This framework allows for mixtures of finitely many species, populations with continuously distributed traits, and, more generally, combinations of singular and continuous type distributions.
A motivation comes from continuum descriptions of cell adhesion, sorting and segregation. The type variable $y\in Y$ may encode cell type or adhesion-molecule expression \cite{ST1994ESCS}, a differentiation status or a gene expression profile \cite{Velten*2017HHSC, CHBIH2008TNCL}, or a collection of intracellular regulatory variables \cite{Bendall*2011SMCD}. We here treat these characteristics as fixed labels and assume that the mass of each type is conserved. Our model thus describes spatial rearrangement in regimes where changes in cellular state and population size can be neglected.

In~\cite{FaBaCa2024LCMC}, Falc\'o, Baker and Carrillo formally derive a local fourth-order system from nonlocal adhesive interactions in a short-range limit.
Their numerical simulations show that the local model reproduces a range of cell-sorting patterns associated with differential adhesion.
A related phenomenological fourth-order model of tissue growth was proposed earlier by Chertock, Degond, Hecht and Vincent~\cite{CDHV_2019}, where repulsive interactions induce segregation of initially mixed cell populations. The model is derived by a Wasserstein gradient-flow ansatz for the conservative part, with the growth terms being incorporated afterwards.
For a second-order tumour growth model with continuously distributed phenotypes $y\in[0,1]$, where mass transport is driven by the gradient of a common population pressure, we refer to David~\cite{David_2023}.

\paragraph{Model description and motivation}
In the present work, we consider $Y:=\mathbb{R}^{d'}$ and describe the types by a fixed distribution $\nu\in\calP(\mathbb{R}^{d'})$, the support of which can thus be interpreted as the set of all types.
A state is then characterised by a measure
$\mu$ in the space $\calP_{2,\nu}(X\times Y)$ of probability measures on $X\times Y$
with $Y$-marginal $\nu$ and finite spatial second moment, where we consider as the physical space $X=\overline\Omega\subset\bbR^d$ for a smooth bounded domain $\Omega$ or $X=\bbR^d$.
Disintegrating $\mu$ with respect to $\nu$ yields a Borel family $y\mapsto\mu^y$ describing the spatial distribution of type $y$ so that $\mu=\int_Y\mu^y\,\dd\nu(y)$.
The fibred Wasserstein distance is then given by
\[\mathcal{W}_{2,X,\nu}^2(\mu,\widehat\mu):=\int_Y \mathcal{W}_{2,X}^2(\mu^y,\widehat\mu^y)\,\dd\nu(y),\]
where $\calW_{2,X}$ denotes the quadratic Wasserstein distance
with the intrinsic length distance $d_X$.
For $X=\mathbb R^d$ or if $\Omega$ is convex, $d_X$ coincides with the Euclidean distance, and, in this case, following~\cite{peszek2023heterogeneous}, we simply write $\mathcal{W}_{2,\nu}$.
The fibred distance encodes the restriction that transport preserves the type variable and, for finitely many types,
reduces to a weighted product Wasserstein distance. 

In this manuscript we study the convergence of the fibred JKO scheme towards solutions of fourth-order evolution equations 
with a formal fibred Wasserstein gradient flow structure 
\begin{align*}%
	\partial_t\mu =-\operatorname{grad}_{\calW_{2,X,\nu}}\calE(\mu).
\end{align*}
The driving energy is of Cahn--Hilliard type and takes the form
$\calE=\calE_1+\calE_0$, where in the present study the gradient energy $\calE_1$ is diagonal in the type
variable.
Writing $\mu=u\scrL^d\otimes\nu$, we set
\begin{align*}
    \calE_1(\mu)
    &:=\frac12\int_X\int_Y
        \alpha(y)|\nabla_xu(x,y)|^2\,\dd\nu(y)\dd x,\\
    \calE_0(\mu)
    &:=\int_X \sff(U(x))\,\dd x
      +\frac12\int_X\int_{Y\times Y}
        \beta(y,y')u(x,y)u(x,y')\,
        \dd(\nu\otimes\nu)(y,y')\dd x,
\end{align*}
where
\[ U(x):=\int_Y\kappa(y)u(x,y)\,\dd\nu(y), \qquad
    \sff(s):=\frac{\vartheta}{4}s^2(1-s)^2,
    \qquad \vartheta\ge0.\]
Here, $\alpha,\kappa\in L^\infty(Y,\nu)$ satisfy
$\alpha\ge\alpha_*>0$ and $\kappa\ge\kappa_*>0$,
while $\beta$ is a $\nu\otimes\nu$-essentially bounded, symmetric, non-negative interaction kernel.
We set $\calE(\mu)=+\infty$ unless $\mu=u\scrL^d\otimes\nu$ with
$u\in L^2(Y,\nu;H^1(X))$ and the displayed integrals are finite.

The gradient contribution acts separately on each type, whereas
the lower-order energy couples the types through the weighted
total density $U$ and the pairwise repulsion.
For $\vartheta>0$, the (non-convex) double-well potential describes the preferred states $U=0$ and $U=1$ of the mixture with strength controlled by $\vartheta$.
The repulsion term penalises spatial overlap between types
with $\beta(y,y')\ge0$; note that despite the non-negativity of the kernel,
this quadratic contribution also need not be convex.

In terms of the density, the associated gradient-flow equation takes the PDE form
\begin{subequations}\label{eq:PDEsystem.intro}
\begin{align}\label{eq:CE.intro}
\partial_tu+\operatorname{div}_x j_u=0,\qquad j_u=-u\nabla_x w_u,
\end{align}
where the chemical potential is given by
\begin{align}\label{eq:chempot.intro}
w_u(x,y)=-\alpha(y)\Delta_xu(x,y)+\kappa(y)\sff'(U(x))+\int_Y\beta(y,y')u(x,y')\,\dd\nu(y').
\end{align}
\end{subequations}
On bounded domains, the equations are complemented by the boundary conditions $\partial_nu=0$ and $j_u\cdot n=0$.

As an example, consider the case where the measure $\nu$ has finite support so that it can be written in the form $\nu = \sum_{i=1}^nm_i\delta_{y_i}$, where $m_i>0$ are positive weights with $\sum_{i=1}^nm_i=1 $ and $y_i\in Y$ correspond to species labels. In this setting, we can identify a measure $\mu\in\calP_{\nu}(X{\times}Y)$ and its $\scrL^d{\otimes}\nu$-density $u$ with a vector of measures $(\mu_1,\ldots,\mu_n)\in \calM(X)^n$ and a vector of probability densities $(u_1,\ldots, u_n)\in L^1(X)^n$, respectively. More precisely, we have that
$\mu = \sum_{i=1}^n m_i u_i\scrL^d\otimes \delta_{y_i}$. Consequently, the driving energy takes the form
\begin{equation}
\mathcal E(\mu)
=
\sum_{i=1}^n
\int_X \frac{m_i\alpha_i}{2} |\nabla u_i(x)|^2\,\dd x
+
\int_X
\msf\left(
\sum_{i=1}^n m_i\kappa_i u_i(x)
\right)\,\dd x
+\sum_{i,j=1}^n
\frac{m_i m_j\beta_{ij}}{2}
\int_Xu_i(x) u_j(x)\,\dd x,
\end{equation}
where $\alpha_i= \alpha(y_i)$, $\kappa_i = \kappa(y_i)$, $\beta_{ij} = \beta(y_i,y_j)$.
Therefore, the evolution equation~\eqref{eq:PDEsystem.intro} takes the form
of the coupled system
\[
\partial_t u_i
=
\operatorname{div}_x
\Bigg(
u_i\nabla_x
\bigg(
{-}\alpha_i\Delta_x u_i
+
\kappa_i{\sf f}'(U)
+
\sum_{j=1}^n m_j\beta_{ij}u_j
\bigg)
\Bigg),
\qquad i=1,\dots,n.
\]

The choice of the lower-order energy $\mathcal E_0$ is motivated by effective incompressible multi-phase descriptions. 
Chertock et al.~\cite{CDHV_2019} formally obtain a two-species Hele–Shaw growth model without surface tension from their fourth-order model with active repulsion,
while Kroemer and Laux~\cite{KL_2022} derive, under an energy convergence assumption, a Hele--Shaw flow with surface tension from a Cahn--Hilliard equation with linear mobility and double-well potential.
The present system combines inter-species repulsion with a double-well potential for the total (weighted) density $U$, and, in a suitable scaling regime, we expect an effective finite-species sharp-interface description of multi-phase Hele--Shaw type with surface tension.
In the context of cell sorting, related sharp-interface descriptions have recently been studied using Keller--Segel-type approximations~\cite{JangMellet_2026}.

\paragraph{Analytical contributions}
The present manuscript is devoted to the analysis of systems with a diagonal structure of the gradient energy $\calE_1$.
The diagonal structure allows us to express the flux distributionally in terms of an
energy stress tensor, which only requires control of the densities and
their spatial gradients, but does not involve second- or higher-order derivatives. 
More precisely, the contribution of the Dirichlet energy to the constitutive flux admits the representation, cf.~\cite{Glasner_2003,KL_2022},
\begin{align}\label{eq:flux.tensorial.intro}
\alpha(y)u\nabla_x\Delta_xu=\operatorname{div}_x\mathbb T_1(u)
+\alpha(y)\nabla_x\Big(\operatorname{div}_x(u\nabla_xu)-|\nabla_xu|^2\Big),
\end{align}
where $\mathbb T_1$ denotes the energy stress tensor associated with the Dirichlet energy, given by
\[\mathbb T_1(u):=\alpha(y)\bigg(
\frac12|\nabla_xu|^2 \mathbb{I}_d-\nabla_xu\otimes\nabla_xu\bigg).\]
Our primary aim is to establish subsequential convergence of the fibred JKO scheme to solutions of the fourth-order equations, within a framework that captures the energy-dissipation mechanism and allows for a weak--strong stability estimate. 
The special structure~\eqref{eq:flux.tensorial.intro} allows us to carry out the variational construction using intrinsic transport on general smooth bounded domains without the usual convexity restriction on the spatial geometry, while ensuring weak--strong uniqueness for the same solution concept.
A specific difficulty due to the fibred setting is the absence of a regularising mechanism in the type variable,
so that higher-order spatial control does not generally yield the compactness needed to identify the nonlinear flux.
To cope with this circumstance, we introduce solution concepts
involving non-negative defect quantities,
where the limiting density-flux pair $(u,j)$ 
satisfies the continuity equation $\pa_tu+\divv j=0$ in $\mathcal{D}'$ and, for a.e.\ $t>0$, the energy-dissipation inequality
\begin{align}\label{eq:edi.intro}
	E(t)+\int_0^t\int_X\int_Y\frac{|j|^2}{u}\,\dd\nu\dd x\dd s\le \calE(\mu_0),\qquad E(t)\ge\calE(\mu(t)).
	\end{align}
Here $E$ is the limiting value of the energies along approximate solutions, and $E^{\mathrm{def}}(t):=E(t)-\mathcal E(\mu(t))\ge0$ denotes the energy defect, measuring the possible loss of energy under weak convergence. For bounded domains $\Om\subset\mathbb{R}^d$ and $X=\ol\Om$, the difference $j-j_u$ between limiting and constitutive flux can be controlled by the energy defect thanks to the Poincar\'e--Wirtinger inequality that allows to control the density $u$ in the term $\alpha\nabla_x\operatorname{div}_x(u\nabla_xu)$ on the right-hand side of~\eqref{eq:flux.tensorial.intro}. For $X=\mathbb{R}^d$ a further defect quantity providing stronger control at infinity is needed in general.
A natural choice is the entropy defect associated to the entropy functional $\mathcal{H}(\mu)=\int_X\int_Y u\log u\,\dd\nu\dd x$, $\mu=u\dd x\dd\nu$, for which a suitable dissipation inequality can be derived along the scheme via flow interchange.

We first study the variational scheme on the whole space $\mathbb R^d$
 where optimal transport maps are available and where the theory on fibred optimal transport was carefully developed~\cite{peszek2023heterogeneous}. 
 Because of difficulties at spatial infinity, 
 the compactness and convergence arguments, however, are different from our analysis on bounded domains and to cope with the quartic potential part, we restrict to $d\le3$ whenever $\supp\nu$ is not a singleton.
A key step in the construction on $\mathbb{R}^d$ is a variational comparison of the piecewise constant and De Giorgi interpolants $u_\tau$ and $\wt u_\tau$, that exploits the dominant convex part of the energy, where convexity is understood in the traditional affine sense, see Lemma~\ref{lem:energy-hierarchy}.
It implies that the energies along the two interpolants converge to the same limit and allows us to further conclude the strong convergence of the difference
\begin{align}\label{eq:strong-conv.difference}
u_\tau-\wt u_\tau\longrightarrow0\quad\text{strongly in }L^2((0,T)\times Y, \scrL^1\otimes\nu;H^1(X)),
\end{align}
without relying on compactness. 
Together with the Euler--Lagrange equations, this yields a common weak limit of the associated fluxes,
obeying the  energy-dissipation inequality~\eqref{eq:edi.intro}.
Flow interchange additionally yields an entropy dissipation inequality, and spatial compactness for the piecewise constant interpolant, that allows us to identify the flux term and recover a weak formulation whenever $\supp\nu$ is finite.

On smooth bounded domains, we do not rely on flow interchange for compactness, which allows us to admit non-convex geometries with intrinsic transport.
For the compactness and convergence analysis of the scheme, we exploit a slope-coercivity estimate for Cahn--Hilliard-type energies with respect to the Wasserstein metric,
adapted from the recent preprint by Gau and two of the authors~\cite{GHL_preprint}, which in particular yields uniform bounds in $L^2((0,T)\times Y,\scrL^1\otimes\nu;H^2(\Omega))$ for both interpolants.
In this way, we are able to construct energy-dissipating solutions without relying on the standard heat-flow interchange argument, which is not available for non-convex domains.
Let us mention that, in order to cover the fibred setting in this case, we need to establish basic measurable selection results for Kantorovich potentials and optimal dynamic plans for the intrinsic cost. 
We further derive a first-variation formula for the squared (fibred) intrinsic Wasserstein distance on non-convex bounded domains.

In both settings, when $\nu$ has finite support, the defects vanish and, after passage to a subsequence, the schemes converge to weak solutions satisfying the energy-dissipation inequality and, if $X=\mathbb{R}^d$, an entropy-dissipation inequality.

Furthermore, we establish weak--strong uniqueness results for both settings and general base measure $\nu\in \mathcal{P}(Y)$, which are consistent with the existence results obtained by the minimising movements approximation.

Finally, let us mention that non-diagonal gradient energies, relevant for modelling differential adhesion, are studied in a parallel work using a functional setting based on Young measures.

\paragraph{Related work} 
The geometric framework for fibred optimal transport on $X=\mathbb{R}^d$ and an associated gradient-flow theory in the spirit of~\cite{AGS} were developed by Peszek and Poyato~\cite{peszek2023heterogeneous} for a class of energies generating evolution variational inequality solutions. 
Their abstract well-posedness theory is based on $\lambda$-convexity along generalised geodesics of the admissible energy functionals, which we cannot rely on here
due to the lack of geodesic semi-convexity of the Dirichlet energy in the Wasserstein space.
A minimising movements scheme in the fibred Wasserstein space for a geodesically non-convex 
energy was studied by Canc\`es, Matthes, Medina and Schmitzer~\cite{CMMS2024CCWGa}, who established subsequential convergence to the associated second-order incompressible drift--diffusion equations.
In that model, the nonlinearity stems from a common pressure $\Pi=\Pi(t,x)$, independent of $y$, which acts as a Lagrange multiplier for the $X$-marginal constraint related to incompressibility. As a consequence, in this equation the limiting flux can be identified without strong convergence of the density with respect to $y$.

For scalar models, variational schemes for degenerate fourth-order equations in transport metrics were systematically studied by Matthes, McCann and Savar\'e~\cite{MMS} and Lisini, Matthes and Savar\'e~\cite{LMS_2012}, who established subsequential convergence to weak solutions.
In these works, the flow-interchange technique constitutes a key tool, among other things,
for obtaining improved control and compactness of the approximate solutions.
For related results on fourth-order systems with two components we refer to~\cite{LM_2014,CEFF_2024}; for adjusted JKO schemes incorporating a pointwise volume constraint, see~\cite{CancesMatthesNabet2019,CancesMatthes2023}.

A sharp energy-dissipation inequality in the spirit of~\eqref{eq:edi.intro} was not established in the gradient-flow approaches above.
 To the authors' knowledge, in the variational approximation of degenerate fourth-order equations, such an inequality first appeared in the paper~\cite{KL_2022} by Kroemer and Laux for the scalar Cahn--Hilliard-type equation with linear mobility. Their approach is based on the De Giorgi variational interpolant. However, we have not been able to verify the asserted $L^2_tH^2_x$ control from flow interchange used in the proof of~\cite[Proposition~3.12]{KL_2022} to get compactness for the De Giorgi interpolant.
 The recent work by Gau and two of the authors~\cite{GHL_preprint} considering variational schemes for dynamic transport metrics of a Cahn--Hilliard-type energy does include a derivation of the energy-dissipation inequality for convex bounded domains. 
In the present work, besides the fibred setting, a new challenge compared to~\cite{GHL_preprint} concerns the more involved spatial setting, covering both non-convex bounded smooth domains and the unbounded whole-space case $\mathbb{R}^d$. A natural question, which we leave for future research, concerns 
 the extension of our results to the case of non-convex smooth domains that are {\em unbounded}. 

\paragraph{Outline} The manuscript is organised as follows. Section~\ref{sec:Prelim} introduces the functional setting and solution concepts and states the main results. In Section~\ref{sec:diagE}, we construct dissipative solutions on $\R^d$ via the fibred JKO scheme, establishing the energy and entropy-dissipation inequalities and the flux-defect control. In Section~\ref{sec:non-convex}, we construct energy-dissipating solutions on smooth bounded, possibly non-convex domains with intrinsic optimal transport  by means of a slope-coercivity estimate. In Section~\ref{sec:uniq.wkstr} we establish weak--strong uniqueness results for the solution concepts considered. Finally, the appendix collects auxiliary results on fibred optimal transport, interpolation inequalities, and convergence in negative Sobolev spaces.

\section{Main results}\label{sec:Prelim}

We first introduce the fibred optimal transport framework and the energy functional, then define the solution concepts and state the main results. We distinguish between the whole space and smooth bounded domains,
where transport is measured using the intrinsic length distance.

\subsection{Notations and preliminaries}

We first collect the notation used throughout the paper. 
We consider closed sets $X\subseteq\R^d$ and $Y=\mathbb R^{d'}$ 
and denote by $\calP(X)$ and $\calP(Y)$ the
spaces of Borel probability measures on $X$ and $Y$, respectively.
The spaces of non-negative finite Borel measures on $X$ or $Y$ are denoted by $\calM_+(X)$
or $\calM_+(Y)$. The spaces of signed measures are called $\calM(X)$
or $\calM(Y)$, respectively. 
We consider  a fixed Borel probability measure $\nu\in\calP(Y)$ and assume that
 \begin{equation}
 \label{eq:assNuCompSupp}
 \supp \nu \subset Y\text{ compactly.}
 \end{equation}

In the analysis below, we will distinguish two cases, namely
\begin{equation}
\label{eq:domainCases}
X = \R^d\quad\text{or}\quad X= \ol \Om,~\text{where }\Omega\subset\R^d\text{ is a smooth bounded  domain}.
\end{equation}

In the remaining text, we will write $f\lesssim_{a,b} g$ for non-negative quantities $f,g$ if there exists a constant $C>0$ that is allowed to depend on $a$ and $b$ such that $f\leq C g$.

\paragraph{Fibred Wasserstein distance}
The first central ingredient is the so-called
\emph{fibred Wasserstein distance} introduced in \cite{peszek2023heterogeneous},
see also \cite{CHSW2025CWDA}, where the name \emph{conditional Wasserstein distance} is used. Let $P_Y:X\times Y\to Y$ be the projection onto the second factor, i.e., $P_Y(x,y) = y$. We call
a measure $\mu\in\calP(X{\times}Y)$ \emph{fibred over $\nu$} if $(P_Y)_\#\mu=\nu$.
The set of all such measures is denoted by
\[
\calP_\nu(X{\times}Y)
\coloneqq 
\{\mu\in\calP(X{\times}Y):(P_Y)_\#\mu=\nu\}.
\]
By the disintegration theorem, every
$\mu\in\calP_\nu(X{\times}Y)$ admits a disintegration with respect to
$\nu$, which is denoted by $\mu=\int_Y\mu^y\,\mathrm d\nu(y)$, where $\mu^y\in\calP(X)$ is defined for $\nu$-almost every $y\in Y$,
the map $y\mapsto\mu^y$ is Borel measurable, and the family
$\{\mu^y\}_{y\in Y}$ is unique up to $\nu$-null sets. More precisely, for
every non-negative Borel function
$\phi:X\times Y\to[0,\infty]$,
\[
\int_{X\times Y}\phi(x,y)\,\mathrm d\mu(x,y)
=
\int_Y\left(\int_X\phi(x,y)\,\mathrm d\mu^y(x)\right)\mathrm d\nu(y).
\]
We further denote by $\calP_{2,\nu}(X{\times}Y)$ the subset of fibred probability measures with integrable second moments with respect to $X$,
i.e.,
\[
\calP_{2,\nu}(X{\times}Y)
\coloneqq 
\left\{
\mu\in\calP_\nu(X{\times}Y):
\textstyle\int_{X\times Y}|x|^2\,\mathrm d\mu(x,y)<\infty
\right\}.
\]
Equivalently, if $\mu=\int_Y\mu^y\,\mathrm d\nu(y)$, then
$\int_Y\int_X|x|^2\,\mathrm d\mu^y(x)\,\mathrm d\nu(y)<\infty$.

Following \cite{peszek2023heterogeneous}, we define the
\emph{fibred $2$-Wasserstein distance} on
$\calP_{2,\nu}(X{\times}Y)$ by
\[
\mathcal W_{2,\nu}(\mu_1,\mu_2)
\coloneqq 
\left(
\int_Y
\mathcal W_{2,X}(\mu_1^y,\mu_2^y)^2\,\mathrm d\nu(y)
\right)^{1/2},
\]
where $\mathcal W_{2,X}$ denotes the usual $2$-Wasserstein distance on
$\calP_2(X)$. For an overview of the properties of this metric used below, we refer to
Appendix~\ref{ssec:Notation}.

Later on, see Subsection~\ref{ss:nonconvex}, we will discuss the case $X=\ol\Om$ with $\Om\subset\mathbb{R}^d$ a smooth domain, possibly non-convex, with intrinsic length distance $d_X$. In the case that $\Omega$ is convex, 
the results from~\cite{peszek2023heterogeneous} used below remain valid. Indeed, since every $\mu\in\mathcal P_2(X)$ can be identified with its push-forward under the inclusion $X\hookrightarrow\mathbb R^d$, we may regard $\mathcal P_2(X)$ as an isometrically embedded closed subspace of $\mathcal P_2(\mathbb R^d)$. For the non-convex case, necessary results are collected in Subsection~\ref{ss:DiffIntrinsicWasserstein}.

\paragraph{Function spaces} 
 Throughout the text, unless specified otherwise, all Lebesgue and Bochner spaces involving $Y$ are understood with respect to the measure $\nu$. Spaces of continuous functions on $Y$ are understood with respect to the topology of $Y=\mathbb{R}^{d'}$. 
  Furthermore, throughout these sections, $\nabla_x=\nablax$ denotes the Euclidean gradient with respect to the $x$-variable, where $X=\mathbb{R}^d$ or $X=\ol\Om$ for a domain $\Om\subset\mathbb{R}^d$.
For any $\phi\in \rmL^\infty([0,T]\times X\times Y)$ and $u\in\rmL^1([0,T]\times X\times Y)$, we define the function
\begin{equation*}
    \langle u,\phi\rangle_Y(t,x) \coloneqq  \int_Y\phi(t,x,y)u(t,x,y)\,\dd\nu(y).
\end{equation*}
Often $u$ will be the density of a measure $\mu$ with respect to $\scrL^{d+1}{\otimes}\nu$, in which case we may also use the notation $\langle \mu,\phi\rangle_Y \coloneqq \langle u,\phi\rangle_Y$ to emphasize the measure $\mu$.

In the case of a bounded domain, i.e., $X = \ol\Om$, we also use the notation
\begin{align*}
    C_{\mathrm{tan}}^k(\ol\Omega;\mathbb R^d):= \left\{
\xi\in C^k(\ol\Omega;\mathbb R^d):\xi\cdot n=0\ \text{on }\partial\Omega\right\},
\end{align*}
where $n$ denotes the outer unit normal vector on $\partial\Omega$.

\paragraph{Energy functional}
The evolution is driven by a Cahn--Hilliard-type energy
$\calE : \calP_{2,\nu}(X{\times} Y) \to [0,\infty]$, which we decompose as
$\calE = \calE_1 + \calE_0$. The gradient contribution $\calE_1$ acts separately
on each type, while the lower-order contribution $\calE_0$ couples
the types, see below. We define the gradient
contribution by
\begin{equation}
\label{eq:defGradEnergy}
\calE_1(\mu) = 
\begin{cases}
\displaystyle\frac{1}{2}\int_X\int_Y \alpha(y)|\nabla_x u(x,y)|^2\,\dd\nu\dd x&\text{if }\mu = u\scrL^d{\otimes}\nu ~\text{with}~\nabla u\in L^2(X{\times} Y;\R^d),\\
+\infty&\text{otherwise}.
\end{cases}
\end{equation}
We assume that the coefficient satisfies $\alpha\in L^\infty(Y)$ and is uniformly bounded from below such that
\begin{equation}\label{eq:assAlpha}
0< \alpha_* \leq \alpha(y) \leq \alpha^*:=\|\alpha\|_{L^\infty(Y)}\text{ for $\nu$-a.e. }y\in Y.
\end{equation}
The lower-order contribution $\calE_0: \calP_{2,\nu}(X{\times}Y)\to[0,\infty]$ couples the types through two terms:
a double-well potential of the weighted total density
$U(x)=\int_Y\kappa(y)u(x,y)\,d\nu(y)$ and a pairwise repulsion energy
that penalises spatial overlap between types. More precisely, we set 
\begin{subequations}
\label{eq:defLowerEnergy}
\begin{align}
\calE_0(\mu) &= \calF(\mu) + \calR(\mu),\quad\text{where}
\\\label{eq:defEnergyDoubleWell}
\calF(\mu) &= \begin{cases}\displaystyle
\int_X \sff(U(x))\,\dd x &\text{if }\mu = u\scrL^d{\otimes}\nu,\\
+\infty &\text{otherwise},
\end{cases}\quad\text{with } U(x):= \int_{Y}\kappa(y) u(x,y)\, \dd \nu,
\\\label{eq:defEnergyRepulsion}
\calR(\mu) &= 
\begin{cases}\displaystyle
\frac12\int_X\int_{Y\times Y}
        \beta(y,y')u(x,y)u(x,y')\,
        \dd(\nu{\otimes}\nu)(y,y')\dd x&\text{if }\mu = u\scrL^d{\otimes}\nu,\\
+\infty &\text{otherwise}.
\end{cases}
\end{align}
\end{subequations}
We impose the following assumptions on the coefficients $\kappa$ and $\beta$ 
and the potential $\sff$
\begin{subequations}
\begin{align}
\label{eq:assKappa}
&0< \kappa_* \leq \kappa(y) \leq \kappa^*:=\|\kappa\|_{L^\infty(Y)}\text{ for $\nu$-a.e. }y\in Y,
\\[0.25em]\label{eq:assBeta}
&\beta\in L^\infty(Y{\times} Y)\text{ with }\beta(y,y') = \beta(y',y)\geq 0~\forall\, y,y'\in Y,
\\\label{eq:DoubleWellSplit}
&\sff(s) = \frac{\vartheta}{4}s^2(1{-}s)^2~\text{for }s\geq 0~\text{with}~
\sff=\sff_\cvx+\sff_\ccv,~\sff_\cvx(s)=\sff(s){+}\frac\vartheta2s^2,\text{ and }\sff_\ccv(s)=-\frac\vartheta2s^2.
\end{align}
\end{subequations}
\begin{subequations}
Throughout, we allow $\vartheta=0$, in which case the double-well
contribution vanishes. Whenever an argument requires $\vartheta>0$,
in particular to exploit the coercivity of the potential, we state
this explicitly.

With the particular choice of $\sff$, it holds that $\sff_\cvx\ge0$, $\sff_\ccv\le0$, and
\begin{equation}\label{eq:potential-curvature}
	\sff_\cvx''(s)=\frac32\vartheta\big(s^2+(s-1)^2\big)\gtrsim\vartheta (1+s^2),\qquad \sff_\ccv''(s)=-\vartheta.
\end{equation}
For $U,V\in\mathbb R$ and $Z=U-V$, direct calculation gives
\begin{align}
	\sff(U)-\sff(V)-\sff'(V)Z&=\frac{3\vartheta}2Z^2\left(V-\frac12+\frac Z3\right)^2+\frac\vartheta{12}Z^4-\frac\vartheta8Z^2\ge\frac\vartheta{12}|Z|^4-\frac\vartheta8|Z|^2,\label{eq:pot-rel.coerc}\\
	\sff''(U)-\sff''(V)&=3\vartheta(U+V-1)Z,	\label{eq:pot.aux}\\
	\sff'(U)-\sff'(V)-\sff''(V)Z&=\vartheta\bra[\bigg]{3V-\frac32}Z^2+\vartheta Z^3.\label{eq:pot.exp}
\end{align}
Since $\sff_\ccv$ is quadratic, \eqref{eq:pot.aux}--\eqref{eq:pot.exp} also hold with $\sff$ replaced by $\sff_\cvx$. 
\end{subequations}
\begin{rem*}
We use the quartic double-well potential for definiteness. More general
potentials can be treated by the same strategy, provided they admit
a similar decomposition and satisfy suitable growth, coercivity, and
regularity assumptions.    
\end{rem*}

In the remaining text, we will also use the shorthand notation
\begin{equation}\label{eq:defOpB}
(\mathsf Bu)(x,y)\coloneqq\int_Y\beta(y,y')u(x,y')\,\dd\nu(y')\quad\text{such that}\quad
    \mathcal{R}(u) = \frac{1}{2}\int_{X}(\mathsf{B}u,u)_{L^2(Y)}\dd x.
\end{equation}

\paragraph{Entropy functional}
In parts of the analysis, we supplement the energy estimates with
estimates for the logarithmic entropy. Its formal fibred Wasserstein
gradient flow is the heat equation acting separately on each type,
which we use through flow interchange to obtain additional spatial
regularity and an entropy-dissipation inequality, see, e.g., Definition~\ref{def:diss_sol} below.
The entropy functional $\calH:\calP_{2,\nu}(X{\times} Y)\to \R_\infty\coloneqq \R\cup \{+\infty\}$, used in parts of the manuscript, is defined as
\begin{equation}
\label{eq:defEntropy}
\calH(\mu) \coloneqq 
\begin{cases}\displaystyle
\int_X\int_Y u(x,y)\log u(x,y) \,\dd\nu(y)\dd x&\text{if }\mu = u \scrL^d{\otimes}\nu,\\
+\infty&\text{otherwise.}
\end{cases}
\end{equation}
Note that the integrability of the second moments and the compact support of $\nu$ yields that $\calH(\mu)$ is bounded from below for any $\mu\in\calP_{2,\nu}(X\times Y)$.

\subsection{Main theorems}

We now introduce the solution concepts and state our main results. 

We specify the set of admissible initial data as
\begin{equation}
\label{eq:defSetAnu}
\calA_\nu\coloneqq \big\{\mu=u\scrL^d{\otimes}\nu \in \calP_{2,\nu}(X{\times}Y)\,
\big|\, u\in L^2(Y;H^1(X)),\, U= \langle\kappa, u\rangle_Y\in L^4(X)\big\}.
\end{equation}
Note that every $\mu\in \calA_\nu$ has finite entropy $\mathcal{H}(\mu)<\infty$. Indeed, the $L^2$ condition controls its positive part, while the second moment controls its negative part.

We consider the energy functional $\calE=\calE_1 + \calE_0$ with $\calE_1$ and $\calE_0$ as in \eqref{eq:defGradEnergy} and \eqref{eq:defLowerEnergy}. 
In particular, we have the decomposition of the flux
$j_u= j_u^{(1)}+j_u^{(0)}$ with $j_u^{(1)}\simeq -u \nabla\frac{\delta\calE_1}{\delta u}$ and $j_u^{(0)}\simeq -u \nabla\frac{\delta\calE_0}{\delta u}$.
The theory below exploits the fact that the flux contribution $j_u^{(1)}$ can be expressed, using the associated energy stress tensor (see~\cite{Glasner_2003,KL_2022}), as
\begin{align}\label{eq:flux-tensor}
j_u^{(1)}=\alpha(y)u\nabla_x\Delta_x u
	&=\alpha(y)\divv_x\Big(\big(u\Delta_x u+\frac12|\nabla_x u|^2\big)I-\nabla_x u\otimes\nabla_x u
		\Big)\nonumber
	\\&=\alpha(y)\divv_x\Big(\big(\divv_x(u\nabla_x u)-\frac12|\nabla_x u|^2\big)I-\nabla_x u\otimes\nabla_x u
	\Big).
\end{align}
More precisely, given $ \mu\in\mathcal{P}_{2,\nu}(X{\times} Y)$ with $\calE(\mu)<\infty$ and 
$\scrL^{d}\otimes \nu$-density $u$, we define the flux 
$j_u$ distributionally by its action on
$\xi\in L^\infty\big(Y;C_{c}^2(X;\mathbb R^d)\big)$ if $X=\mathbb{R}^d$, via
\begin{align}\label{eq:flux.distrib}
\begin{split}
	\langle j_u,\xi\rangle\coloneqq{}&\int_X\int_Y
	\alpha(y)\bigg[
		\frac12|\nabla_xu|^2\divv_x\xi
		+(\nabla_xu\otimes\nabla_xu):\nabla_x\xi
		+u\nabla_xu\cdot\nabla_x\divv_x\xi
	\bigg]\,\dd\nu\dd x
	\\&
     + \int_X\int_Y \big[(\msf'(U)\kappa(y)+\mathsf{B} u)
        \bra[\big]{\nablax u\cdot\xi + u\divv_x\xi}\big]\,\dd\nu\dd x.
        \end{split}
\end{align}
If $X=\ol\Om$ is the closure of a bounded smooth domain, $j_u$ is defined via~\eqref{eq:flux.distrib} 
for test vector fields $\xi\in L^\infty(Y;C_{\mathrm{tan}}^2(\ol\Omega;\mathbb R^d))$.

\begin{remark}[Scalar case]\label{rem:Singleton}
For a single species, i.e. $\nu= \delta_{y_0}$ for $y_0\in Y$, 
we can compute for sufficiently regular $U(x)= \kappa(y_0)u(x,y_0)$ and a compactly supported (resp.\ tangential) smooth vector field $\xi$
\begin{align}\label{eq:flux.pot.def.scalar}
\int_X \sff'(U)\kappa(y_0)\big(\nabla_xu(\cdot,y_0)\cdot \xi + u(\cdot,y_0)\divv\xi\big)\dd x=\int_X \big({-}\sff(U)\divv\xi + U\sff'(U)\divv\xi\big)\dd x.
\end{align}
See, e.g., \cite[Lemma 3.8]{KL_2022}. In particular, we can write this expression free of derivatives of $U$.  
In the fibred setting, this reduction is no longer valid since, in general, $\int_Y\kappa(y)\nabla u(x,y)\cdot \xi(x,y)\dd \nu(y) \neq \nabla U(x)\cdot\int_Y\xi(x,y)\,\dd\nu(y)$.
\end{remark}

Next, for convenience we introduce an energy-dissipation property common to both
solution concepts. To account for a possible loss of energy under weak
convergence, we introduce a non-increasing function $E(t)$ that bounds
the energy $\calE(\mu_t)$ from above; for related ideas, cf.\ e.g.\ the solution
framework in \cite[Section~2.1]{FLJ_2026} for the Euler system of gas dynamics.
The relation between the flux $j$ and the constitutive flux $j_u$ will be specified separately through suitable defect estimates, cf.\ Definition \ref{def:diss_sol} below.

\begin{definition}[Energy-dissipation property]\label{def:e-diss.property} Let $\calE=\calE_1+\calE_0$ be given by~\eqref{eq:defGradEnergy}--\eqref{eq:DoubleWellSplit} and consider $\mu_0\in \calA_\nu$.
We say that a triple $(u,j,E)$ is {\em energy-dissipating} 
with initial data $\mu_0=u_0\scrL^d{\otimes}\nu$
 if the following is satisfied, with $\mu_t\coloneqq u(t,\cdot)\,\scrL^d{\otimes}\nu:$
	\begin{enumerate}
\item  Energy monotonicity: the map $E:[0,\infty)\to\mathbb{R}$ is non-increasing and satisfies
\begin{align*}
	E(t)&\ge \calE(\mu_t)\;\text{ for a.e.\ } t,\qquad\lim_{t\downarrow0}E(t)=\calE(\mu_0).
\end{align*}
		\item Energy-dissipation inequality: 
		$j\in L^1([0,T]{\times} X{\times} Y;\scrL^{d+1}{\otimes}\nu)$ 
        for all $T<\infty$, and for a.e.\ $0<t<\infty$ 
			\begin{align}\label{eq:EDI.diagE}
			E(t)+\int_0^t\int_X\int_Y\frac{|j|^2}{u}\,\dd \nu\dd x\dd\tau\le \calE(\mu_0).
		\end{align}
		\item  
Continuity equation: for $\nu$-a.e. $y\in Y$, $\pa_tu+\divv j=0,$ $u|_{t=0}=u_0$ in $\mathcal{D}'_{t,x}$.
More precisely, for every
$\phi\in C^\infty_c([0,T)\times X)$,
\begin{align}\label{eq:CE.diagE}
	\int_0^T\int_X
	\big(u\pa_t\phi	+j\cdot\nabla_x\phi\big)\,\dd x\dd t
	+\int_Xu_0(\cdot,y)\phi_{|t=0}\,\dd x=0\qquad\text{for }\nu\text{-a.e.\ }y\in Y.
\end{align}
\end{enumerate}
\end{definition}

\paragraph{Dissipative solutions on $X=\mathbb{R}^d$}
On the whole space, we complement the energy-dissipation property
in Definition~\ref{def:e-diss.property} with an entropy-dissipation inequality, which encodes improved spatial
regularity and provides additional control of the densities at infinity.

\begin{definition}[Dissipative solution]\label{def:diss_sol} 
Let $X=\mathbb{R}^d$. 
We call a quadruple $(u,j,E,H)$ a {\em dissipative} solution to \eqref{eq:PDEsystem.intro} with initial data $\mu_0=u_0\scrL^d{\otimes}\nu\in \calA_\nu$
	if the triple $(u,j,E)$ is energy-dissipating in the sense of Definition~\ref{def:e-diss.property} and if the following additional properties are satisfied:
		\begin{enumerate}[resume]
		\item Entropy bound:  $H\in BV_{\loc}([0,\infty))$ with
			$H(t)\ge \mathcal{H}(u(t))$ for a.e.\ $t>0$ and  
			$\lim_{t\downarrow0}H(t)=\mathcal{H}(u_0)$.
 \item Defective flux control:  for every $T>0$ and every 
$\xi\in L^\infty\bigl((0,T)\times Y;C_b^2(X;\mathbb R^d)\bigr),$ it holds that
\begin{align}\label{eq:mv-flux.control}
	\bigg|\int_0^T\langle j-j_u,\xi\rangle\dd t\bigg|
	&\lesssim\norm{\xi}_{L^\infty((0,T)\times Y;C_b^2(X))}\int_0^T	\bigl(E_\dft(t)+H_\dft(t)\bigr)\,\dd t .
\end{align}
where $E_\dft(t)\coloneqq E(t)-\calE(u(t))$ and $H_\dft(t)\coloneqq H(t)-\mathcal{H}(u(t))$.
\item Entropy-dissipation inequality: $u\in L^2([0,T]\times Y;H^2(X))$ for all $T>0$,  and for a.e.\ $t>0$ 
\begin{align}\label{eq:entr.diss}
&H(t)+\int_0^t\mathcal{D}_\mathcal{H}(u)\,\dd s\le \calH(\mu_0) -\int_0^t\int_X \sff_\ccv''(U)|\nabla U|^2\dd x\dd s - \int_0^t\int_{X}\int_Y \msB\nabla_x u\cdot\nabla_x u\,\dd\nu\dd x\dd s,\\\nonumber
&\text{where}\quad
	\mathcal{D}_\mathcal{H}(u)=\int_X\int_Y\big(\alpha(y)|\Delta_x u|^2+{\sf f}_\cvx''(U)|\nabla_x U|^2\big)\,\dd\nu\dd x\ge0,\quad {\sf f}_\ccv''\equiv-\vartheta.
\end{align}
	\end{enumerate}
\end{definition}
Due to the fibred setting and associated lack of strong compactness with respect to $y$, we cannot generally identify the flux $j$ obtained in the construction with the constitutive flux $j_u$. Estimate~\eqref{eq:mv-flux.control} is a substitute that controls the flux defect by the energy and entropy defect.
The latter is needed here due to the unbounded space domain, where the energy alone does not generally control densities at infinity. An important motivation for the above concept is that the fibred JKO scheme with respect to the energy $\mathcal{E}$ generates dissipative solutions, which is the content of the next theorem. Furthermore, as we will see below in  Theorem~\ref{thm:wk-str.uniq}, this concept enjoys a weak--strong uniqueness principle, meaning that whenever a strong solution exists, every dissipative solution with the same initial data coincides with it.

\begin{theorem}[Dissipative solutions generated by fibred JKO scheme]\label{thm:existence} 
Let $X=\mathbb{R}^d$. If $\vartheta>0$ suppose that $d\leq 3$.
For every $\mu_0=u_0\scrL^d{\otimes}\nu\in \calA_\nu$ there exists a dissipative solution $(u,j,E,H)$ to \eqref{eq:PDEsystem.intro} with the following additional properties:
 letting $u_\tau, \wt u_\tau$ denote piecewise constant and De Giorgi variational interpolants obtained along the minimising movement scheme for $\calE$ in the fibred Wasserstein space with initial value $u_0$, there exists a subsequence $\tau\downarrow0$ such that
 for all $T<\infty$
 \begin{align}\label{eq:conv.interpolants.weak}
 	&u_\tau, \wt u_\tau\rightharpoonup u\quad\text{in }L^2([0,T]\times Y;H^1(X)).
 \end{align}
 Furthermore, we have for a.e. $t>0$
 \begin{align*}
 E(t)=\lim_{\tau\downarrow0}\calE(u_\tau(t))=\lim_{\tau\downarrow0}\calE(\wt u_\tau(t)), \qquad H(t)=\lim_{\tau\downarrow0}\mathcal{H}(u_\tau(t))=\lim_{\tau\downarrow0}\mathcal{H}(\wt u_\tau(t)),
 \end{align*}
and the weak convergence
\begin{align}\label{eq:jtau.weakL1.thm}
	j_\tau,\wt\jmath_\tau\rightharpoonup j 	\quad\text{in }L^1([0,T]\times X\times Y).
\end{align}
If $\nu$ has finite support, one obtains in addition $E(t)=\calE(\mu(t))$ and $H(t)=\calH(\mu(t))$ for a.e.\ $t>0$, $j=j_u$,
 and the convergence~\eqref{eq:conv.interpolants.weak} holds in the strong sense.   If $\supp\nu$ is a singleton, the same result holds in every space dimension $d\in \mathbb{N}_+$ also for $\vartheta>0$, where the potential term of the flux is determined distributionally by the right-hand side of~\eqref{eq:flux.pot.def.scalar}.
\end{theorem}
 A key point in the proof of Theorem~\ref{thm:existence}, given in Section~\ref{sec:diagE}, is the use of the coercivity properties of the energy and the convexity (in the affine sense) of the gradient part $\mathcal{E}_1$. Exploiting these properties in the variational scheme, we obtain a common limiting energy for the different interpolants $u_\tau$ and $\wt u_\tau$---without requiring strong compactness.
In the case of a single species, Theorem~\ref{thm:existence} contains the main assertions of~\cite[Section~3]{KL_2022}.

We next formulate the weak--strong uniqueness principle mentioned above.
\begin{theorem}[Weak--strong uniqueness of dissipative solutions]\label{thm:wk-str.uniq}
	Let $d\ge1$, $X=\mathbb R^d$, and suppose that $v\in L^\infty((0,T)\times Y;H^3(X))$ 
	is a positive weak solution to~\eqref{eq:PDEsystem.intro} with initial data $v(0)=u_0$, satisfying  
	\begin{align}\label{eq:strongsol.reg.diss}
		\nabla\log v,\ \nabla w_v	\in L^\infty\big((0,T)\times Y;	C_b^2(X;\mathbb R^d)\big),
	\end{align}
	where $w_v:=-\alpha\Delta v+\kappa\sff'(V)+\mathsf Bv$ and  $V:=\int_Y\kappa v\,\dd\nu$.
	
	Then any dissipative solution $(u,j,E,H)$ with initial data $\mu_0=u_0\scrL^d\otimes\nu\in\calA_\nu$ coincides with $v$, i.e.
	\[u=v,\qquad j=j_v\quad\text{a.e. on }(0,T)\times X\times Y,\]
	and \[E(t)=\calE(v(t)),\qquad H(t)=\mathcal H(v(t))	\quad\text{for a.e. }t\in(0,T).	\]
\end{theorem}
The proof combines energy- and entropy-dissipation inequalities by means of a stability estimate similar to~\cite{HJ_2025,HJT_2024}. It is postponed to Section~\ref{ssec:wsu.diag.diss}.

For the specific models considered here, due to the diagonal structure of the gradient energy, we will see in Theorem~\ref{thm:wk-str.uniq.e-diss} below that, on  smooth bounded domains, weak--strong uniqueness can even be achieved for a notion of energy-dissipating solution. 
Let us mention that, in principle, reliance of a weak--strong stability estimate on an entropy-dissipation inequality can also be avoided on the whole space, but comes at the price of a less natural flux-defect control, involving the defect of some negative Sobolev norm.

\paragraph{Energy-dissipating solutions on non-convex domains}
\label{ss:nonconvex}
On bounded domains, a slope coercivity property adapted from~\cite{GHL_preprint}, cf.~Proposition~\ref{prop:coercivity.slope}, allows us to avoid reliance on the entropy dissipation inequality for improved spatial regularity. As a consequence, we obtain convergence of the fibred JKO scheme even in non-convex domains. We are not aware of such a result for degenerate fourth-order systems in the literature.

Let $X=\ol\Om$, where $\Om\subset\mathbb{R}^d$ is a bounded smooth domain, possibly non-convex, with intrinsic length distance $d_X$. The resulting metric space
$(X,d_X)$ is geodesic, and the associated quadratic Wasserstein
distance is defined by
\begin{align}\label{eq:defW2.intrinsic}
\mathcal{W}_{2,X}^2(\mu_0,\mu_1)=\inf_{\pi\in\Gamma(\mu_0,\mu_1)}\int_{X\times X}d_X^2(x,x')\,\dd\pi(x,x').
\end{align}

\begin{definition}[Energy-dissipating solution]\label{def:e-diss.sol}
An energy-dissipating triple $(u,j,E)$ in the sense of Definition~\ref{def:e-diss.property} is called an	\emph{energy-dissipating solution} to~\eqref{eq:PDEsystem.intro} if
\begin{align}\label{eq:def.e-diss.flux}
	\big|\langle j(t,\cdot)-j_{u(t)},\xi\rangle\big|&\lesssim\norm{\xi}_{L^\infty(Y;C^2(\ol\Om))}E_\dft(t)
\end{align}
for a.e.\ $t>0$ and all $\xi\in L^\infty\big(Y;C_{\mathrm{tan}}^2(\ol\Omega;\mathbb R^d)\big)$.
	Here, $\langle j_{u(t)},\xi\rangle$ is defined according to~\eqref{eq:flux.distrib}, while $\langle j,\xi\rangle=\int_Y\int_Xj\cdot\xi\,\dd x\dd\nu(y)$. 
\end{definition}

Our main constructive existence result for bounded domains reads as follows.
\begin{theorem}[Energy-dissipating solutions generated by fibred OT with intrinsic cost]\label{thm:ex.Ediss.nonconvex}
For every $\mu_0\in \mathcal{A}_\nu$, the fibred JKO scheme with intrinsic cost generates 
an energy-dissipating solution $(u,j,E)$ in the sense of Definition~\ref{def:e-diss.sol}, which additionally enjoys for all $T<\infty$ the regularity
\begin{subequations}\label{eq:ediss-sol.reg}
\begin{align}\label{eq:wsu.e.regularity}
	u\in  L^2((0,T){\times} Y;H^2(\Omega)),\qquad \vartheta U^2\in L^2(0,T;H^1(\Omega)),
\end{align}
and satisfies the homogeneous Neumann boundary condition
\begin{align}\label{eq:homNeumann.bc}
	\partial_nu=0\quad\text{for a.e. }t>0,x\in\pa\Om\text{ and $\nu$-a.e.\ }y\in Y.
\end{align}
\end{subequations}
If $\supp\nu\subset Y$ is finite, the scheme subsequentially converges to an energy-dissipating weak solution consisting of an energy-dissipating triple $(u,j_u,E)$ with vanishing energy defect, i.e., $E(t)=\mathcal{E}(\mu_t)$ for a.e.\ $t>0$ and the identified flux $j_u$.
\end{theorem}

The following weak--strong uniqueness result ensures that every energy-dissipating solution with the 
additional regularity~\eqref{eq:ediss-sol.reg} coincides with any sufficiently regular, strictly positive solution with same initial data. Note that this additional regularity is obtained in the construction via the fibred JKO scheme.

\begin{theorem}[Weak--strong uniqueness of energy-dissipating solutions]\label{thm:wk-str.uniq.e-diss}
Suppose there exists a jointly measurable, strictly positive solution $v\in L^2((0,T)\times Y;H^2(\Om))$ of \eqref{eq:PDEsystem.intro}, with $v(0,\cdot)=u_0$ %
satisfying %
\begin{align*}
	&v(\cdot,\cdot,y)\in C^1([0,T];C^2(\overline\Omega))\cap C([0,T];C^5(\overline\Omega))\qquad\text{for $\nu$-a.e.\ $y$},\\
	&v,\partial_tv,w_v,\partial_tw_v\in L^\infty((0,T)\times\Omega\times Y),\\
	&\nabla w_v\in L^\infty((0,T)\times Y;W^{2,\infty}(\Omega;\mathbb R^d)),\qquad \partial_nv=\partial_nw_v=0\quad\text{for } x\in\partial\Omega.
\end{align*}
 If ${\sf f}\not\equiv0$, further suppose that $d\le5$.
Then, any energy-dissipating solution $(u,j,E)$ with initial data $\mu_0=u_0\scrL^d\otimes\nu\in \calA_\nu$ satisfying~\eqref{eq:ediss-sol.reg} 
must coincide with $v$ on $[0,T]$, i.e.,
 $u=v$ holds $\scrL^{1+d}\otimes\nu$-a.e.\ in $[0,T]\times \Om\times Y$, $j=j_v$ and $E(t)=\calE(v(t))$ for a.e.\ $t\in[0,T]$.
\end{theorem}

The following corollary is a consequence of the constructive existence theorems combined with the specific form of the weak--strong uniqueness results, which ensure vanishing energy and entropy defects on the existence interval of any strong solution with the same initial data.

\begin{corollary}[Strong convergence] Suppose there is a strong solution $u$ to \eqref{eq:PDEsystem.intro}  on the time interval $[0,T]$ as in Theorem~\ref{thm:wk-str.uniq} or Theorem~\ref{thm:wk-str.uniq.e-diss}. Then, under the assumptions of the respective existence result (cf.\ Theorem~\ref{thm:existence} or~\ref{thm:ex.Ediss.nonconvex})  the minimising movement scheme emanating from $u_0$ converges strongly to the limit $u$ in the sense that for both piecewise constant and De Giorgi variational interpolants $u_\tau, \wt u_\tau$ one has
	$u_\tau, \wt u_\tau\to u$ in $L^2([0,T]\times Y;H^1(X))$.
\end{corollary}

\section{Dissipative solutions by minimising movements in 
\texorpdfstring{$\boldsymbol{\mathcal{W}_{2,\nu,\mathbb{R}^d}}$}{fibred Wasserstein metric}}
\label{sec:diagE}
Here, we present in detail the construction of dissipative solutions (Theorem~\ref{thm:existence}) in the case $X=\R^d$ via the minimising movement scheme (MMS) in
the fibred setting for the case of the functional $\calE =\calE_1+\calE_0$ on $\calP_{2,\nu}(X{\times} Y)$ as defined in Section~\ref{sec:Prelim} and the fibred Wasserstein $\calW_{2,\nu}$.

\subsection{Properties of the fibred Moreau--Yosida approximation}
We start with the usual definitions for the time-incremental JKO scheme.
\begin{definition}\label{def:MoreauYosida}
    Given $\mu \in\calP_{2,\nu}(X{\times} Y)$ and $h>0$, 
    we define the Moreau--Yosida approximation
    \begin{align}
        \calE_h(\mu) &\coloneqq \inf_{\wh\mu\in\calP_{2,\nu}(X{\times}Y)} \bra[\bigg]{\frac{\calW_{2,\nu}(\wh\mu,\mu)^2}{2h} + \calE(\wh\mu)},\label{eq:MoreauYosida}
    \intertext{and the set of its minimisers}
        \sfJ_h(\mu) &\coloneqq \argmin_{\wh\mu \in \calP_{2,\nu}(X{\times}Y)}\bra[\bigg]{\frac{\calW_{2,\nu}(\wh\mu,\mu)^2}{2h} + \calE(\wh\mu)}\subset \calP_{2,\nu}(X{\times} Y).
    \end{align}
\end{definition}

The following result gives direct a priori estimates for minimisers in $\sfJ_h(\mu)$.
\begin{lemma}\label{lem:bounds}
    Let $\sfF \subset \calP_{2,\nu}(X{\times} Y)$ be a bounded subset of probability measures 
     with uniformly bounded energy $\calE$, i.e., there exists $\mu \in\calP_{2,\nu}(X{\times}Y)\cap\mathrm{dom}(\calE)$ such that
    \begin{align*}
        \sup_{\wh\mu \in\sfF} \calW_{2,\nu}(\wh\mu,\mu)<\infty
        \qquad\text{and}\qquad
        \sup_{\wh\mu \in\sfF}\calE(\wh\mu) <\infty.
    \end{align*} 
    Then, for all $\wh\mu\in\sfF$ we have $\wh\mu \ll \scrL^d{\otimes}\nu$ and denoting by $\wh u$ the density, it holds
    \begin{equation}
    \begin{gathered}
    \label{eq:incrementalApriori}
        \sup_{\wh\mu\in\sfF} m_{2,X}(\wh\mu)<\infty,\quad \text{where}\quad m_{2,X}(\wh\mu)\coloneqq \int_X\int_Y \abs{x}^2 \wh u(x,y)\,\dd \nu(y)\dd x 
        \quad\text{and}\\
        \sup_{\wh\mu \in\sfF}\Bigg\{\int_X\int_Y \abs{\nablax\wh u}^2\,\dd\nu\dd x +\vartheta\int_X |\wh U|^4\dd x\Bigg\}<\infty,\quad \text{where}\quad\wh U(x) = \int_Y \kappa(y)\wh u(x,y)\dd \nu.
        \end{gathered}
    \end{equation}
    Here, $\vartheta\geq0$ is the coefficient in the definition of the double-well potential $\msf$ (cf.~\eqref{eq:DoubleWellSplit}).
    In particular, for any $\phi_1\in \rmL^\infty(Y;C^1(X))$, $\phi_2\in \rmL^\infty(Y;C(X))$ we have
    \begin{align*}
        \sup_{\wh\mu\in\sfF}\int_X \abs{\nablax\skp{\wh\mu,\phi_1}_Y}^2\,\dd x <\infty
        \qquad\text{and}\qquad 
        \sup_{\wh\mu\in\sfF}\int_X \abs{\skp{\wh\mu,\phi_2}_Y}^p\,\dd x  <\infty,
    \end{align*}
    for all $p\in[1,2^\ast)$ and for $p = 2^\ast$ if $d\ne 2$.
\end{lemma}

\begin{proof}
    The absolute continuity of $\wh \mu$ with respect to $\scrL^d{\otimes}\nu$ follows immediately from the boundedness of the energy.
    The moment control in \eqref{eq:incrementalApriori} is a direct consequence of the triangle inequality for $\calW_{2,\nu}$. 
    The derivative control in \eqref{eq:incrementalApriori} 
    follows from the uniform bound on $\calE_1$ 
    and the assumption \eqref{eq:assAlpha}. 
    Using Jensen's inequality, this also implies the regularity of $\skp{\wh\mu,\phi}_Y$ when $\phi$ is differentiable in $x$. 
    For such $\phi$ we then obtain higher integrability via the Gagliardo--Nirenberg inequality. 
    The higher integrability for $\phi\in \rmL^\infty(Y;C(X))$ is then obtained via an approximation by $x$-differentiable $\phi_n$.
    Indeed, this follows since $L^2(Y;C^1(X))\subset L^2(Y;C(X))$ densely and for $\psi_n=\phi_n-\phi\to 0$ in $L^2(Y;C(X))$ it holds
    \begin{align*}
        \norm{\skp{u,\psi_n}_Y}_{\rmL^p(X)} 
        &\lesssim \norm{u}_{\rmL^p(X;\rmL^2(Y))}\norm{\psi_n}_{\rmL^\infty(X;\rmL^2(Y))}
        \le \norm{u}_{\rmL^2(Y;\rmL^p(X))}\norm{\psi_n}_{\rmL^2(Y;\rmL^\infty(X))} 
        \overset{n\to 
        \infty}{\longrightarrow}0.
    \end{align*}
    This concludes the proof.
\end{proof}

Dealing with the nonlocal (and in general non-convex) repulsion contribution $\calR$ (cf.~\eqref{eq:defRelEnergy}) 
in the fibred setting is not straightforward. However, 
the following lemma shows that we can approximate the interaction kernel $\beta$ by tensor products
in a suitable way. Since $\nu$ is a finite measure, we have $\beta\in L^2(Y{\times} Y)$ (consequently, the integral operator $\sfB$ in \eqref{eq:defOpB} is Hilbert--Schmidt and hence compact).
Since finite rank kernels of the form $\sum_{l=1}^{N}a_l(y)b_l(y')$ for $a_l,b_l\in L^\infty(Y)$
are dense in $L^2(Y\times Y)$, we find
\begin{equation}
\label{eq:finiteRankApprox}
\beta_n(y,y')
=
\frac{1}{2}\sum_{l=1}^{N_n}\big(a_{l}^n(y)b_{l}^n(y')+ a_{l}^n(y')b_{l}^n(y)\big),
\qquad a_{n,l},b_{n,l}\in L^\infty(Y),
\end{equation}
such that $\beta_n\to \beta$ in $L^2(Y\times Y)$.
Moreover, we can choose the approximation such that non-negativity is preserved and
\begin{equation} \label{eq:boundBetaApprox}
	\|\beta_n\|_{L^\infty(Y\times Y)}\leq \beta^*.
\end{equation}
For example, we may choose an increasing sequence of finite measurable partitions 
$\{A_{n,1},\ldots,A_{n,N_n}\} $ of $Y$ and set
\[
c_{n,ij}
=
\dashint_{A_{n,i}\times A_{n,j}}
\beta(y,y')\,\dd\nu(y)\,\dd\nu(y'),\quad \text{and}\quad
\beta_n(y,y')
=
\sum_{i,j=1}^{N_n}
c_{n,ij}\,
\mathbf1_{A_{n,i}}(y)\mathbf1_{A_{n,j}}(y').
\]

We define the approximating operator and functional
\begin{align}
	\label{eq:defRepulsionApprox}
    \begin{split}
    &(\sfB^nu) (x,y)\coloneqq \int_Y\beta_n(y,y')u(x,y')\dd\nu(y'),\quad
	\calR^n(\mu)\coloneqq\frac12\int_X (\sfB^n u,u)_{L^2(Y)}\,\dd x,\\&
    \text{for }\mu = u\scrL^d\otimes \nu \text{ with }u\in L^2(X\times Y).
    \end{split}
\end{align}
\begin{lemma}
\label{l:ApproxRepul}
	Let the repulsion energy $\calR$ and its approximation $\calR^n$ be defined as in \eqref{eq:defEnergyRepulsion} and \eqref{eq:defRepulsionApprox}, respectively. Then, 
	there exists a constant $C>0$ independent of $n$ such that for all $\mu\in \calP_{2,\nu}(X{\times} Y)$ with density $u\in L^2(X\times Y)$ it holds
	\begin{align}
		\label{eq:estimateRepulsionDerivApprox}
		\|\mathsf{B}^n- \mathsf{B}\|_{\calL(L^2(Y))}&\leq \|\beta_n{-}\beta\|_{L^2(Y\times Y)},\\
		\label{eq:estimateRepulsionApprox}
		|\calR^n(\mu)-\calR(\mu)|
		&\leq
		\frac{1}{2}\|\beta_n{-}\beta\|_{L^2(Y\times Y)}
		\|u\|_{L^2(X\times Y)}^2.
    \end{align}
\end{lemma}

\begin{proof}Let us set $K_n(y,y'):=\beta(y,y')-\beta_n(y,y')$ such that
$K_n \to 0$ in $L^2(Y\times Y)$. By the Cauchy--Schwarz inequality, we get for every $v\in L^2(Y)$
and a.e.\ $y\in Y$
\[
|((\sfB{-}\sfB_n)v)(y)|
\leq 
\|K_n(y,\cdot)\|_{L^2(Y)}\|v\|_{L^2(Y)}.
\]
Hence, after integrating the square of this estimate with respect to $\nu$,
we obtain $\|(\sfB^n-\sfB)v\|_{L^2(Y)}
\leq
\|K_n\|_{L^2(Y{\times} Y)}
\|v\|_{L^2(\nu)}$, from which \eqref{eq:estimateRepulsionDerivApprox} directly follows.
Now, \eqref{eq:estimateRepulsionApprox} is a simple consequence of \eqref{eq:estimateRepulsionDerivApprox}.
\end{proof}

The preceding bounds and kernel approximation allow us to apply
the direct method and establish the existence of minimisers for
the incremental problem.

\begin{proposition}\label{prop:MY_well-posedness}
    Let $h>0$ and $\mu \in\calP_{2,\nu}(X{\times}Y)\cap\mathrm{dom}(\calE)$. Then, $\sfJ_h(\mu)\ne \emptyset$, i.e., the incremental problem has minimisers.
\end{proposition}
\begin{proof}
    We apply the direct method of the calculus of variations. 

    Since $\calE \ge 0$, we find a sequence $(\mu_k)_k\subset\calP_{2,\nu}(X{\times}Y)$, satisfying $\frac{1}{2h}\calW_{2,\nu}(\wh\mu_k,\mu)^2 + \calE(\wh\mu_k)\le C$ and approximating the infimum arbitrarily closely. 
    Since all $\wh\mu_k$ have mass 1, there exists $\wh\mu\in\calP_{2,\nu}(X{\times}Y)$ such that (along a subsequence) $\wh\mu^k\rightharpoonup^\ast\wh\mu$ weakly-$^\ast$ in $\calM_+(X{\times}Y)$. 
    The energy contribution $\calE_1$ and $\calW^2_{2,\nu}(\mu,\cdot)$ are lower semicontinuous along this sequence.

    Moreover, we can assume that the density $\wh u_k$ of $\wh \mu_k$ with respect to $\scrL^d\otimes \nu$
    also converges weakly in $L^2(Y;H^1(X))$. In particular, since $\wh \mu_k$  has uniformly bounded second moments, we can also assume that $\langle \wh u_k,\phi\rangle_Y $ converges strongly in $L^2(X)$ for any $\phi\in L^\infty(Y)$, in particular, $\wh U_k = \langle \wh u_k,\kappa \rangle_Y\to \wh U=\langle \wh u,\kappa \rangle_Y$ strongly in $L^2(X)$. 
    Since $\msf$ is convex up to a quadratic function, we obtain the lower semicontinuity of $\calF$.
    To prove that $\calR$ is lower semicontinuous, we approximate $\beta$ by tensor products $\beta_n$ as in \eqref{eq:finiteRankApprox}. For fixed $n\in\bbN$, the functional $\calR^n$ is continuous along the sequence $\wh \mu_k$. Indeed, we have that $\wh U^{l,n}_k:=\langle \wh u_k, a_l^n\rangle_Y\to \langle \wh u, a^n_l\rangle_Y=:\wh U^{l,n}$ strongly 
    in $L^2(X)$ (resp.\ $\bar U^{l,n}_k:=\langle \wh u_k, b^n_l\rangle_Y\to \langle \wh u, b^n_l\rangle_Y=:\bar U^{l,n}$ strongly 
    in $L^2(X)$) and hence
    \[
    \lim_{k\to \infty}\calR^n(\wh \mu_k) = \lim_{k\to\infty}\frac12 \sum_{l=1}^{N_n}\int_{X} \wh U^{l,n}_k(x)\bar U^{l,n}_k(x)\,\dd x = 
    \frac12\sum_{l=1}^{N_n}\int_{X} \wh U^{l,n}(x)\bar U^{l,n}(x)\,\dd x = \calR^n(\wh \mu).
    \]
    Finally, we choose $R>0$ such that $\|\wh u_k\|_{L^2(X\times Y)},\|\wh u\|_{L^2(X\times Y)}\leq R$, which exists due to the uniform energy bound and the Gagliardo--Nirenberg inequality. For $\eps>0$, we fix $n\in\bbN$  and $k_0\in\bbN$ such that $\|\beta_n{-}\beta\|_{L^2(Y\times Y)}R^2\leq 2\eps/3$ (cf. Lemma~\ref{l:ApproxRepul}) and $|\calR^n(\wh\mu_k)-\calR^n(\wh \mu)|\leq \eps/3$ for all $k\geq k_0$. 
    Then, it holds that
    \[
    |\calR(\wh \mu_k)-\calR(\wh \mu)| \leq |\calR(\wh\mu_k)-\calR^n(\wh\mu_k)| + |\calR^n(\wh\mu_k)-\calR^n(\wh\mu)| + |\calR^n(\wh\mu)-\calR(\wh\mu)|\leq \eps.
    \]
    This shows the continuity of the repulsion contribution and in turn the lower semicontinuity of the lower-order energy $\calE_0$.
\end{proof}

 We follow \cite[Lemma 3.7]{KL_2022} to establish the Euler--Lagrange equation by solely relying on $H^1$-regularity with respect to  the $x$-variable.
\begin{lemma}[Discrete flux]\label{lem:discrete_flux}
    Let $h>0$ and $\mu \in\calP_{2,\nu}(X{\times}Y)\cap\mathrm{dom}(\calE)$ and $\wh\mu \in \sfJ_h(\mu)$ with density $\wh u$. 
There exists a jointly measurable map $T:X\times Y\to X$ such that, for $\nu$-a.e.\ $y\in Y$,  $T(\cdot,y)$
   is the 2-Wasserstein optimal transport map from $\wh\mu^y$ to $\mu^y$.  Moreover, the flux defined by
	\begin{align*}
        \hat\jmath(x,y) \coloneqq \frac{1}{h}(x-T(x,y))\wh u(x,y)
    \end{align*}
    satisfies
    \begin{align}\label{eq:flux-bound_L1}
        \norm{\hat\jmath}_{{L}^1(X{\times}Y)} 
        \le \bra[\bigg]{\int_X\int_Y \frac{\abs{\hat\jmath(x,y)}^2}{\wh u(x,y)}\,\dd\nu(y)\dd x}^{1/2}
        = \frac{1}{h} \calW_{2,\nu}(\mu,\wh\mu).
    \end{align}
\end{lemma}
\begin{proof}
Since $\wh\mu^y\ll\scrL^d$ for $\nu$-a.e.\ $y\in Y$, Brenier's theorem \cite[Theorem~2.12(ii)]{villani2003topics} guarantees the existence of a $\wh\mu^y$-a.e.\ unique optimal transport map $T^y:X\to X$ from $\wh\mu^y$ to $\mu^y$. 
 Since $y\mapsto(\wh\mu^y,\mu^y)$ is Borel, we may now simply invoke~\cite[Theorem~1.1]{fontbona2010measurability}, see also~\cite[Theorem~3.0.3]{Toneian_2019}, to infer the existence of a jointly measurable map $T:X\times Y\to X$ satisfying $T(\cdot,y)=T^y$ $\wh\mu^y$-a.e.\ for $\nu$-a.e.\ $y$.
  Finally, the inequality in \eqref{eq:flux-bound_L1} is a direct application of the Cauchy-Schwarz inequality. The identity in \eqref{eq:flux-bound_L1} follows from the optimality of $T$.
\end{proof}

The next lemma establishes monotonicity of the energies and transport distances with
respect to the time-step parameter, together with a quantitative estimate comparing
the minimisers. These estimates will yield strong convergence of the difference between
the two interpolants.
\begin{lemma}
\label{lem:energy-hierarchy}
    Let $h>0$, $\mu^0 \in\calP_{2,\nu}(X{\times}Y)\cap\mathrm{dom}(\calE)$. 
    For $\theta\in(0,1]$ let $\mu^\theta \in \sfJ_{\theta h}(\mu^0)$. 
    Then, the following inequalities hold for all $\theta\in(0,1)$ 
    \begin{gather}        
        \label{eq:energy.squeeze}
        \calE(\mu^1) \le \calE(\mu^\theta)\le \calE(\mu^0),\\
        \label{eq:distance.squeeze}
        \calW_{2,\nu}(\mu^\theta,\mu^0)\le \calW_{2,\nu}(\mu^1,\mu^0).
    \end{gather}
  Furthermore,
    there exists a fixed constant $c>0$ such that for all $\theta\in(0,1)$
 \begin{subequations}\label{eq:strongH1.Econtrol}
 	    \begin{align}
 		c\Big(\|\nabla u^\theta-\nabla u^1\|_{L^2(X{\times}Y)}^2 + \vartheta\|U^\theta{-}U^1\|^4_{L^4(X)}\Big)&\le \mathcal{E}(\mu^\theta)-\mathcal{E}(\mu^1) + \calB_0(\mu^\theta-\mu^1), 
 		\label{eq:strongH1.Econtrol.th1}
 		\\	c\big(\|\nabla u^0-\nabla u^\theta\|_{L^2(X{\times}Y)}^2+ \vartheta\|U^0{-}U^\theta\|^4_{L^4(X)}\Big)&\le \mathcal{E}(\mu^0)-\mathcal{E}(\mu^\theta)+ \calB_0(\mu^\theta-\mu^0),
 		\label{eq:strongH1.Econtrol.0th}
 	\end{align}
 \end{subequations}
 where $u^\theta, u^i$, $i=0,1$, denote the densities of the measures $\mu^\theta,\mu^i$ 
 with respect to $\scrL^d\otimes\nu$, $\nabla\coloneqq \nabla_x$ 
 is the weak derivative in the space $X$, and 
 \begin{equation}
 \label{eq:errorB0Energy}
 \calB_0(\omega) \coloneqq \tilde c\vartheta
 \|\Sigma\|_{L^2(X)}^2 -\frac{1}{2}\calR(\omega),\quad\omega = \sigma\scrL^d\otimes\nu\in\calM(X\times Y)
  \end{equation}
  for some constant $\tilde c> 0$ and $\Sigma(x) = \int_Y \kappa(y) \sigma(x,y)\dd\nu(y)$, respectively.
\end{lemma}
\begin{proof}
	The second inequality in \eqref{eq:energy.squeeze} is immediate by choosing $\mu^0$ as a competitor in the optimality condition for $\mu^\theta$.
	For the first inequality, we choose $\mu^\theta$ as a competitor in the optimality condition for $\mu^1$ and vice versa. This gives the two inequalities
	\begin{align}
    \label{eq:u^1-optimality}
		\calE(\mu^1)+\frac{1}{2h}\calW_{2,\nu}^2(\mu^1,\mu^0)&\le \calE(\mu^\theta)+\frac{1}{2h}\calW_{2,\nu}^2(\mu^\theta,\mu^0),\\
        \calE(\mu^\theta)+\frac{1}{2h\theta}\calW_{2,\nu}^2(\mu^\theta,\mu^0)&\le \calE(\mu^1)+\frac{1}{2h\theta}\calW_{2,\nu}^2(\mu^1,\mu^0).
        \notag
	\end{align}
	Upon adding these two inequalities, we infer
	\begin{align*}
		0\le \Big(\frac{1}{\theta}-1\Big)\big(\calW_{2,\nu}(\mu^1,\mu^0)^2-\calW_{2,\nu}(\mu^\theta,\mu^0)^2\big),
	\end{align*}
    i.e., \eqref{eq:distance.squeeze} holds since $\theta\in (0,1)$.
    Inserting this inequality back into \eqref{eq:u^1-optimality} and reordering the terms 
    then yields the first inequality in \eqref{eq:energy.squeeze}.
    
  It remains to show~\eqref{eq:strongH1.Econtrol}. We only prove~\eqref{eq:strongH1.Econtrol.th1}, because the argument for~\eqref{eq:strongH1.Econtrol.0th} is similar. 
    
    We use $\mu^{\theta,1}\coloneqq \frac12(\mu^\theta+\mu^1)\in\calP_{2,\nu}(X{\times}Y)$ 
    as a competitor in the optimality condition for $\mu^1$ to infer
    \begin{align*}
    \calE(\mu^1)+\frac{1}{2h}\calW_{2,\nu}^2(\mu^1,\mu^0)\le \calE(\mu^{\theta,1})+\frac{1}{2h}\calW_{2,\nu}^2(\mu^{\theta,1},\mu^0).
    \end{align*}
    To proceed, observe that by virtue of the affine convexity of $\mu\mapsto\calW_{2,\nu}^2(\mu,\wh\mu)$ for fixed $\wh\mu$
    \[
    \calW_{2,\nu}^2(\mu^{\theta,1},\mu^0)\le \frac12\calW_{2,\nu}(\mu_{\theta},\mu^0)^2+\frac12\calW_{2,\nu}^2(\mu_{1},\mu^0)
    \le\calW_{2,\nu}^2(\mu_{1},\mu^0),
    \]
    where the second step follows from~\eqref{eq:distance.squeeze}. From the last two lines we infer
    $\calE(\mu^1)\le\calE(\mu^{\theta,1})$. At the same time, the uniform convexity of $\calE_1$ in the homogeneous Sobolev space $L^2(Y;\dot H^1(X))$ implies that, for some fixed $c>0$,
    \begin{align*}
    \calE_1(\mu^{\theta,1})\le\frac12\calE_1(\mu^\theta)+\frac12\calE_1(\mu^1)
    -\frac{c}{2}\|\nabla u_\theta-\nabla u_1\|_{L^2(X{\times}Y)}^2.
    \end{align*}
    Furthermore, decomposing in the case $\vartheta>0$ 
    the double-well potential $\msf = \msf_\text{cvx}+\msf_\text{conc}$ as in \eqref{eq:DoubleWellSplit}, we obtain the inequality
    \[
    \msf\big((1{-}\theta)s_0+\theta s_1\big)\leq (1{-}\theta)\msf(s_0) + \theta\msf(s_1) -{\theta(1{-}\theta)}\Big(\frac{\vartheta}{16}|s_0{-}s_1|^4-\frac{\vartheta}{8}|s_0{-}s_1|^2\Big),\quad\theta\in(0,1),~ s_0,s_1\geq0.    
    \]
    Therefore, we arrive at the estimate
    \begin{align*}
        \calF(\mu^{\theta,1}) + \frac{\vartheta}{64}\|U_1{-}U_\theta\|^4_{L^4(X)} \le\frac12\calF(\mu^\theta)+\frac12\calF(\mu^1)
        +\frac{\vartheta}{32}\|U_\theta -U_1\|_{L^2(X)}^2
    \end{align*}
    with the usual definitions of $U_1$ and $U_\theta$ (cf.~\eqref{eq:defLowerEnergy}).  
    Finally, since $\calR$ is quadratic, we get
    \[
    \calR(\mu^{\theta,1}) = \frac12\calR(\mu^\theta) +\frac12\calR(\mu^1) - \frac{1}{4}\calR(\mu^\theta-\mu^1).    
    \]
    In combination, we deduce~\eqref{eq:strongH1.Econtrol.th1}.
\end{proof}

We next relate the variation of the Moreau--Yosida approximation with
respect to the time-step parameter to the corresponding transport
cost, obtaining an identity that will yield the discrete
energy-dissipation inequality.

\begin{lemma}\label{lem:Moreau-Yosida-Properties}
    Given $\mu \in \calP_{2,\nu}(X{\times}Y)\cap \mathrm{dom}(\calE)$, $h>0$ consider the scalar quantities
    \begin{align*}
        \calW_{2,\nu,h}^-(\mu) \coloneqq \inf_{\wh\mu \in \sfJ_h(\mu)}\calW_{2,\nu}(\mu,\wh\mu) 
        \qquad\text{and}\qquad 
        \calW_{2,\nu,h}^+(\mu) \coloneqq \sup_{\wh\mu \in \sfJ_h(\mu)}\calW_{2,\nu}(\mu,\wh\mu).
    \end{align*}
    There exists an at most countable set $\calN(\mu)\subset (0,\infty)$ such that $\calW_{2,\nu,h}^-(\mu) = \calW_{2,\nu,h}^+(\mu)$ for all $h \in (0,\infty)\setminus\calN(\mu)$.

    Furthermore, for all $\mu\in\calP_{2,\nu}(X{\times}Y)$ the map $h\mapsto\calE_h(\mu)$ is locally Lipschitz on $(0,\infty)$ and
    \begin{align*}
        \frac{\dd}{\dd h}\calE_h(\mu) = -\frac{\bra[\big]{\calW^\pm_{2,\nu,h}(\mu)}^2}{2h^2}
    \end{align*}
    for all $h \in (0,\infty)\setminus\calN(\mu)$

    In particular, for all $\mu\in\calP_{2,\nu}(X{\times}Y)\cap \mathrm{dom}(\calE)$, $h\in(0,\infty)$, and $\wh\mu\in\sfJ_h(\mu)$ we have
    \begin{align*}
        \frac{\calW^2_{2,\nu}(\mu,\wh\mu)}{2h} + \int_0^h \frac{\bra[\big]{\calW^\pm_{2,\nu,t}(\mu)}^2}{2t^2} \,\dd t = \calE(\mu) -\calE(\wh\mu).
    \end{align*}
\end{lemma}
\begin{proof}

    Recalling from \cite[Proposition~3.4]{peszek2023heterogeneous} that $\calP_{2,\nu}(X{\times}Y)$ is a Polish space the first claim follows from \cite[Lemma~3.1.2]{AGS}.
    Then, the second claim is a consequence of \cite[Theorem~3.1.4]{AGS}.
\end{proof}

We now derive the Euler--Lagrange equation for one  step of the
incremental
minimisation problem by taking spatial inner variations separately
in each fibre. This identifies the discrete flux 
with the constitutive flux in distributional form, using only
$L^2(Y;H^1(X))$ regularity of the minimising density.
Following Remark~\ref{rem:Singleton}, we highlight that the 
in the case where $\supp\nu$ is a singleton, the equation for the lower order contribution to the constitutive flux can be written free of any derivatives of $u$.
In particular, in this setting the restriction to $d\leq 3$ can be removed as
all terms are well-defined due to the $L^4$ bound on $U_\tau$ in \eqref{eq:incrementalApriori}.

\begin{proposition}[Euler--Lagrange equation]\label{prop:E-L_1}
    Let $h>0$ and $\mu \in\calP_{2,\nu}(X{\times}Y)\cap\mathrm{dom}(\calE)$ and $\wh\mu \in \sfJ_h(\mu)$ with density $\wh u$. Let $\hat\jmath$ be the flux introduced in Lemma~\ref{lem:discrete_flux}.
    Then, for all $\xi\in L^\infty(Y; C^2_c(X;\R^d))$  the Euler--Lagrange equation holds, namely
    \begin{equation}
    \label{eq:E-L_full}
    \begin{aligned}
        \int_X\int_Y \hat\jmath\cdot\xi  \,\dd\nu \dd x
        &= \int_X\int_Y \alpha(y) \big[\bra[\big]{\tfrac12\divv_x \xi I + \nablax\xi}\nablax \wh u + \wh u \nablax\divv_x\xi\big]
        \cdot \nablax \wh u \,\dd\nu(y)\dd x\\
        &\quad
        + \int_X\int_Y \big[(\msf'(\wh U)\kappa(y)+\mathsf{B}\wh u)
        \bra[\big]{\nablax \wh u\cdot\xi + \wh u\divv_x\xi}\big]\,\dd\nu(y)\dd x,
    \end{aligned}
    \end{equation}
    where $\wh U(x) = \int_Y \kappa(y)\wh u(x,y)\dd \nu(y)$ as before.
\end{proposition}

\begin{proof}
Consider $\xi\in L^\infty(Y; C^2_c(X;\R^d))$ and define for some $t_0>0$ the flow map
$\Phi:[0,t_0]\times X\times Y \to\R^d$ by $\Phi_0(x,y) = x$ 
and $\partial_t\Phi_t = \xi(\Phi_t,y)$, which is well-defined and measurable. 
Next, set $\wh\mu_t \coloneqq {(\Phi_t,\mathrm{id}_Y)}_\#\wh\mu$, 
which hence satisfies (in the sense of distributions) the continuity equation $\partial_t\wh\mu_t + \divv_x(\wh\mu_t \xi_t) = 0$. 
Then, Proposition~\ref{prop:time-derivative_metric} yields
\begin{align*}
    \frac{\mathrm{d}}{\mathrm{d}t}\bigg|_{t=0}\frac{1}{2h}\calW_{2,\nu}^2(\wh\mu_t,\mu)
    = \frac{1}{h}\int_X\int_Y (x- T(x,y))\cdot \xi(x,y)\,\dd\wh\mu(x,y)
    = \int_X\int_Y \xi\cdot\hat\jmath\,\dd\nu\dd x
\end{align*}
With this, the Euler--Lagrange equation for the minimization problem in \eqref{eq:MoreauYosida}
reads
\begin{align}\label{eq:E-L_flux-form}
    \int_X\int_Y \xi\cdot\hat\jmath\,\dd\nu\dd x = -\frac{\dd}{\dd t}\bigg|_{t=0} \calE(\wh\mu_t).
\end{align}
Regarding the right-hand side, we keep in mind the ${L}^2(Y;H^1(X))$-control over $\wh u$ due to the finiteness of $\calE_1$ and argue as in the proof of \cite[Lemma 3.7]{KL_2022} to obtain
\begin{align*}
   \frac{\dd}{\dd t}\bigg|_{t=0}  \calE_1(\wh\mu_t)
    &= -\int_X\int_Y \alpha(y) \big[\bra[\big]{\tfrac12\divv_x \xi I + \nablax\xi}\nablax \wh u + \nablax\divv_x\xi\wh u\big]
     \cdot \nablax\wh u \,\dd\nu(y)\dd x
\end{align*}
for the cohesion part, while for the lower-order part a standard calculation yields
\begin{align*}
   \frac{\dd}{\dd t}\bigg|_{t=0}  \calE_0(\wh\mu_t)_{t=0}
    &= \int_{X\times Y} \big(\msf'(\wh U)\kappa(y) + \msB\wh u\big)\partial_t \wh u|_{t=0} \dd\nu \dd x\\
    &= -\int_{X\times Y} \big(\msf'(\wh U)\kappa(y) + \msB\wh u\big)(\nabla \wh u\cdot \xi +\wh u\divv \xi) \dd\nu \dd x.
\end{align*}
Inserting these into \eqref{eq:E-L_flux-form}, we obtain \eqref{eq:E-L_full} as claimed.
\end{proof}

\subsection{Fibred JKO scheme and discrete inequalities}\label{ssec:fibredJKO}

We now iterate the incremental minimisation problem in Definition~\ref{def:MoreauYosida} 
with time step $h=\tau\in(0,1]$ to construct the
fibred JKO scheme and introduce the piecewise constant and De Giorgi
interpolants, together with their associated fluxes.

From now on and throughout the rest of Section~\ref{sec:diagE}, 
we consider an initial measure $\mu_0\in \calA_\nu$, cf.\eqref{eq:defSetAnu}.
By an elementary interpolation argument, this implies that also $\calH(\mu_0)<\infty$.

\begin{definition}\label{def:JKP_interpolations}
    We define the JKO scheme via
    \begin{align*}
        \mu^0_\tau \coloneqq \mu_0,\quad
        \mu^k_\tau \in \sfJ_\tau(\mu^{k-1}_\tau) \text{ for } k\in\bbN.
    \end{align*}
    We define the piecewise constant interpolation $\mu_\tau:[0,\infty)\to\calP(X{\times}Y)$ by
    $\mu_\tau(0)=\mu_\tau^0$  and
    \begin{align*}
        \mu_\tau(t)\coloneqq \mu^k_\tau
        \quad \text{for}\quad t\in(\tau (k{-}1),\tau k],
    \end{align*}
	and introduce its piecewise constant flux $j_\tau$ given by
	\begin{align*}
        j_\tau(t,x,y) &\coloneqq j^k_\tau(x,y)
        \quad \text{for}\quad t\in(\tau (k{-}1),\tau k],
    \end{align*}
    where $j^k_\tau$ is as in Lemma~\ref{lem:discrete_flux} with $h=\tau, \mu = \mu^{k-1}_\tau$, $\wh \mu =\mu^k_\tau$.
    
    Furthermore, using classical measurable selection results (for details, see e.g.~\cite[Section~2.2]{GHL_preprint}) we may choose a measurable map
 $\wt\mu_\tau:[0,\infty)\to\calP_{2,\nu}(X{\times}Y)$,  the {\em De Giorgi interpolant},  satisfying $\wt\mu_\tau(0)=\mu_\tau^0$  and
    \begin{align*}
        \wt\mu_\tau(t)\in \sfJ_{t-(k{-}1)\tau}(\mu^{k-1}_\tau),
        \;\;\; t\in(\tau (k{-}1),\tau k),\qquad \text{and }\;\;\wt\mu_\tau(\tau k)=\mu_\tau^k
    \end{align*}
    with associated De Giorgi flux $\wt\jmath_\tau$ defined for $k\ge 1$ and a.e.\ $t\in((k{-}1)\tau,k\tau]$ as in Lemma~\ref{lem:discrete_flux} with $h=t-(k{-}1)\tau, \mu = \mu^{k-1}_\tau$, $\wh\mu =\wt\mu_\tau(t)$.
\end{definition}

The one-step estimates of Lemma~\ref{lem:bounds}, 
together with the energy decrease along the scheme, yield the following uniform bounds for both interpolants.

\begin{lemma}\label{lem:GMM_bounds}
    There exist constants $C_1,C_2 >0$ depending only on
    $X,Y,\mu_0$, such that for all $T\in(0,\infty)$ it holds
    \begin{subequations}
    \label{eq:bounds_k-uniform}
    \begin{align}
    \label{eq:moment-bound_k-uniform}
    \begin{split}
        &\sup_{\tau\in(0,T]}\esssup_{t\in[0,T]} 
        m_{2,X}(\mu_\tau) \le C_1 (1+T),
        \quad\text{and}\\
        &\sup_{\tau>0}\esssup_{t\in[0,\infty)} \big(\|\nablax u_\tau(t)\|_{L^2(X\times Y)}^2 + \vartheta\|U_\tau(t)\|_{L^4(X)}^4\big)
        \le C_2.
        \end{split}
    \end{align}
    In particular, there exists $C_3=C_3(X,Y,\mu_0)>0$ such that for all $\phi\in\rmL^\infty([0,\infty){\times} Y;C^1(X))$ we have
    \begin{align}
    \label{eq:nabla-bound_k-uniform}
        &\sup_{\tau>0} \esssup_{t\in[0,\infty)} \norm{\nabla\skp{u_\tau(t),\phi(t)}_Y}_{{L}^2(X;\bbR^d)} \le 
        C_3\| \phi \|_{\rmL^\infty([0,\infty){\times} Y;C^1(X))},
    \end{align}
    and for all $p\in[1,2^\ast)$ or $p = 2^\ast$ if $d\ne 2$, there exists $C_4=C_4(X,Y,\mu_0,p)>0$ such that for all $\phi\in \rmL^\infty([0,\infty){\times} Y;C(X))$
    \begin{align}
    \label{eq:p-bound_k-uniform}
        &\sup_{\tau>0} \esssup_{t\in[0,\infty)} \norm{\skp{u_\tau(t),\phi(t)}_Y}_{{L}^p(X)} \le C_4\|\phi\|_{\rmL^\infty([0,\infty)\times Y;C(X))}.
    \end{align}
    \end{subequations}
    Furthermore, for all $0\le s\le t<\infty$, $\tau>0$ we have
    \begin{align}
    \label{eq:W-bound_embedded}
        \calW_{2,\nu}(\mu_\tau(s),\mu_\tau(t)) \le \sqrt{2\calE(\mu_0)}\bra[\big]{\sqrt{t-s}+\sqrt{\tau}}.
    \end{align}
    The above bounds also hold for the De Giorgi interpolant $\wt\mu_\tau$.
\end{lemma}
\begin{proof}
    By construction, 
    we have for all $\tau > 0$, $k\in\bbN$ the estimate
    \begin{align}\label{eq:GMM_bounds-proof_Ebound}
        \frac{\calW_{2,\nu}^2(\mu^{k-1}_\tau,\mu^k_\tau)}{2\tau} + \calE(\mu^k_\tau)\le \calE(\mu^{k-1}) \le \calE(\mu_0)
    \end{align}
    By Lemma~\ref{lem:bounds} (and using \eqref{eq:energy.squeeze} to control the energy for the De Giorgi interpolant), this ensures \eqref{eq:bounds_k-uniform}.
    To show \eqref{eq:W-bound_embedded} for the piecewise constant interpolation observe that $s \in ((k_1{-}1)\tau,k_1\tau]$, $t \in ((k_2{-}1)\tau,k_2\tau]$ for some $k_1, k_2\in\bbN$, $k_1\le k_2$.
    Hence, using Cauchy-Schwarz and \eqref{eq:GMM_bounds-proof_Ebound}, it follows
    \begin{align*}
        \calW_{2,\nu}(\mu_\tau(s),\mu_\tau(t)) 
        &\le \sum_{k=k_1+1}^{k_2} \calW_{2,\nu}(\mu^{k-1}_\tau,\mu^k_\tau)
        \le \sqrt{2\calE(\mu_0)}\bra[\big]{\sqrt{t-s}+\sqrt{\tau}}.
    \end{align*}
    In light of \eqref{eq:distance.squeeze}, the same argument also proves \eqref{eq:W-bound_embedded} for the De Giorgi interpolation.
\end{proof}

We next combine the identity from Lemma~\ref{lem:Moreau-Yosida-Properties} with
the discrete flux representation from Lemma~\ref{lem:discrete_flux} to obtain an
energy-dissipation inequality involving both interpolants. 
The resulting
flux bounds will provide the weak compactness needed to pass to
the limit.

\begin{proposition}[Energy identity and flux bounds]\label{prop:disc_EDI}
	For all but countably many $\tau>0$ and for a.e.\ $t\in[0,\infty)$, the interpolations introduced in Definition~\ref{def:JKP_interpolations} satisfy
    \begin{align}\label{eq:disc_EDI}
		\calE(\mu_\tau(t)) +\int_0^t\int_X\int_Y
		\bra[\bigg]{
		\frac{\abs{j_\tau}^2}{2u_\tau}
		+ \frac{\abs{\wt\jmath_\tau}^2}{2\wt u_\tau} 
		}
		\,\dd\nu\dd x\dd t
		\le \calE(\mu_0).
	\end{align}
    In particular, 
    for every $A\in\calB([0,\infty)\times X\times Y)$, the following estimates hold
    \begin{align}\label{eq:flux_uniform-integrability}
        \norm{j_\tau}_{{L}^1(A)} \le \sqrt{2\calE(\mu_0)\, \norm{u_\tau}_{{L}^1(A)}},\qquad  \quad\norm{\wt\jmath_\tau}_{{L}^1(A)} \le \sqrt{2\calE(\mu_0)\, \norm{\wt u_\tau}_{{L}^1(A)}}.
    \end{align}
\end{proposition}
\begin{proof}
	We apply Lemma~\ref{lem:Moreau-Yosida-Properties} to obtain for each $k\in\bbN$ 
	\begin{align*}
		\frac{\calW^2_{2,\nu}(\mu^{k-1}_\tau,\mu^k_\tau)}{2\tau} 
		+ \int_0^\tau \frac{\bra[\big]{\calW^\pm_{2,\nu,s}(\mu^{k-1}_\tau)}^2}{2s^2}\, \dd s 
		= \calE(\mu^{k-1}_\tau) -\calE(\mu^k_\tau).
	\end{align*}
	Next, Lemma~\ref{lem:discrete_flux} allows us to identify the left-hand side to obtain for $k\in\bbN$ 
	\begin{align*}
		\tau\int_X\int_Y
		\frac{\abs{j_\tau}^2}{2u_\tau}\,\dd\nu\dd x 
		+ \int_0^\tau\int_X\int_Y \frac{\abs{\wt\jmath_\tau}^2}{2\wt u_\tau} \,\dd t 
		= \calE(\mu^{k-1}_\tau) -\calE(\mu^k_\tau).
	\end{align*}
	This proves \eqref{eq:disc_EDI} upon summing over $k$ and recalling the definitions of the interpolations. 
    The bound \eqref{eq:flux_uniform-integrability} follows from the non-negativity of $\calE$ by setting $s=0$ and letting $t\to\infty$ in \eqref{eq:disc_EDI}, and applying the Cauchy-Schwarz inequality.
\end{proof}

In the present case of $X$ being simply the whole Euclidean space $\mathbb{R}^d$ (and more generally $X\subset \mathbb{R}^d$ a smooth convex domain), 
we may further apply the flow-interchange technique introduced in \cite{MMS} to derive improved uniform bounds for the piecewise constant interpolation and more specifically a discrete entropy-dissipation inequality. The key observation is that
the logarithmic entropy $\mathcal H$ generates the fibrewise heat $S^t$
flow in the fibred Wasserstein geometry, see \cite{peszek2023heterogeneous}. Indeed, writing
$S^t\mu=(S^tu)\mathcal L^d\otimes\nu$, the fibrewise heat flow is a $0$-flow for the entropy and
satisfies an evolution variational inequality.

\begin{lemma}[Discrete entropy inequality]\label{l:disc_entropy-inequality}
	For all $T\in(0,\infty)$, it holds that 
	$u_\tau\in L^2([0,T]\times Y;H^2(X))$ and
	\begin{multline}\label{eq:entropy-inequality.discr}
		\calH(\mu_\tau(t)) + \int_0^t\int_X\pra[\bigg]{\int_Y\big(
			\alpha(y) \abs{\Delta_x u_\tau}^2+[\msB \nabla_x u_\tau ]\cdot\nabla_x u_\tau\big)\,\dd\nu(y)
			+\msf''(U_\tau)|\nabla_xU_\tau|^2 \,}\dd x\dd r
		\le \calH(\mu_0)
	\end{multline}
	for all $\tau \in(0,T]$ and a.e.\ $t\in[0,T]$.
\end{lemma}
\begin{proof}
	Given $\mu=u\scrL^d{\otimes}\nu\in\calP_{2,\nu}(X{\times}Y)$, denote for $x\in X$, $y\in Y$, $t>0$
	\begin{align*}
		S^t u(x,y)\coloneqq [G^t\ast u(\cdot,y)](x) \coloneqq \frac{1}{(4\pi t)^{d/2}}\int_X\exp\bra[\bigg]{{-}\frac{\abs{x-\bar x}^2}{4t}} u(\bar x,y)\,\dd \bar x
	\end{align*}
	the heat semigroup in $X$ applied to $u(\cdot,y)$ and let 
    $S^t\mu \coloneqq (S^t u) \scrL^d{\otimes}\nu$. Note that
    $\langle S^t u,\phi\rangle_Y = S^t\langle u,\phi\rangle_Y$ for all $\phi\in C(Y)$.
	We apply \cite[Theorems~4.6 and 3.34]{peszek2023heterogeneous} to obtain for all $\mu\in\calP_{2,\nu}(X{\times}Y)$, $\tilde\mu\in\mathrm{dom}(\calE)$
	\begin{align*}
		\frac{1}{2}\frac{\dd^+}{\dd t}\calW_{2,\nu}^2(\tilde\mu,S^t\mu) \le \calH(\tilde\mu) - \calH(S^t\mu).
	\end{align*}
	With this, fixing $\tau>0$, $k\in\bbN$, and using that $\mu^k_\tau$ minimises $\tilde\mu\mapsto\frac12\calW_{2,\nu}^2(\mu^{k-1}_\tau,\tilde\mu) + \tau\calE(\tilde\mu)$,
	it follows
	\begin{align}\label{eq:entropy-identity_proof-step1}
		-\tau\limsup_{s\downarrow 0} \int_0^1\pra[\bigg]{\frac{\dd}{\dd r}\calE(S^r\mu^k_\tau)}_{r = st}\,\dd t 
		\le \pra[\bigg]{\frac{1}{2}\frac{\dd^+}{\dd t} \calW_{2,\nu}^2(\mu^{k-1}_\tau,S^t\mu^k_\tau)}_{t=0}
		\le \calH(\mu^{k-1}_\tau) - \calH(\mu^k_\tau).
	\end{align}
	Using that the semigroup solves the diffusion equation with respect to  $x$, exploiting \eqref{eq:assAlpha}, and integrating by parts%
	, we obtain 
	\begin{multline}
		\label{eq:entropy-identity_proof-step2}
		\frac{\dd}{\dd t}\calE(S^t\mu^k_\tau) = 
		- \int_X\int_Y
		\alpha(y) 
		\abs{\Delta_x S^t u^k_\tau}^2\,\dd\nu\dd x
		+\int_X  \msf'(S^tU_\tau^k)\Delta_xS^t U^k_\tau\, \dd x+ \int_{X}\int_Y [\msB S^t u_\tau^k] \Delta_x S^tu_\tau^k \,\dd\nu\dd x\\
 = 
		- \int_X\int_Y
		\alpha(y) 
		\abs{\Delta_x S^t u^k_\tau}^2\,\dd\nu\dd x
		-\int_X  \msf''(S^tU_\tau^k)|\nabla_xS^t U^k_\tau|^2\, \dd x- \int_{X}\int_Y [\msB \nabla_xS^t u_\tau^k ]\cdot\nabla_x S^tu_\tau^k \dd\nu\dd x
\end{multline}
With the decomposition of $\msf= \msf_\text{cvx}+\msf_\text{conc}$  into convex and (quadratic) concave part (see also proof of Lemma~\ref{lem:energy-hierarchy}), we find that
\begin{equation}
\label{eq:entropy-identity_proof_Cauchy-Schwarz+Young}
		\frac{\dd}{\dd t}\calE(S^t\mu^k_\tau) \le - c \norm{\Delta_x S^t \mu^k_\tau}^2_{L^2(X\times Y)} 
         +C \big(\vartheta\|\nabla_x S^t U_\tau^k \|_{L^2(X)}^2 + \|\nabla_x S^t u_\tau^k \|_{L^2(X\times Y)}^2\big).
\end{equation}

Combining this with \eqref{eq:entropy-identity_proof-step1}, recalling that the second summand is uniformly bounded in both $k$ and $\tau$ by \eqref{eq:moment-bound_k-uniform} and \eqref{eq:nabla-bound_k-uniform}, we arrive at
\begin{align*}
	\tau\liminf_{s\downarrow 0} \int_0^1\pra[\bigg]{\int_X\int_Y\abs{\Delta_x S^r u^k_\tau(x,y)}^2\,\dd\nu(y)\dd x}_{r = st}\dd t
	&\le C\bra[\big]{\calH(\mu^{k-1}_\tau) - \calH(\mu^k_\tau) + \tau}.
\end{align*}
Since for $\nu$-a.e.\ $y$ it holds $u^k_\tau(\cdot,y)\in H^1(X)$, we also obtain $S^r u^k_\tau(\cdot,y) \to u^k_\tau(\cdot,y)$ strongly in $H^1(X)$ and hence also in ${L}^p(X)$ for $p < 2^\ast$ ($p\le2^\ast$ if $d \ne 2$). 
Furthermore, recall that $\int_X \abs{\Delta f}^2\dd x = \int_X \abs{\nabla^2 f}^2\dd x$ is lower semicontinuous with respect to  weak convergence in $H^2(X)$.
Hence, since the negative part of $\calH(\mu^k_\tau)$ is controlled by $m_{2,X}(\mu^k_\tau)$ (cf.\ e.g.\ \cite[Eq.\ (14)]{JKO98FPE}), we infer $u_\tau\in L^2([0,T]\times Y;H^2(X))$.

To prove the discrete entropy inequality in \eqref{eq:entropy-inequality.discr}, 
we insert \eqref{eq:entropy-identity_proof-step2} into \eqref{eq:entropy-identity_proof-step1}, sum with respect to  $k$ and again use that $S^r u^k_\tau(\cdot,y) \to u^k_\tau(\cdot,y)$ strongly in ${L}^p(X)$ and that $\Delta_x S^r u^k_\tau(\cdot,y) \rightharpoonup \Delta_x u^k_\tau(\cdot,y)$ weakly in ${L}^2(X)$ as shown before to estimate the $\limsup$ with respect to  $s$. This concludes the proof.
\end{proof}

\begin{corollary}
    The piecewise constant interpolation $u_\tau$ is such that $\msf''_\mathrm{cvx}(U_\tau)|\nabla U_\tau|^2\in L^1(X)$, where $\msf_\mathrm{cvx}$ denotes the convex part of $\msf$.
\end{corollary}

Finally, the piecewise constant and the De Giorgi interpolant satisfy 
Euler--Lagrange equations, respectively. This follows directly from the one-step result in Proposition~\ref{prop:E-L_1}. Let us remark that in the case of $\supp\nu$ being a singleton, the restriction $d \leq 3$ can be removed, cf.\ Remark~\ref{rem:Singleton}.
\begin{proposition}[Euler--Lagrange equations]\label{prop:E-L_2}
    Let $\tau>0$. For every $\xi\in L^\infty(Y; C^2_c(X;\R^d))$ and a.e.\ $t\ge 0$ 
    the following identity holds:
    \begin{equation}\label{eq:flux.EL}
    \begin{aligned}
        \int_X\int_Y \xi\cdot j_\tau\,\dd\nu\dd x
        &= \int_X\int_Y \alpha(y) \big[\bra[\big]{\tfrac12\divv_x \xi I + \nablax\xi}\nablax u_\tau + u_\tau \nablax\divv_x\xi \big]
        \cdot \nablax u_\tau \,\dd\nu\dd x
        \\
       &\quad
        + \int_X\int_Y (\msf'(U_\tau)\kappa(y)+\mathsf{B} u_\tau)
        \bra[\big]{\nablax u_\tau\cdot\xi + u_\tau\divv_x\xi}\,\dd\nu\dd x,
    \end{aligned}
    \end{equation}
    where $U_\tau$ is defined as in \eqref{eq:defLowerEnergy}. 
    The same result holds for the De\,Giorgi interpolant $\wt u_\tau$ with the flux $\wt\jmath_\tau$.
\end{proposition}

\subsection{Compactness and convergence%
}
\label{ssec:comp-conv}

We now pass to the limit as $\tau\downarrow0$. The uniform moment
bounds \eqref{eq:moment-bound_k-uniform} and equicontinuity in time \eqref{eq:W-bound_embedded} in Lemma~\ref{lem:GMM_bounds} yield a common narrowly continuous
limit for both interpolants, while the energy bounds provide weak
compactness of their densities in $L^2(0,T;L^2(Y;H^1(X)))$.
Although compactness in the type variable $y$ is unavailable, averages
against suitable test functions in $y$ converge strongly (see Lemma~\ref{l:GMM_convergence_strong} below). Combining
this convergence with the variational comparison from Lemma~\ref{lem:energy-hierarchy},
we show that the two interpolants have the same limiting energy
and that their difference converges strongly. These properties
will allow us to identify a common limiting flux and retain the
full transport dissipation.

\begin{lemma}[Narrow compactness]\label{l:GMM_time-convergence_bounds}
    For all $T\in(0,\infty)$, there exists $\mu \in \AC^2([0,T];\calP_{2,\nu}(X{\times}Y))$
    such that up to a subsequence both $\mu_\tau(t) \rightharpoonup \mu(t)$ and $\wt\mu_\tau(t) \rightharpoonup \mu(t)$ narrowly in $\calP(X{\times}Y)$ uniformly for $t\in[0,T]$, where $t\mapsto\mu(t)$ is $\frac{1}{2}$-Hölder continuous in the sense that,
    for $0\le s\le t\le T$,
    \begin{align}\label{eq:Hoelder-continuity}
        \calW_{2,\nu}(\mu(s),\mu(t)) \le C\sqrt{t-s}.
    \end{align}
\end{lemma}
\begin{proof}
    By~\eqref{eq:W-bound_embedded} both interpolants $\mu_\tau$ and $\wt\mu_\tau$ are $\tfrac{1}{2}$-Hölder continuous up to an error vanishing as $\tau\to 0$.
    Furthermore, since their second moments with respect to $X$ are $\tau$-uniformly bounded by~\eqref{eq:moment-bound_k-uniform} and their marginal in $Y$ is fixed to $\nu$, both sequences are uniformly tight and hence narrowly relatively compact by Prokhorov's theorem (cf.\ e.g.\ \cite[Theorem 5.1.3]{AGS}).
    This allows us to apply the refined Ascoli-Arzelà theorem \cite[Proposition 3.3.1]{AGS} to obtain curves $\mu,\wt\mu:[0,T]\to\calP(X\times Y)$ each satisfying~\eqref{eq:Hoelder-continuity}, 
    such that $\mu_\tau(t) \rightharpoonup \mu(t)$ and  $\wt\mu_\tau(t) \rightharpoonup\wt\mu(t)$ narrowly in $\calP(X\times Y)$ uniformly with respect to $t\in[0,T]$ as $\tau\to 0$, where we have also used that $\calW_{2,\nu}$ is lower semicontinuous with respect to narrow convergence (cf.\ \cite[Proposition~3.12]{peszek2023heterogeneous}).
    Due to~\eqref{eq:GMM_bounds-proof_Ebound} and~\eqref{eq:distance.squeeze}, we have $\mu = \wt\mu$.
    To obtain the claimed absolute continuity, for $\tau >0$ we set $v_\tau(t)\coloneqq \tfrac{1}{\tau} \calW_{2,\nu}(\mu_\tau^{k-1},\mu_\tau^k)$ for $t \in (\tau(k-1),\tau k)$, which is $\tau$-uniformly bounded in $\rmL^2([0,T])$ by~\eqref{eq:GMM_bounds-proof_Ebound} and hence has a weak limit $v\in\rmL^2([0,T])$.
    Then, exploiting again the narrow lower semicontinuity of $\calW_{2,\nu}$ and using the triangle inequality in the prelimit, we find for $0\le s \le t \le T$
    \begin{align*}
        \calW_{2,\nu}(\mu(s),\mu(t)) 
        \le \liminf_{\tau\to0} \calW_{2,\nu}(\mu_\tau(s),\mu_\tau(t)) 
        \le \lim_{\tau\to 0}\int_s^t v_\tau(r)\,\dd r 
        = \int_s^t v(r)\,\dd r.
    \end{align*}
    In particular, we have $\mu \in \AC^2([0,T];\calP_{2,\nu}(X{\times}Y))$ and the proof is complete.
\end{proof} 

We now combine the common narrow convergence from Lemma~\ref{l:GMM_time-convergence_bounds} with
the uniform spatial gradient and moment bounds from Lemma~\ref{lem:GMM_bounds}.

\begin{lemma}[Spatial regularity and modes of convergence]\label{l:GMM_convergence_strong}
    Let $\mu_0 \in \calP_{2,\nu}(X{\times}Y)$ with $\calE(\mu_0)<\infty$. 
    Fix $T\in(0,\infty)$ and let $u$ be the $(\scrL^{d+1}{\otimes}\nu)$-density of the limit $\mu \in \AC^2([0,T];\calP_{2,\nu}(X{\times}Y))$ obtained in Lemma~\ref{l:GMM_time-convergence_bounds}
    \begin{enumerate}[label={\normalfont(\roman*)}]
    \item\label{l-item:GMM_convergence_weak} The limit density $u$ is such that $\nabla u\in L^\infty(0,T;L^2(X{\times Y}))$. Moreover, the De Giorgi  interpolation $\wt\mu_\tau$ 
	and the piecewise constant interpolation $\mu_\tau$ with densities $\wt u_\tau$ and $u_\tau$, respectively, satisfy for the same subsequence as in Lemma \ref{l:GMM_time-convergence_bounds}
	\begin{align} \label{eq:interpolWeakConv}
        \wt u_\tau, u_\tau\rightharpoonup u
        \text{ weakly in $L^2([0,T]{\times} Y;H^1(X))$}.
    \end{align}
    \item\label{l-item:GMM_convergence_strong} For all $q \in[1,2^*)$ 
    and all $\phi\in C([0,T];L^\infty(Y;C(X)))$ both interpolations satisfy along the same subsequence
	the strong convergences
	\begin{align}\label{eq:interpolTestedStrongConvA}
        \skp{\wt u_\tau,\phi}_Y,\skp{u_\tau,\phi}_Y \to \skp{u,\phi}_Y 
        \text{ strongly in }L^\infty(0,T; L^q(X)).
    \end{align}
    \end{enumerate}
\end{lemma}

\begin{remark}
\label{rem:betterConvergenceSingl}
    Note that if $\supp \nu$ is a singleton and  $\vartheta>0$ the restriction $d\leq 3$ can be dropped. 
    Due to the uniform $L^4$ bound for $U_\tau$, we can conclude strong convergence of $ \skp{u_\tau,\phi}_Y$
    and $ \skp{\wt u_\tau,\phi}_Y$ in $L^\infty(0,T; L^q(X))$ for all $1\leq q< 4$. In connection with the entropy-dissipation inequality in \eqref{eq:entropy-inequality.discr}, we can show that $U_\tau^2$ is bounded in $L^\infty(0,T;L^2(X))\cap L^2(0,T;H^1(X))$. This result can be exploited to improve the strong convergence  to the case 
    $q= 4$, see Subsection~\ref{ssec:ExistNonconvex} for a similar argument.
\end{remark}

\begin{proof}

	\ref{l-item:GMM_convergence_weak}: Using the bounds from Lemma~\ref{lem:GMM_bounds} yields the convergences in
	\eqref{eq:interpolWeakConv} both for the De Giorgi and the piecewise constant interpolants, up to a further subsequence. 
    However, since the limit is uniquely identified in  Lemma~\ref{l:GMM_time-convergence_bounds}, the whole subsequence has to converge.

	\ref{l-item:GMM_convergence_strong}: 
    We focus on the piecewise constant interpolation, the proof for the De Giorgi interpolation being analogous.
    Let us first consider $\wt\phi\in \rmL^\infty(Y;C^1(X))$.     
    The narrow convergence $\mu_\tau(t)\weakly\mu(t)$ uniformly in $t\in[0,T]$ in Lemma~\ref{l:GMM_time-convergence_bounds} implies, in particular, the uniform-in-time convergence of $\langle{u_\tau}(t),\wt\phi\rangle_Y$ to  $\langle u(t),\wt\phi\rangle_Y$ in $H^{-m}(X)$ for $m>d/2$, see Theorem~\ref{thm:negSobolev}. In  the general case of a time-dependent test function
    $\phi\in C([0,T];L^\infty(Y;C^1(X)))$, we fix a further partition 
    $\{t_i^N\}_{i=0}^N$ of $[0,T]$ with $t_i^N= iT/N$ for $i=0,\ldots,N$, leading us for arbitrary $t\in[0,T]$ to the estimate
    \begin{align*}
\sup_{t\in [0,T]} \|\langle z_\tau(t),\phi(t)\rangle_Y\|_{H^{-m}(X)}
&\leq \max_{1\leq i\leq N} \sup_{t\in [0,T]} \|\langle z_\tau(t),\phi(t_i)\rangle_Y\|_{H^{-m}(X)}
\\
&\quad +\max_{1\leq i\leq N} \sup_{t\in [0,T]} \|\langle z_\tau(t),\phi(t) {-} \phi(t_i)\rangle_Y\|_{H^{-m}(X)}\\
& \leq \max_{1\leq i\leq N} \sup_{t\in [0,T]} \|\langle z_\tau(t),\phi(t_i)\rangle_Y\|_{H^{-m}(X)}+C_{m,d}C_\phi\delta_N,
    \end{align*}
    where the first term on the right-hand side vanishes as $\tau\to 0$ due to the above discussion 
    with time-independent test function $\wt \phi$ and the second term can be made arbitrarily small 
    by choosing $N$ large since $\phi$ is uniformly continuous with respect to $t\in[0,T]$. 
    
    Combining the strong convergence in $H^{-m}(X)$ with the uniform boundedness of $\langle{u_\tau},\phi\rangle_Y$ in the space
    $L^\infty(0,T;H^1(X))$, we get by Sobolev interpolation the strong convergence of the latter in $L^\infty(0,T; L^q(X))$ for $q\in [2,2^*)$. 
    Indeed, we  have for $z_\tau = u_\tau - u$ that
    \[
    \|\langle z_\tau(t),\phi\rangle_Y\|_{L^q(X)}
\leq
\|\langle z_\tau(t),\phi\rangle_Y\|_{H^{-m}(X)}^{\theta}
\|\langle z_\tau(t),\phi\rangle_Y\|_{H^1(X)}^{1-\theta},
    \]
    for some $\theta\in(0,1)$ depending on $d,m,q$. For $q\in (1,2)$, the
    strong convergence follows from interpolation between the Lebesgue spaces
    $L^1(X)$ and $L^2(X)$ (note that $\|\langle z_\tau,\phi\rangle_Y\|_{L^\infty(0,T;L^1(X))} \leq 2\|\phi\|_{L^\infty}$ as $u_\tau$ and $u$ have unit mass). 
    The limiting case $q=1$, is a consequence of 
    the moment bounds: 
    \begin{align*}
    \|\langle z_\tau(t),\phi\rangle_Y\|_{L^1(X)}
    &\leq     \|\langle z_\tau(t),\phi\rangle_Y\|_{L^1(B_R)} + 
        \|\langle z_\tau(t),\phi\rangle_Y\|_{L^1(X\setminus B_R)}\\
    & \leq C R^{d/2}\|\langle z_\tau(t),\phi\rangle_Y\|_{L^2(X)} + \frac{\|\phi\|_{L^\infty(Y;C(X))}}{R^2}\big(m_{2,X}(\mu_\tau(t)) + m_{2,X}(\mu(t))\big),
    \end{align*}
    where the first term in the right-hand side vanishes as $\tau\to 0$ and the second term can be made arbitrarily small.

    Finally, the above argument extends to $\phi\in C([0,T]; \rmL^\infty(Y;C(X)))$, since $C([0,T];L^2(Y;C^1(X)))\subset C([0,T];L^2(Y;C(X)))$, arguing analogously to the end of the proof of Lemma~\ref{lem:bounds}.
\end{proof}

The convergence of the type averages together with the variational
comparison from Lemma~\ref{lem:energy-hierarchy} are now used to show that 
both interpolants have the same limiting energy and that their difference converges strongly.
Let us emphasise that this argument does not require compactness in the type variable
or the additional regularity obtained by flow interchange.

\begin{proposition}[Energy convergence]\label{prop:energy-convergence}
There exists a non-increasing function $E:[0,\infty)\to\mathbb{R}$ with $E(0) = \calE(\mu_0)$ such that along a subsequence $\tau\downarrow0$ it holds for a.e.\ $t>0$,
\begin{subequations}\label{eq:Econv}
	\begin{align}
    \label{eq:energy-limit}
    E(t) &= \lim_{\tau\downarrow0}\calE(\mu_\tau(t)) = \lim_{\tau\downarrow0}\calE(\wt\mu_\tau(t)).
\end{align}	
Moreover, for all $p\in[1,\infty)$ and all $T\in(0,\infty)$, we have
\begin{align}
\nabla u_\tau-\nabla\wt u_\tau &\to0\qquad\text{ in }L^p(0,T;L^2(X{\times} Y)),\quad 
	\label{eq:strongGrad.diff} \\  
		u_\tau-\wt u_\tau&\to 0\qquad\text{ in }L^p(0,T;L^q(X{\times} Y)),\quad%
		 q\in[1,2+\tfrac{2}{d}).	\label{eq:strongH1.diff.p}
\intertext{Additionally, in the case $\vartheta>0$, it holds that}
  U_\tau - \wt U_\tau&\to 0\qquad\text{ in }L^p(0,T;L^4(X)),\quad\label{eq:strongL4}
\end{align}
\end{subequations}
If $\nu$ is supported on finitely many points, then  $E(t)=E(u(t))$ and the limits~\eqref{eq:strongGrad.diff} and~\eqref{eq:strongH1.diff.p} strengthen to the individual limits $\nabla u_\tau,\nabla\wt u_\tau \to \nabla u$ in $L^p_tL^2_{x,y}$ and $u_\tau,\wt u_\tau \to u$ in $L^p_tL^q_{x,y}$, respectively. 
\end{proposition}

\begin{proof}
	We first show~\eqref{eq:energy-limit}.
    From Lemma~\ref{lem:energy-hierarchy} it follows that both $t\mapsto\calE(\mu_\tau(t))$ and $t\mapsto\calE(\wt\mu_\tau(t))$ are non-increasing.
    Hence, by Helly’s selection theorem, there exist limit functions $\wt E, E:[0,T]\to[0,\calE(\mu_0)]$ such that along a subsequence $\tau\downarrow 0$ it holds for a.e.\ $t>0$
    \begin{align*}
        &\wt E(t)=\lim_{\tau\downarrow0}\calE(\wt\mu_\tau(t)), \qquad
        E(t)=\lim_{\tau\downarrow0}\calE(\mu_\tau(t)).
    \end{align*}
    Since both interpolations keep the initial datum fixed, the identities $E(0)= \wt E(0) = \calE(\mu_0)$ are guaranteed.
    To show \eqref{eq:energy-limit} it suffices to prove for all $0 < t < \infty$ that
    \begin{align*}
        \int_0^t\wt E(s)\,\dd s=\int_0^t E(s)\,\dd s.
    \end{align*}
    Indeed, after subtracting the respective identity on a shorter time interval, this gives $\int_I\wt E(s)\,\dd s=\int_I E(s)\,\dd s$ for every compact interval $I\subset[0,\infty)$, which implies the assertion.

    To this end, we combine \eqref{eq:energy.squeeze} with the definitions of the interpolations and integrate over $[0,t]$ to infer for a.e.\ $t> 0$
    \begin{align*}
        \int_\tau^{t+\tau}\calE(\mu_\tau)\,\dd s\le \int_0^t\calE(\wt\mu_\tau)\,\dd s\le \int_0^t\calE(\mu_\tau)\,\dd s.
    \end{align*}
	From this, we conclude the asserted energy convergence property~\eqref{eq:energy-limit} by observing that
	\begin{align*}
        0\le \int_0^t\calE(\wt\mu_\tau)\,\dd s - \int_\tau^{t+\tau}\calE(\mu_\tau)\,\dd s
        \le \int_0^\tau\calE(\mu_\tau)\,\dd s-\int_t^{t+\tau} \calE(\mu_\tau)\,\dd s\le \tau\pra{\calE(\mu_0)-\calE(\mu_\tau(t+\tau))}
        \overset{\tau\downarrow 0}{\longrightarrow} 0. 
    \end{align*}
 Next, \eqref{eq:strongH1.Econtrol} implies that $z_\tau\coloneqq u_\tau-\wt u_\tau$ and $Z_\tau\coloneqq U_\tau - \wt U_\tau$ satisfy
\begin{align*}
	\int_0^t\big(\|\nabla z_\tau\|_{L^2(X{\times}Y)}^2 + \vartheta\|Z_\tau\|_{L^4(X)}^4 \big) \dd s\lesssim  \int_0^t\calE(\wt\mu_\tau)\,\dd s - \int_0^t\calE(\mu_\tau)\,\dd s + \int_0^t\calB_0(\wt\mu_\tau-\mu_\tau)\dd s,
\end{align*}
where $\calB_0(\wt\mu_\tau-\mu_\tau)$  is the error term arising
from the non-convexity of the lower order energy (cf.\ \eqref{eq:errorB0Energy}). In particular, the latter vanishes due to the strong convergence of $U_\tau$ and $\wt U_\tau$ in \eqref{eq:interpolTestedStrongConvA}. Note in particular, that Lemma~\ref{l:ApproxRepul} yields the strong continuity in $L^2(X\times Y)$.
 
The strong convergence~\eqref{eq:strongGrad.diff} for $p\in[1,2]$ is a consequence of the energy convergence in time-integrated form. For $p>2$ we use interpolation inequalities
\[
    \|\nabla z_\tau\|_{L^p_tL^2_{x,y}}^p\le \|\nabla z_\tau\|_{L^\infty_tL^2_{x,y}}^{p-2}\|\nabla z_\tau\|_{L^2_tL^2_{x,y}}^2,\quad\text{and}\quad
        \|Z_\tau\|_{L^p_tL^4_{x}}^p\le \|Z_\tau\|_{L^\infty_tL^4_{x}}^{p-2}\| Z_\tau\|_{L^4_tL^4_{x}}^2.
\]
The first factors on the right-hand sides are bounded above by $C(\calE(\mu_0))$, while, as  shown above, the second factors converge to zero. 

The convergence~\eqref{eq:strongH1.diff.p} follows from~\eqref{eq:strongGrad.diff} and the interpolation inequality~\eqref{eq:interpol.H1-M}.

If $\supp\nu=\{y_i\}_{i=1}^n\subset Y$ is finite, then the narrow convergence $\mu_\tau, \wt\mu_\tau\weakly\mu$ in Lemma~\ref{l:GMM_convergence_strong} implies 
\[d_{BL,\nu}(\mu_\tau,\mu)=\sum_{i=1}^nd_{BL}(\mu_\tau^{y_i},\mu^{y_i})m_i\to 0,\quad\qquad m_i\coloneqq \nu(\{y_i\}),\]
and the same for $\wt\mu_\tau$.
Thus, using the $L^2_{t,y}H^2_x$ control for the piecewise constant interpolant from Lemma~\ref{l:disc_entropy-inequality} and interpolating between $H^2(X)$ and $(W^{1,\infty}(X))^*$, we obtain the asserted strong convergence of both time-interpolants to $u$. 
\end{proof}

\subsection{Proof of Theorem~\ref{thm:existence}}\label{ssec:proof.ex.diag}

We now complete the proof of Theorem~\ref{thm:existence} by passing to the limit in
the approximate equations and dissipation 
inequalities. The convergence results of Section~\ref{ssec:comp-conv} 
first allow us to identify a common limiting
flux $j$ for both interpolants and establish the energy-dissipation
property in \eqref{eq:EDI.diagE}, including the continuity equation \eqref{eq:CE.diagE}. We then construct the
limiting entropy, derive the entropy-dissipation inequality \eqref{eq:entr.diss}, and
control the discrepancy between the limiting and constitutive fluxes
by the energy and entropy defects, cf.~\eqref{eq:mv-flux.control}. For finitely many types, strong
convergence eliminates these defects and yields a weak solution.

\subsubsection{Energy-dissipating property}

\paragraph{Convergence and energy dissipation}
By the uniform bounds in Proposition~\ref{prop:disc_EDI}, we can extract a subsequence such that $j_\tau$ and $\wt\jmath_\tau$ converge weakly in ${L}^1([0,T]\times X\times Y;\bbR^d)$. By virtue of the strong convergence properties~\eqref{eq:strongGrad.diff} and~\eqref{eq:strongH1.diff.p}, and the Euler--Lagrange equations in Proposition~\ref{prop:E-L_2}, which express $j_\tau$, $\wt\jmath_\tau$ in terms of gradient and density only, the subsequential weak limits must coincide, i.e., there exists $j\in {L}^1([0,T]\times X\times Y;\bbR^d)$
such that %
\begin{align*}
	j_\tau, \;\wt\jmath_\tau \rightharpoonup j \quad\text{ weakly in ${L}^1([0,T]{\times} X{\times} Y;\bbR^d)$}.
\end{align*}
Hence, combined with the weak convergence in Lemma~\ref{l:GMM_convergence_strong} and the energy convergence in Proposition~\ref{prop:energy-convergence}, we can pass to the limit in the discrete energy inequality~\eqref{eq:disc_EDI} using the weak lower semicontinuity of the integral terms with convex integrand to infer its continuous version~\eqref{eq:EDI.diagE}.

\paragraph{Continuity equation}
We next turn to the derivation of the distributional formulation of the continuity equation~\eqref{eq:CE.diagE}, which is obtained from a mostly standard Taylor expansion argument adapted to the fibred setting.
Let $\varphi\in C_c^\infty([0,T)\times X\times Y)$ be given. Further, denote by  $\wh T^k_\tau=(T^k_\tau,\mathrm{\id}_Y):X\times Y\to X\times Y$ the (augmented) fibred optimal transport map from $\mu_\tau^k$ to $\mu_\tau^{k-1}$, where the measurable map $T^k_\tau$ is defined according to Lemma~\ref{lem:discrete_flux}. Then,
\[(\wh T^k_\tau)_{\#}\mu_\tau^k=\mu_\tau^{k-1}\]
so that, with $\vp_k\coloneqq \vp(k\tau,\cdot)$,
\[\int_{X\times Y}\varphi_{k-1}\,\dd(\mu_\tau^k-\mu_\tau^{k-1})=\int_{X\times Y}
\big(\vp_{k-1}(x,y)-\vp_{k-1}(T^k_\tau(x,y),y)\big)\,\dd\mu^k_\tau(x,y).\]
Recalling $j_\tau^k(x,y)=\frac{x-T_\tau^k(x,y)}{\tau}u_\tau^k(x,y),$ 
we deduce, upon Taylor expansion in $x$ with $\nabla=\nabla_x$,
\[\bigg|\int_{X\times Y}\vp_{k-1}\,\dd(\mu_\tau^k-\mu_\tau^{k-1})
-\tau\int_{X\times Y}\nabla\vp_{k-1}\cdot j_\tau^k\,\dd\nu\dd x\bigg|
\lesssim \|\nabla^2\varphi\|_{L^\infty_{t,x,y}}\calW_{2,\nu}^2(\mu_\tau^k,\mu_\tau^{k-1}).\]
We then take the sum over $k$. 
The resulting right-hand side vanishes as $\tau\downarrow0$ due to the bound
\[\sum_k\calW_{2,\nu}^2(\mu_\tau^k,\mu_\tau^{k-1})\le \tau\calE(\mu_0).\]
It remains to identify the limit of the term on the left-hand side. To this end, we use discrete integration by parts in time, 
\[\sum_k\int_{X\times Y}\vp_{k-1}\,\dd(\mu_\tau^k-\mu_\tau^{k-1})
=-\sum_k\int_{X\times Y}(\vp_k-\vp_{k-1})\,\dd \mu_\tau^k-\int_{X\times Y}\vp(0)\,\dd\mu_0,\] 
where we took $\tau>0$ small enough  so that due to the compact support of $\vp$ in time the temporal boundary integral at the final time vanishes.
We combine this with the narrow convergence of $\mu_\tau$ and the weak convergence $j_\tau\weakly j$ to infer, as $\tau\downarrow0$,
\[\int_{X\times Y}\vp(0)\,\dd\mu_0+\int_0^T\int_X\int_Y\big(\partial_t\varphi u+\nabla\varphi\cdot j\big)\,\dd\nu\dd x\dd t=0,\]
where we used once more the smoothness of $\vp$ to ensure in particular convergence of the discrete time derivative.

Finally, if $\supp\nu$ is finite, the strengthened assertions follow directly from Proposition~\ref{prop:energy-convergence}, which establishes the strong convergence properties of the interpolants to the limit $u$ in this case.

\subsubsection{Entropy-dissipation inequality and flux defect control}
\label{sss:EntDissIneqFluxDefect}

It remains to establish the entropy-dissipation inequality and the
flux-defect estimate required by Definition~\ref{def:diss_sol}. We first identify
a common limiting entropy for both interpolants, then use the energy
and entropy defects to control the discrepancy between the limiting
flux $j$ and the constitutive flux $j_u$.
 \begin{lemma}[Limit entropy]\label{l:H.lim}
	There exists $H\in BV_{\loc}([0,\infty))$ with $H(0) = \calH(\mu_0)$ such that, along a subsequence $\tau\downarrow0$, it holds for a.e.\ $t>0$ and for $t=0$
	\begin{align}
		\label{eq:entropy-limit}
		H(t) &= \lim_{\tau\downarrow0}\calH(\mu_\tau(t)) = \lim_{\tau\downarrow0}\calH(\wt\mu_\tau(t)).
	\end{align}	
\end{lemma}

\begin{proof}For $t=0$ the identity~\eqref{eq:entropy-limit} is immediate.
	In view of the entropy inequality \eqref{eq:entropy-inequality.discr}, we see that the increase in entropy from the $(k{-}1)$th to the $k$th step can be estimated by the non-convex contribution of the energy $\calE_0$ to the dissipation, i.e.,
\[
\calH(\mu^k_\tau) - \calH(\mu^{k-1}_\tau) \leq \tau\Big(
\vartheta\|\nabla U_\tau^k\|_{L^2(X)}^2 +
\beta^*\|\nabla u_\tau^k\|_{L^2(X\times Y)}^2\Big) \leq C_*\tau,
\]
where we used \eqref{eq:moment-bound_k-uniform} for the last inequality with a constant $C_*>0$.
    We define $G_\tau:[0,\infty) \to \bbR$ by
	\begin{align*}
		G_\tau(t) \coloneqq \calH(\mu_\tau(t)) - C_*k\tau,\quad\text{for }t\in((k{-}1)\tau, k\tau].
	\end{align*}
	   By definition $G_\tau$ is non-increasing and $\tau$-uniformly bounded so that by Helly's selection theorem there exists a non-increasing function $G:[0,\infty)\to\bbR$ such that, after passage to a subsequence, it holds that $ G(t)=\lim_{\tau\downarrow0}G_\tau(t)$ for a.e.\ $t>0$ and for $t=0$.
	\begin{align*}
		\calH(\mu_\tau(t)) \overset{\tau\downarrow 0}{\longrightarrow} H(t) \coloneqq G(t) + C_*t 
        \qquad\text{for a.e. $t>0$.}
	\end{align*}
	To establish the assertion, it now suffices to show that, as $\tau\downarrow0$,
	\begin{align}\label{eq:entrop.diff.conv}
		\calH(\mu_\tau)-\calH(\wt\mu_\tau)\longrightarrow 0\qquad\text{in $L^1(0,T)$}.
	\end{align}
	This essentially follows from the strong convergence of the difference of the densities combined with the uniform moment control. First, the logarithmic entropy function $h(s)=s\log s$ satisfies for $\theta\in(0,1)$ 
	\[|h(s)-h(r)|\lesssim_{\theta}|s-r|^\theta+|s-r|(s+r).\]
	Therefore, abbreviating $z_\tau\coloneqq u_\tau-\wt u_\tau$ and taking $\theta\in(\tfrac{d}{d+2},1)$, so that $\frac{1-\theta}{\theta}d<2$,
	and $q\in(\frac{1-\theta}{\theta}d,2)$,
	we obtain
	\begin{align*}
		\int_{X\times Y}|h(u_\tau)-h(\wt u_\tau)|\,\dd x\dd\nu(y)&\lesssim_\theta \int_{X\times Y}|z_\tau|^\theta\,\dd x\dd\nu(y)
		+ \|z_\tau\|_{L^2(X{\times}Y)}\big(\|u_\tau\|_{L^{2}(X{\times}Y)}+ \|\wt u_\tau\|_{L^{2}(X{\times}Y)}\big)
		\\&\lesssim  \bigg(\int_{X\times Y}(1+|x|^q)|z_\tau|\,\dd x\dd\nu(y)\bigg)^\theta + \|z_\tau\|_{L^2(X{\times}Y)},
	\end{align*}
	where the last step follows from H\"older's inequality applied to $(1+|x|^q)^\theta|z_\tau|^\theta$ and $(1+|x|^q)^{-\theta}$.
	The second term on the right-hand side converges to zero in $L^1(0,T)$  by~\eqref{eq:strongH1.diff.p}. For the first term on the right-hand side the strong $L^1(0,T)$-convergence follows by combining~\eqref{eq:strongH1.diff.p} with the uniform second moment control due to $q<2$.
\end{proof}

The asserted entropy-dissipation inequality~\eqref{eq:entr.diss} follows from the discrete version~\eqref{eq:entropy-inequality.discr} as follows.
Let us use the split of the double-well potential into convex and non-convex part (cf.\ \eqref{eq:DoubleWellSplit}) and write the
entropy-dissipation inequality as
\begin{multline*}
\calH(\mu_\tau(t))
+
\int_0^t\int_{X}\int_Y\alpha|\Delta u_\tau|^2\,\dd\nu\dd x\dd s
+\int_0^t\int_X \sff''_{\mathrm{cvx}}(U_\tau)|\nabla U_\tau|^2\,\dd x\dd s
\\+
\int_0^t\int_X\sff_\ccv''(U_\tau)|\nabla U_\tau|^2\dd x\dd s
+
\int_0^t\int_X(\sfB\nabla u_\tau,\nabla u_\tau)_{L^2(Y)}\,\dd x\dd s
\leq \calH(\mu_0).
\end{multline*}
With the strong convergence in \eqref{eq:interpolTestedStrongConvA} and the uniform bound in $L^2(0,T;H^2(X))$ from Lemma~\ref{l:disc_entropy-inequality}, spatial interpolation gives
\[
\|U_\tau-U\|_{L^2(0,T;H^1(X))}^2
\lesssim
\|U_\tau-U\|_{L^2(0,T;L^2(X))}
\|U_\tau-U\|_{L^2(0,T;H^2(X))}
\to 0,
\]
i.e., strong convergence of $\nabla U_\tau$ in $L^2([0,T]\times X)$. From this, we can conclude the convergence of the non-convex contribution 
$\int_0^t\int_X\sff_\ccv''(U_\tau)|\nabla U_\tau|^2\dd x\dd s= -\vartheta\|\nabla U_\tau\|_{L^2((0,t)\times X)}^2$. Moreover, we can argue similarly for the nonlocal part using the approximation $\sfB_n$ in Lemma \ref{l:ApproxRepul}.

The convex contributions are lower semicontinuous. 
Clearly, we have along a subsequence that
\[
\int_0^t\int_{X\times Y}\alpha|\Delta u|^2\,\dd\nu\dd x\dd s
\leq
\liminf_{\tau\to0}
\int_0^t\int_{X\times Y}\alpha|\Delta u_\tau|^2\,\dd\nu\dd x\dd s.
\]
Since $\sff''_\mathrm{cvx}$ is continuous and non-negative, Fatou's lemma, together with the strong gradient convergence from above, and a further almost-everywhere converging subsequence, gives
\[
\int_0^t\int_X \sff''_{\mathrm{cvx}}(U)|\nabla U|^2\dd x
\leq
\liminf_{\tau\to0}
\int_0^t\int_X \sff''_{\mathrm{cvx}}(U_\tau)|\nabla U_\tau|^2\dd x.
\]

Finally, Lemma \ref{l:H.lim} gives $\calH(u_\tau(t))\to H(t)$ for a.e.\ $t$. Passing to the limit yields
\[
H(t)+\int_0^t\calD_\calH(u)\,\dd s\leq \calH(\mu_0) +
\vartheta\|\nabla U\|_{L^2((0,t)\times X)}^2
-
\int_0^t\int_X(\sfB\nabla u,\nabla u)_{L^2(Y)}\,\dd x\dd s.
\]
This shows~\eqref{eq:entr.diss}, so that the entropy-dissipation inequality of Definition~\ref{def:diss_sol} is verified.

\paragraph{Flux defect control} 
Testing the Euler--Lagrange equations with functions of separated variables and localising in $y$, we first obtain the corresponding spatial identity for a.e.\ $(t,y)$. Since $j_\tau\in L^1$ and the coefficients in the variational formula on the right-hand side in~\eqref{eq:flux.EL} are integrable, approximation in $x$ then yields the Euler--Lagrange equations for every $\xi\in L^\infty([0,T]{\times} Y;C_b^2(X;\mathbb R^d))$.
Let $j_u$ denote the flux associated with $u$ as in \eqref{eq:flux.distrib}.
To derive the estimate for the error term $j-j_u$, %
we subtract from the Euler--Lagrange equations in Proposition~\ref{prop:E-L_2} the equation for $j_u$ in \eqref{eq:flux.distrib}, which yields, for every $\xi\in L^\infty([0,T]{\times} Y;C_b^2(X;\mathbb R^d))$, $\scrL^1$-a.e.\ in $[0,T]$, the identity
	\begin{align*}
	\la j_\tau - j_u,\xi\ra
	&= \int_X\int_Y \alpha(y) \pra[\big]{A_\xi:(\nablax u_\tau{\otimes} \nablax u_\tau -\nablax u{\otimes}\nablax u) + b_\xi\cdot (u_\tau\nablax u_\tau - u\nablax u) } \,\dd \nu\dd x\\
    &\quad+\int_X\int_Y\kappa(y)\Big[\big(\sff'(U_\tau)\nabla u_\tau - \sff'(U)\nabla u\big)\cdot \xi
    +\big(\sff'(U_\tau) u_\tau - \sff'(U) u\big) \divv\xi\Big]\,\dd\nu\dd x\\
    &\quad+ \int_X\int_Y \Big[\big(\sfB [u_\tau]\nabla u_\tau - \sfB [u]\nabla u\big)\cdot \xi
    +\big(\sfB [u_\tau] u_\tau - \sfB [u] u\big) \divv\xi\Big]\dd\nu\dd x,
\end{align*}
where we abbreviated $A_\xi \coloneqq \frac{1}{2}(\divv_x\xi) I + \nablax\xi$ and $b_\xi \coloneqq \nablax\divv_x\xi$.

Define, for non-negative $v,w\in L^1(Y;H^1(X))$ satisfying
$\int_Xv(\cdot,y)\,\dd x=\int_Xw(\cdot,y)\,\dd x=1$ for $\nu$-a.e.\ $y$, the relative entropy and entropy functional
\begin{align}
    \label{eq:defRelEnergy}
	\mathcal{E}_{\rel}(v|w)
	&\coloneqq \calE(v) -\calE(w) - \rmD\calE(w)[v-w],\\
    \label{eq:midEntropy}
	\mathcal{H}_{\mathrm{mid}}(v,w)
	&\coloneqq \int_Y\int_X
	\bigg[
	\frac12h(v)+\frac12h(w)
	-h\bigg(\frac{v+w}{2}\bigg)
	\bigg]\,\dd x\dd\nu,
	\qquad h(r)\coloneqq r\log r.
\end{align}
We compute that the relative energy is given for $z= v{-}w$, $V(x)= 
\int_Y \kappa(y) v(x,y)\dd\nu$, $W(x)= 
\int_Y \kappa(y) w(x,y)\dd\nu$, and $Z = V{-}W$ by
\begin{equation*}
\begin{split}
\mathcal{E}_{\rel}(v|w)
&= \int_{X}\int_Y 
\frac{\alpha(y)}{2}|\nabla z|^2 \,\dd\nu\dd x 
+
\int_X\big({\sf f}(V)-{\sf f}(W)-{\sf f}'(W)Z\big)\,\dd x
+\frac{1}{2}\int_{X} (\sfB z,z)_{L^2(Y)}\dd x
\\
&\geq \frac{\alpha_*}{2}\|\nabla z\|_{L^2(X\times Y;\bbR^d)}^2 +\frac{\vartheta}{12}\|Z\|_{L^4(X)}^4 - \Big(\frac{\vartheta}{8}+\frac{\beta^*}{2}\Big)\|z\|_{L^2(X\times Y)}^2
\end{split}
\end{equation*}
where we used \eqref{eq:pot-rel.coerc} and \eqref{eq:boundBetaApprox}. In particular, for any $\lambda>\vartheta/4+\beta^*>0$, we have the coercivity
estimate
\begin{equation}
\label{eq:relEnergyCoerc}
\mathcal{E}_{\rel}(v|w) +\frac{\lambda}{2}\|z\|_{L^2(X\times Y)} ^2\gtrsim \|z\|_{L^2(Y;H^1(X))}^2 +\vartheta\|Z\|_{L^4(X)}^4.
\end{equation}
Moreover, observe that \[
\mathcal{H}_{\mathrm{mid}}(v,w)=\frac12\mathcal{H}_{\rel}(v|m)+\frac12\mathcal{H}_{\rel}(w|m),\quad\text{where } m\coloneqq (v+w)/2,
\] where $\mathcal{H}_{\rel}(a|m)\coloneqq \int_Y\int_X \big(a\log\frac{a}{m}-a+m\big)\,\dd x\dd \nu$. The functional $\calH_\mathrm{mid}$ is also called Jensen–Shannon divergence. Applying fibrewise  the Csiszár--Kullback--Pinsker inequality, we obtain
\begin{align}\label{eq:KP}
	\int_Y\|v(\cdot,y)-w(\cdot,y)\|_{L^1(X)}^2\,\dd\nu(y)\lesssim \mathcal{H}_{\mathrm{mid}}(v,w).
\end{align}

\begin{proposition}[Flux defect control]
\label{pr:defectFlux}
For $\xi\in L^\infty([0,T]{\times}Y;C_b^2(X;\R^d))$, it holds that
\[
\limsup_{\tau \to 0}\bigg|\int_0^T\langle j_\tau - j_u,\xi \rangle \,\dd t\bigg| \lesssim  \|\xi\|_{L^\infty((0,T)\times Y;C_b^2(X))}\int_0^T\bigl(E_\dft(t)+H_\dft(t)\bigr)\,\dd t
\]
\end{proposition}

\begin{proof}
We set $z_\tau = u_\tau- u$ and rewrite the right-hand side to infer  
\[
\int_0^T\int_X\int_Y \xi\cdot(j_\tau - j_u)\,\dd\nu\dd x\dd t=\int_0^T\mathcal{Q}_\tau(t)\,\dd t
+  \int_0^T\calA_\tau(t)\,\dd t
+\int_0^T\mathrm{err}_\tau(t)\,\dd t,
\] 
with 
\begin{align*}
\mathcal{Q}(z)&\coloneqq\;\int_X\int_Y \Big[ \alpha(y) \big(A_\xi:\nablax z{\otimes}\nablax z + z  b_\xi\cdot\nablax z  \big) + \sfB [z]z \divv\xi + 
\sfB [z]\nabla z \cdot\xi \Big]\,\dd \nu(y) \dd x,
\\
\calA_\tau(t) &\coloneqq\int_X\int_Y\kappa(y) \Big[z_\tau\big(\sff'(U_\tau) {-}\sff'(U)\big)+u\big(\sff'(U_\tau){-}\sff'(U){-}\sff''(U)Z_\tau)\big)\Big]\divv\xi\,\dd\nu\dd x
\\
&\qquad + \int_X\int_Y \kappa(y)\Big[\big(\sff'(U_\tau) {-}\sff'(U)\big)\nabla z_\tau+\big(\sff'(U_\tau){-}\sff'(U){-}\sff''(U)Z_\tau)\big)\nabla u\Big]\cdot \xi \dd \nu\dd x,
\\
\mathrm{err}_\tau(t) &\coloneqq\;\int_X\int_Y \alpha(y) \pra[\big]{A_\xi:(
\nabla u{\otimes} \nabla z_\tau + \nabla z_\tau {\otimes} \nabla u
) 
    + b_\xi\cdot (z_\tau \nabla u + u \nabla z_\tau)\,\dd \nu(y) \dd x \\&\qquad + \int_X\int_Y\divv \xi\big(\sfB [z_\tau] u + \sfB[u]z_\tau\big)  + \xi \cdot \big( \sfB[z_\tau]\nabla u + \sfB[u]\nabla z_\tau\big) } \,\dd \nu(y) \dd x
	\\&\qquad+ \int_X\int_Y\kappa(y)\bigl[z_\tau\sff'(U)+u\sff''(U)Z_\tau\bigr]\divv \xi\dd \nu\dd x
    \\&\qquad+ \int_X\int_Y\kappa(y)\bigl[\sff'(U)\nabla z_\tau+\sff''(U)(U_\tau{-}U)\nabla u\bigr]\cdot \xi\dd \nu\dd x.
\end{align*}
Using the convergence properties in Proposition~\ref{prop:energy-convergence}, we conclude that $\int_0^T\mathrm{err}_\tau(t)\,\dd t\to0$ as $\tau\downarrow0$. 

Moreover, since $\alpha\leq\alpha^\ast =\|\alpha\|_{L^\infty(Y)}$, 
we obtain with \eqref{eq:relEnergyCoerc}
\begin{align*}
	\bigg|\int_0^T\mathcal{Q}_\tau(t)\,\dd t\bigg|
	\lesssim_{\alpha^\ast, \beta^*}
	\|\xi\|_{L^\infty((0,T)\times Y;C_b^2(X))}
	\int_0^T\big(
	\mathcal{E}_{\rel}(u_\tau|u)
	+\lambda\|u_\tau{-}u\|_{L^2(X\times Y)}^2
	\big)\,\dd t.
\end{align*}
To show the analogous upper bound for $\calA_\tau$, we use that
\[
\begin{split}
|\sff'(U_\tau) - \sff'(U)|
&\lesssim_\vartheta (1+|U_\tau|^2+|U|^2)|Z_\tau|,\\
|\sff'(U_\tau){-}\sff'(U){-}\sff''(U)Z_\tau|
&\lesssim_\vartheta (1+|U|+|U_\tau|)|Z_\tau|^2.
\end{split}
\]
Thus, we conclude with Hölder's inequality
\begin{align*}
\|\sff'(U_\tau) - \sff'(U)\|_{L^2(X)}
&\lesssim_\vartheta
\|Z_\tau\|_{L^2(X)}+\big(\|U\|_{L^6(X)}^2+\|U_\tau\|_{L^6(X)}^2\big)\|Z_\tau\|_{L^6(X)},\\
\|\sff'(U_\tau){-}\sff'(U){-}\sff''(U)Z_\tau\|_{L^2(X)}
&\lesssim_\vartheta
\|Z_\tau\|_{L^4(X)}^2+(\|U\|_{L^6(X)}+\|U_\tau\|_{L^6(X)})\|Z_\tau\|_{L^6(X)}^2.
\end{align*}
Thus, since $d\le 3$, a Sobolev embedding and the coercivity estimate for $\calE_\mathrm{rel}$ in \eqref{eq:relEnergyCoerc} yield
\[
\bigg|\int_0^T\calA_\tau(t)\dd t\bigg|
\lesssim_{}\|\xi\|_{L^\infty((0,T)\times Y;C_b^1(X))}
	\int_0^T\big(
	\mathcal{E}_{\rel}(u_\tau|u)
	+\lambda\|u_\tau{-}u\|_{L^2(X\times Y)}^2
	\big)\,\dd t.
\]

By the Gagliardo--Nirenberg inequality and~\eqref{eq:KP}, combined with Young's inequality,
\begin{align}\label{eq:L^2-bound_by_Erel+Hrel}
	\int_0^T\|u_\tau-u\|_{L^2(X\times Y)}^2\,\dd t
	&\lesssim
	\bigg(\int_0^T\|\nabla(u_\tau{-}u)\|_{L^2(X\times Y)}^2\,\dd t
	\bigg)^{\frac{d}{d+2}}
	\bigg(\int_0^T
	\mathcal{H}_{\mathrm{mid}}(u_\tau,u)\,\dd t
	\bigg)^{\frac{2}{d+2}}
	\nonumber\\
	&\lesssim_{\alpha_\ast}
	\int_0^T\mathcal{E}_{\rel}(u_\tau|u)\,\dd t
	+\int_0^T\mathcal{H}_{\mathrm{mid}}(u_\tau,u)\,\dd t.
\end{align}

Moreover, Proposition~\ref{prop:energy-convergence}, the convergence of
the lower-order energy, and Lemma~\ref{l:H.lim} imply
\begin{align*}
	\limsup_{\tau\downarrow0}
	\int_0^T\mathcal{E}_{\rel}(u_\tau|u)\,\dd t
	&\leq\int_0^TE_\dft(t)\,\dd t,\\
	\limsup_{\tau\downarrow0}
	\int_0^T\mathcal{H}_{\mathrm{mid}}(u_\tau,u)\,\dd t
	&\leq\frac12\int_0^TH_\dft(t)\,\dd t.
\end{align*}
Indeed, $(u_\tau+u)/2\rightharpoonup u$, and hence the second inequality
follows from the lower semicontinuity of $\mathcal H$. We therefore
conclude that
\begin{align*}
	\limsup_{\tau\downarrow0}\bigg|\int_0^T\la j_\tau{-}j_u,\xi\ra\,\dd t\bigg|	\lesssim_{\alpha_\ast,\alpha^\ast,\beta^*,\vartheta}
	\|\xi\|_{L^\infty((0,T)\times Y;C_b^2(X))}
	\int_0^T\bigl(E_\dft(t)+H_\dft(t)\bigr)\,\dd t.
\end{align*}
Thus, we arrive at \eqref{eq:mv-flux.control}.
\end{proof}

The flux-defect estimate completes the proof of Theorem~\ref{thm:existence}.
The associated weak--strong uniqueness principle for dissipative solutions (cf.~Theorem~\ref{thm:wk-str.uniq}) will be proved in Section~\ref{ssec:wsu.diag.diss} and is valid in every space dimension.

\begin{remark}
We emphasize that in the case that $\supp\nu$ is a singleton (see Remark~\ref{rem:Singleton})
we can drop the dimensional restriction $d\leq 3$. In this case, we can establish strong convergence
of $U_\tau$ in $L^4$, see Remark~\ref{rem:betterConvergenceSingl}. Thus, the lower-order contribution 
in the constitutive flux converges in $L^1$ and together with strong $L^2_t(H^1_x)$
  convergence, every term in the distributional flux formula converges, so that $j=j_u$ and the flux defect vanishes as in the case where $\supp\nu$ is finite.
\end{remark}

\section{Non-convex domains and intrinsic transport}\label{sec:non-convex}

The main purpose of this section is to prove Theorem~\ref{thm:ex.Ediss.nonconvex}. Throughout, we let $X=\ol\Om$, where $\Om\subset\R^d$ is a bounded smooth domain, equipped with its intrinsic length distance $d_X$. Then $(X,d_X)$ is a compact geodesic space whose topology agrees with the Euclidean topology on $X$. Our interest is in variational approximations based on 
 optimal transport with respect to the intrinsic cost $c(x,x'):=\frac12d_X^2(x,x')$, inducing the optimal transport cost  $\frac12\mathcal{W}_{2,X}^2$.

In contrast to the whole-space setting considered in Section~\ref{sec:diagE}, where we relied on representing optimal couplings by optimal transport maps, we here use Kantorovich potentials and optimal dynamic plans to handle non-convex domains. A property that will be important in the analysis is that Kantorovich potentials associated to an absolutely continuous reference measure $\mu\ll\scrL^d\otimes \nu$ are differentiable in $x$  $\mu$-a.e., with an essentially bounded spatial gradient.
Concerning the fibred structure, the present approach also provides a relatively simple way to establish the required joint measurability in the spatial and fibre variables of the objects used to define the Euler--Lagrange flux.

\subsection{Measurable selections}\label{ssec:select.meas}

Since the framework for fibred optimal transport developed in~\cite{peszek2023heterogeneous} is presented primarily for the space $\R^d$ with Euclidean cost, we first establish the measurability properties needed here for intrinsic transport on $X=\ol\Om$. Specifically, we construct measurable families of Kantorovich potentials and optimal dynamic plans by suitable applications of the standard measurable selection results summarised in~\cite[Section~6.9]{Bogachev_2007}, where we adopt the relevant terminology from that reference. 
Some arguments here even simplify due to the compactness of $X=\ol\Om$. 
The measurable selection results, formulated for a general Polish space $\mathscr{Y}$, will later be invoked with $\mathscr{Y}=Y$ and $\mathscr{Y}=I\times Y$ for a time interval $I\subset\mathbb{R}_{\ge0}$.

Concerning basic aspects of the dual formulation of the optimal transport problem and the properties of $c$-concave functions that will be used in the following, we refer to~\cite[Chapter 1]{Santambrogio_2015}, while for optimal dynamic plans, we refer to~\cite[Section~9.3]{ABS_2024}.
Throughout the following, $\calP(X)$ will be endowed with the narrow topology. Then, $\calP(X)$ is a Polish space, which is compact due to the compactness of $X$.

\begin{lemma}[Measurable Kantorovich potentials]\label{l:Kant-pot.meas}
Let $\mathscr{Y}$ be a Polish space and let $\mathscr{Y}\ni\PARAM\mapsto\mu_\PARAM,\widehat\mu_\PARAM\in\calP(X)$ be Borel. Fix $x_*\in X$. 
There exists a Borel map $\PARAM\mapsto\vp_\PARAM\in C(X)$ such that $\vp_\PARAM$ is a Kantorovich potential from $\mu_\PARAM$ to $\widehat\mu_\PARAM$ for the cost $c$, such that for all $\sfy\in\scrY$
\[\vp_\PARAM=\vp_\PARAM^{cc},\qquad \vp_\PARAM(x_*)=0,\qquad\operatorname{Lip}_{d_X}(\vp_\PARAM)\le L_X,\qquad\|\vp_\PARAM\|_{C(X)}\le L_X^2,\]
where $\vp^c(z):=\inf_{x\in X}(c(x,z)-\vp(x))$ and $L_X\coloneqq\mathrm{diam}(X)$. In particular, $(\PARAM,x)\mapsto\vp_\PARAM(x)$ is jointly Borel. Moreover, there exists a jointly Borel map $G:\mathscr Y\times X\to\R^d$ such that, for every $\PARAM\in \mathscr Y$, 
\begin{align}\label{eq:nablaKant-pot.meas}
	G(\PARAM,\cdot)=\nabla_x\vp_\PARAM\quad\scrL^d\text{-a.e.\ in }\Omega,\qquad |G(\PARAM,x)|\le L_X\quad\text{for all }x\in X.
\end{align}
\end{lemma}

\begin{proof}
We first construct a Borel selection on the compact metric space 
$\mathscr A:=\mathcal{P}(X)\times\mathcal{P}(X)$.
Consider 
\[	\mathscr K\coloneqq\{\vp\in C(X):\vp=\vp^{cc},\ \vp(x_*)=0\}\subset C(X)\]
equipped with the subspace topology of uniform convergence. As we will detail in the following, by restricting to that subset, we obtain a compact set of admissible potentials without changing the dual supremum.
By~\cite[Proposition~1.34]{Santambrogio_2015}, the condition $\vp=\vp^{cc}$ is equivalent to $c$-concavity of $\vp$, which, due to the bound $|c(x,z)-c(x',z)|\le L_Xd_X(x,x')$, implies  that $\vp$ is $L_X$-Lipschitz continuous with respect to $d_X$. Hence, every $\vp\in\mathscr K$ satisfies $\operatorname{Lip}_{d_X}(\vp)\le L_X$, and the property $\vp(x_*)=0$ further yields $\|\vp\|_{C(X)}\le L_X^2$.
Since the $c$-transform $\vp\mapsto\vp^c$ is continuous on $C(X)$ (due to $\|\vp^c-\psi^c\|_{C(X)}\le\|\vp-\psi\|_{C(X)}$), it follows that $\mathscr K$ is closed in $C(X)$. The Arzel\`a--Ascoli theorem then implies that $\mathscr K$ is compact and thus Polish.

We now consider the functional 
\[J: \mathscr A\times\mathscr K\to\mathbb{R}, \qquad J(\mu,\widehat\mu,\vp):=\int_X\vp\,\dd\mu+\int_X\vp^c\,\dd\widehat\mu.\]
Using once more the continuity of the $c$-transform on $C(X)$, it follows that $J$ is continuous. Further observe that imposing the condition~$\vp(x_*)=0$ on admissible $\vp$ does not decrease the supremum due to the invariance of the value $\int_X\vp\,\dd\mu+\int_X\vp^c\,\dd\widehat\mu$ under $(\vp,\vp^c)\mapsto(\vp+a,\vp^c-a)$.
Taking also into account the fact that $\vp^{cc}\ge\vp$ and $(\vp^{cc})^c=\vp^c$, which means that the restriction $\vp=\vp^{cc}$ to $c$-concave functions in $\mathscr K$ does not change the value, we can invoke Kantorovich duality to deduce
\[\max_{\vp\in\mathscr K}J(\mu,\widehat\mu,\vp)=\frac12\calW_{2,X}^2(\mu,\widehat\mu).\]
It follows that the set
\[\Gamma_{\rm pot}:=\bigg\{(\mu,\widehat\mu,\vp)\in\mathscr A\times\mathscr K:	J(\mu,\widehat\mu,\vp)=\frac12\calW_{2,X}^2(\mu,\widehat\mu)\bigg\}\]
has non-empty sections $(\Gamma_{\rm pot})_{(\mu,\widehat\mu)}:=\{\vp\in\mathscr K:(\mu,\wh\mu,\vp)\in\Gamma_{\rm pot}\}$ over every $(\mu,\widehat\mu)\in\mathscr A$.
 Since $X$ is compact, $\calW_{2,X}$ is continuous on $\mathscr A$, so that $\Gamma_{\rm pot}\subset\mathscr A\times\mathscr K$ is closed. In particular, $\Gamma_{\rm pot}\in\mathcal{B}(\mathscr A\times\mathscr K)$ is a Borel set. Moreover, each section $(\Gamma_{\rm pot})_{(\mu,\widehat\mu)}$  is a closed subset of the compact space $\mathscr K$, and hence compact.
We may therefore invoke~\cite[Vol.~II, Theorem~6.9.6]{Bogachev_2007}, which provides us with a Borel map %
$\Phi:\mathscr A\to\mathscr K$ such that $(\mu,\widehat\mu,\Phi(\mu,\widehat\mu))\in\Gamma_{\rm pot}$
for every $(\mu,\widehat\mu)\in\mathscr A$.
Finally, since $\mathscr Y\ni\PARAM\mapsto(\mu_\PARAM,\widehat\mu_\PARAM)$ is Borel, the composition $\PARAM\mapsto\vp_\PARAM:=\Phi(\mu_\PARAM,\widehat\mu_\PARAM)$ is a Borel family of Kantorovich potentials.
 Since the evaluation map $C(X)\times X\ni(\wt\vp,x)\mapsto\wt\vp(x)$ is continuous, the Borel measurability of $\PARAM\mapsto\vp_\PARAM\in C(X)$
implies that $(\PARAM,x)\mapsto\vp_\PARAM(x)$ is jointly Borel.
	
	To construct $G$, observe that each $\vp_\PARAM$ is locally Euclidean $L_X$-Lipschitz in $\Omega = \mathrm{int}(X)$.
For $i=1,\ldots,d$ and $n\in \mathbb{N}$, we therefore consider the jointly Borel measurable difference quotients
	\begin{align*}Q_{n,i}(\PARAM,x):=
	\begin{cases}
		n\big(\vp_\PARAM(x+n^{-1}e_i)-\vp_\PARAM(x)\big),
		&x\in\Omega,\ \operatorname{dist}(x,\partial\Omega)>n^{-1},\\
		0,&\text{otherwise}.
	\end{cases}\end{align*}
Let $Q_n=(Q_{n,1},\ldots,Q_{n,d})$ and define the map
\begin{align*}
G(\PARAM,x):=\begin{cases}\lim_{n\to\infty}Q_n(\PARAM,x)&\text{ if }\lim_{n\to\infty}Q_n\text{ exists in } \overline{B_{L_X}}(0),\\0&\text{ otherwise},
	\end{cases}
\end{align*}
which is jointly Borel by construction. For every $\PARAM\in \mathscr Y$, Rademacher's theorem implies that $Q_n(\PARAM,x)\to\nabla_x\vp_\PARAM(x)$ and $|\nabla_x\vp_\PARAM(x)|\le L_X$ for almost every $x\in\Omega$. 

This proves the assertions.
\end{proof}

Recall that we consider the Banach space $C([0,1];X)$ equipped with the uniform topology induced by $d_X$. 
The subspace $\operatorname{Geo}(X)\subset C([0,1];X)$ of constant-speed length-minimising geodesics is closed.
Furthermore, since the family of curves $\operatorname{Geo}(X)$ is equicontinuous and $X$ is compact, $\GEO(X)\Subset C([0,1];X)$ is compact due to the Arzel\`a--Ascoli theorem.
Crucially relying on the compactness of the geodesic Polish space $X=\ol\Om$, the following result is not difficult to infer from classical measurable choice theory~\cite[Chapter 6]{Bogachev_2007}.

\begin{lemma}[Measurable optimal dynamic plans]\label{l:dyn.plan.meas}
Under the assumptions of Lemma~\ref{l:Kant-pot.meas}, there exists a Borel map $\PARAM\mapsto\wt\pi_\PARAM\in\calP(\GEO(X))$ such that	
\begin{align}\label{eq:dyn.plan.Borel}
(e_0)_\#\wt\pi_\PARAM=\mu_\PARAM,\qquad(e_1)_\#\wt\pi_\PARAM=\widehat\mu_\PARAM,\qquad\int_{\operatorname{Geo}(X)}d_X^2(\gamma_0,\gamma_1)\,\dd\wt\pi_\PARAM
	=\calW_{2,X}^2(\mu_\PARAM,\widehat\mu_\PARAM).
\end{align}
Moreover, there exists a Borel map $v:[0,1]\times\operatorname{Geo}(X)\to\R^d$ such that, for every $\gamma\in\operatorname{Geo}(X)$,
\[v(r,\gamma)=\dot\gamma(r)\quad\text{for a.e.\ }r\in(0,1),\qquad
	|v(r,\gamma)|\le d_X(\gamma_0,\gamma_1)\quad\text{for all }r\in[0,1].\]
If $\gamma_0\in\Omega$, then $v(0,\gamma)=\dot\gamma(0+)$. 
\end{lemma}

\begin{proof}[Proof sketch] 
Concerning the first assertion, we proceed by reduction to the problem of a Borel selection of static optimal plans, whose existence is guaranteed by~\cite[Corollary 5.22]{Villani_2009}.
For the reduction step, it suffices to show the existence of a Borel map $\sigma:X\times X\to\GEO(X)$ such that, for all $x,y\in X$, $\sigma(x,y)$ is a constant-speed length minimising geodesic from $x$ to $y$, since then the push-forward $\sigma_{\#}$ maps every Borel family of optimal static plans to a Borel family of optimal dynamic plans.
For existence of a map $\sigma$ with the required properties, %
one considers the endpoint map $E:=(e_0,e_1):\GEO(X)\to X\times X$, which is continuous and, since $X$ is geodesic, also surjective. Then, thanks to the compactness of $\GEO(X)$, \cite[Theorem~6.9.7]{Bogachev_2007} applies and yields a Borel set $B\subset\GEO(X)$ such that $E|_B:B\to X\times X$ is bijective and has a Borel inverse $\sigma$.

For the assertions on the velocity, one approximates using forward difference quotients. %
The assertion on the initial velocity of geodesics starting in the interior of $\Omega$ follows from the fact that such geodesics coincide with a line segment for a short time.
\end{proof}

The measurable selection results above will be applied to Borel representatives of the disintegrations $\{\mu^y\}$, which we understand as defined on all of $Y$ after extension via a fixed probability measure on a Borel $\nu$-null set.  For the variational interpolant, we will also consider $\{\mu^{t,y}\}$.

\subsection{Differentiability of the intrinsic Wasserstein distance}\label{ss:DiffIntrinsicWasserstein}

The goal of this section is to derive a first variation formula for the squared intrinsic Wasserstein distance.
To this end, we first derive an alternative representation of the intrinsic Wasserstein cost in terms of the gradient of a Kantorovich potential, which is classical in the case of a convex domain.
We directly formulate the results in the fibred setting.

\begin{lemma}[Initial velocity of optimal geodesics]\label{l:ini.vel.geo}
Let $\mu,\widehat\mu\in\calP_\nu(\ol\Om{\times}Y)$ with $\mu\ll\scrL^d\otimes\nu$. Let $y\mapsto\varphi(\cdot,y)$ and $y\mapsto\liftpi^y$ be Borel families of Kantorovich potentials and optimal dynamic plans from $\mu^y$ to $\widehat\mu^y$, respectively, as provided
by Lemmas~\ref{l:Kant-pot.meas},~\ref{l:dyn.plan.meas}. 
Define $\wt\pi=\int_Y\wt\pi^y\otimes\delta_y\,\dd\nu(y)\in \mathcal{P}(\GEO(\ol\Om)\times Y)$ and let $\nabla_x\vp$ denote a jointly Borel measurable representative of the spatial gradient of the fibred Kantorovich potential (cf.~\eqref{eq:nablaKant-pot.meas}).
Then,
\[-\dot\gamma(0+)=\nabla_x\varphi(\gamma_0,y)\qquad\text{for }\liftpi\text{-a.e.\ }(\gamma,y).\]
Consequently,
\begin{align}\label{eq:W2.KantorovichPot}
\calW_{2,X,\nu}^2(\mu,\widehat\mu)&=\int_{\GEO(\ol\Om)\times Y}d_X^2(\gamma_0,\gamma_1)\,\dd\liftpi(\gamma,y)
=\int_{\Om\times Y}|\nabla_x\varphi(x,y)|^2\,\dd\mu(x,y).
\end{align}
\end{lemma}
\begin{proof}
Since $(e_0,\operatorname{id}_Y)_\#\liftpi=\mu\ll\scrL^d\otimes\nu$, it holds for $\liftpi$-a.e.\	$(\gamma,y)\in \GEO(\ol\Om)\times Y$ that $\gamma_0\in\Om$ and
$\vp(\cdot,y)$ is differentiable at $\gamma_0$. For such a geodesic, an initial segment is contained
in $\Om$ and hence is a Euclidean line segment. Then it is classical  (cf.\ e.g.\ the argument in~\cite[Proposition~1.30]{AG_2013}) that
	\[-\dot\gamma(0+)\in\pa_x^+\vp(\gamma_0,y).\]
Differentiability at $\gamma_0$ yields the identity. Moreover, since $\gamma$ has constant speed and is a Euclidean line segment near its initial point $\gamma_0$, it holds that $|\dot\gamma(0+)|=d_X(\gamma_0,\gamma_1)$ for $\liftpi$-a.e.\ 	$(\gamma,y)$. Using the $y$-integrated version of~\eqref{eq:dyn.plan.Borel} for fibrewise optimal dynamic plans and 
the fact that $(e_0,\operatorname{id}_Y)_\#\liftpi=\mu$, we then deduce~\eqref{eq:W2.KantorovichPot}, where we also used that $\mu\ll\scrL^d\otimes\nu$ to replace the domain of integration $\ol\Om\times Y$ by $\Om\times Y$.
\end{proof}
Now we turn to the differentiability of $\mathcal{W}_{2,X,\nu}^2(\cdot,\wh\mu)$ along inner variations with respect to $\ol\Om$. 
To this end, recalling $C_{\mathrm{tan}}^2(\ol\Omega;\mathbb R^d)= \left\{
\xi\in C^2(\ol\Omega;\mathbb R^d):\xi\cdot n=0\ \text{on }\partial\Omega\right\}$,
let $\xi\in L^\infty(Y;C^2_\TAN(\ol\Om;\mathbb{R}^{d}))$ be a fixed but arbitrary tangential vector field.
After modifying $\xi$ on a $\nu$-null set,  we may assume that $y\mapsto\xi(\cdot,y)$ is Borel with uniformly bounded $C^2(\ol\Om)$ norm,
and consider the associated flow map $\Phi_s:\ol\Om{\times}Y\to \ol\Om$, $\Phi_0(\cdot,y)=\operatorname{id}_{\ol\Om}$, $\pa_s\Phi_s=\xi(\cdot,y)\circ\Phi_s$.
Then, in particular, $(s,x,y)\mapsto\Phi_s(x,y)$ is jointly Borel, while  $\Phi_s(\cdot,y)$ is a $C^2$-diffeomorphism of $\ol\Omega$ onto itself.

\begin{proposition}[First variation of the intrinsic Wasserstein distance]\label{prop:var.intrinsicW2} Let $\mu,\wh\mu\in \calP_\nu(\ol\Om{\times}Y)$, $\mu\ll\scrL^{d}\otimes\nu$, and set
	$\mu_s\coloneqq (\Phi_s,\id_Y)_{\#}\mu$. Then, $s\mapsto \mathcal{W}_{2,X,\nu}^2(\mu_s,\wh\mu)$ is differentiable at $s=0$ with 
	\begin{align}\label{eq:differentiationW2.formula}
		\frac12\frac{\dd}{\dd s}\bigg|_{s=0}\mathcal{W}_{2,X,\nu}^2(\mu_s,\wh\mu)=\int_{\Om\times Y}\nabla_x\vp\cdot\xi\,\dd\mu,
	\end{align}
	where $y\mapsto\vp(\cdot,y)$ is the Borel family of Kantorovich potentials obtained in Lemma~\ref{l:Kant-pot.meas}.
\end{proposition}
\begin{proof} For the Euclidean cost (and $\supp\nu$ a singleton), the corresponding result can be found, e.g., in~\cite[Theorem~5.24]{Santambrogio_2015}. To show the assertion for the intrinsic distance, we here follow more closely the classical approach of~\cite[Chapter~8]{AGS}, see in particular Theorem~8.4.7 therein.
Let $g(s)\coloneqq\frac12\calW_{2,X,\nu}^2(\mu_s,\widehat\mu)$, and write $\vp_y\coloneqq\vp(\cdot,y)$.
Then $(\vp_y,\vp^c_y)$ is a Borel family of $c$-conjugate optimal pairs.
From Lemma~\ref{l:Kant-pot.meas}, we know that, for all $y\in Y$, %
 $|\nabla_x\vp_y|\le L_X$ $\scrL^d$-a.e.\ in $\Om$.
	
By construction, $g(s)\ge 	\int_{\Om{\times}Y}\vp\,\dd(\Phi_s,\id_Y)_{\#}\mu+	\int_{\ol\Om{\times}Y}\vp^c\,\dd\wh\mu$ with equality for $s=0$. 
Letting further $\liftpi$ be a fibrewise optimal dynamic plan from $\mu$ to $\wh\mu$, we hence infer
\begin{align}
		\int_{\ol\Om{\times}Y}\bigl(\vp_y\circ\Phi_s(\cdot,y)-\vp_y\bigr)\,\dd\mu\le g(s)-g(0)\le\int_{\operatorname{Geo}(\ol\Om){\times}Y}
		\bigl(c(\Phi_s(\gamma_0,y),\gamma_1)-c(\gamma_0,\gamma_1)\bigr)\,\dd\liftpi(\gamma,y),
	\label{eq:cost.intr.squeeze}
	\end{align}
where the second step uses primal optimality of $\wt\pi$ with competitor $((\Phi_s\circ(e_0,\operatorname{id}_Y),\id_Y),(e_1,\operatorname{id}_Y))_\#\liftpi\in\Gamma_\nu(\mu_s,\wh\mu)$.
We then divide~\eqref{eq:cost.intr.squeeze} by $s\gtrless0$ and aim to show that the resulting lower and upper bounds have  well-defined  limits as $s\to0^\pm$ which both coincide with the right-hand side of~\eqref{eq:differentiationW2.formula}. To this end, first note that, since $\mu^y\ll\scrL^d,$
	\[\frac{\dd}{\dd s}\bigg|_{s=0}\vp_y\circ\Phi_s(\cdot,y)=\nabla\vp_y\cdot\xi(\cdot,y)\qquad\text{$\mu^y$-a.e.\ in }\Om.\]
	For the term involving the cost function, we assert that, as $s\to0^\pm$,
	\begin{align}\label{eq:cost.expansion}
		c(\gamma^s,\gamma_1)-c(\gamma_0,\gamma_1)=%
		\nabla_x\vp_y(\gamma_0)\cdot(\gamma^s-\gamma_0)+o(|\gamma^s-\gamma_0|)\qquad\liftpi\text{-a.e.\ }(\gamma,y),
	\end{align}
	where $\gamma^s=\Phi_s(\gamma_0,y)$.\\[1mm]
{\em Re $\ge$}:  from $c(x,x')\ge\vp_y(x)+\vp^c_y(x')$ and	$\vp_y(\gamma_0)+\vp^c_y(\gamma_1)=c(\gamma_0,\gamma _1)$ for $\liftpi$-a.e.\ $(\gamma,y)$ we infer
	\[ c(\gamma^s,\gamma_1)-c(\gamma_0,\gamma_1)\ge\vp_y(\gamma^s)-\vp_y(\gamma_0)
=\nabla_x\vp_y(\gamma_0,y)\cdot(\gamma^s-\gamma_0)+o(|\gamma^s-\gamma_0|)\qquad\liftpi\text{-a.e.\ }(\gamma,y).\]
{\em Re $\le$}:  by Lemma~\ref{l:ini.vel.geo}, it suffices to show that	\[c(\gamma^s,\gamma_1)-c(\gamma_0,\gamma_1)\le v\cdot(\gamma^s-\gamma_0)+o(|\gamma^s-\gamma_0|),\qquad v\coloneqq -\dot\gamma(0+).\]
Let $\ell\coloneqq d_X(\gamma_0,\gamma_1)$.	If $\ell=0$, then $\gamma_0=\gamma_1$, $c(\gamma_0,\gamma_1)=0$, and $v=0$.
For $|s|>0$ small enough, %
we therefore have $c(\gamma^s,\gamma_1)-c(\gamma_0,\gamma_1)=c(\gamma^s,\gamma_0)=\frac12|\gamma^s-\gamma_0|^2$ for $\liftpi\text{-a.e.\ }(\gamma,y),$
since $d_X$ and the Euclidean distance coincide locally in the interior of $\Om$. 

If $\ell>0$, choose $r>0$ such that $B_r(\gamma_0)\subset\Om$,
and fix $0<t\ll1$ small such that $\gamma_t=\gamma_0-tv\in B_r(\gamma_0).$
Then, writing $h:=\gamma^s-\gamma_0$, for $|h|$ sufficiently small, the segment joining $\gamma^s$ to $\gamma_t$ lies in
$B_r(\gamma_0)$. Hence
\begin{align*}
d_X(\gamma^s,\gamma_1)\le |\gamma^s-\gamma_t|+d_X(\gamma_t,\gamma_1)=|tv+h|+(1-t)\ell=\ell+\frac{v}{\ell}\cdot h+O(|h|^2),
\end{align*}
where the last equality follows from the expansion $|tv+h|=t\ell\big(1+2\frac{v\cdot h}{t\ell^2}+\tfrac{|h|^2}{t^2\ell^2}\big)^{1/2}=t\ell(1+\frac\delta2+O(\delta^2)),$ where $ \delta=2\frac{v\cdot h}{t\ell^2}+\tfrac{|h|^2}{t^2\ell^2}$. 
Taking the square and using $c=\frac12d_X^2$, we infer
\[c(\gamma^s,\gamma_1)-c(\gamma_0,\gamma_1)\le v\cdot h+O(|h|^2),\]
which gives the asserted upper bound and completes the proof of~\eqref{eq:cost.expansion}.
	
In view of $\gamma^s-\gamma_0=\Phi_s(\gamma_0,y)-\gamma_0=s\xi(\gamma_0,y)+O(s^2)$ for $\liftpi\text{-a.e.\ }(\gamma,y)$, the expansion~\eqref{eq:cost.expansion} yields
\[\frac{\dd}{\dd s}\bigg|_{s=0}c(\gamma^s,\gamma_1)=\nabla_x\vp_y(\gamma_0)\cdot\xi(\gamma_0,y)\qquad\text{for $\liftpi$-a.e.\ $(\gamma,y)$}.\]
This establishes the pointwise convergence. To show convergence of the integrals, observe that, for small enough $|s|$, $d_X(\Phi_s(x,y),x)\le |s|\|\xi\|_{L^\infty(\Om\times Y)}$ for $\nu$-a.e.\ $y\in Y$.
Since both $\vp_y$ and $c(\,\cdot\,,\gamma_1)$ are $L_X$-Lipschitz with respect to $d_X$, the two difference quotients obtained from~\eqref{eq:cost.intr.squeeze} are bounded in absolute value by $L_X\|\xi\|_{L^\infty(\Om\times Y)}$. Hence, the dominated convergence theorem and $(e_0,\operatorname{id}_Y)_\#\liftpi=\mu$ show that, after dividing~\eqref{eq:cost.intr.squeeze} by $s$, both outer terms in~\eqref{eq:cost.intr.squeeze} converge to the right-hand side of~\eqref{eq:differentiationW2.formula}, which completes the proof of the assertion.
\end{proof}

\subsection{Coercivity of the energy slope}

We adapt the coercivity estimate from~\cite{GHL_preprint} by Gau and two of the authors to the fibred setting and to intrinsic Wasserstein distances on possibly non-convex bounded domains, which in particular yields $H^2$ regularity of the discrete solutions~$u^k,k\ge1$. The proof is based on variations along the Neumann heat flow.
A key ingredient for estimating the Wasserstein action is the Lions--Villani-type square root estimate~\cite{LV_1995}
\begin{align}\label{eq:square-root.est}
	\|\nabla_x\sqrt{v}\|_{L^4(\Om)}^2\lesssim_d \|\nabla^2_xv\|_{L^2(\Om)}
\end{align}
for non-negative functions $v\in H^2(\Om), v\ge0$ satisfying $\pa_nv=0$ on $\pa\Om$, see~\cite[(3.4)]{GST_2009}.
Besides the $H^2$-control we here also keep the coercive contribution of the quartic potential.
Finally, since the proof only relies on the coercivity of the gradient energy, but not on its diagonal structure, we directly formulate the result in the more general case.

For clarity, we here explicitly denote the spatial gradient and Laplacian by $\nabla_x$ and $\Delta_x$, respectively.
\begin{proposition}[Slope coercivity]\label{prop:coercivity.slope}
	Let $X=\ol\Om$, where $\Om\subset\mathbb{R}^d$ is a bounded smooth domain.
	Let further $A\in \mathcal{L}(L^2(Y,\nu))$ be a self-adjoint bounded linear operator enjoying the coercivity property $\bra{A v,v}_{L^2(Y,\nu)}\ge\alpha_\ast\|v\|_{L^2(Y,\nu)}^2$
	for all $v\in L^2(Y,\nu)$ for some fixed $\alpha_\ast>0$. With $U=\int_Y\kappa u\,\dd\nu$ as before consider the energy
	\begin{align}
\calE(\mu)=\begin{cases}%
		\int_\Om\frac12(A\nabla_xu,\nabla_xu)_{L^2(Y)}\dd x+\mathcal{F}(U)+\mathcal{R}(u),%
&\mu=u\scrL^d\otimes\nu,\; u\in L^2(Y;H^1(\Om)),%
\\+\infty,&\text{otherwise},
	\end{cases}
	\end{align}
with $\mathcal{F}(U)$, $\mathcal{R}(u)$ as in~\eqref{eq:defEnergyDoubleWell},~\eqref{eq:defEnergyRepulsion} with $\vartheta>0$.

If $\calE(\mu)\le M<\infty$ and $|\partial\calE|_{\calW_{2,X,\nu}}(\mu)<\infty$, then $u\in L^2(Y;H^2(\Om))$, $\partial_nu(\cdot,y)=0$ on $\partial\Om$ for $\nu$-a.e.\ $y$, and $U^2\in H^1(\Om)$, and it holds that
	\begin{align}\label{eq:slope.coercivity}
		\int_Y\|u(\cdot,y)\|_{H^2(\Om)}^2\,\dd\nu(y)+\|U^2\|_{H^1(\Om)}^2\lesssim_{M,\Om} 1+|\partial\calE|_{\calW_{2,X,\nu}}^{4/3}(\mu),
	\end{align}
where the hidden constant also depends on the operator $A$ through $\alpha_\ast,\|A\|$, on $\mathcal{R}$ through $\|\beta\|_{L^\infty}$, and on $\vartheta$, $\|\kappa\|_{L^\infty}$.
\end{proposition}

\begin{proof}
    We adapt the argument of~\cite{GHL_preprint}. To this end, let $\mu=u\scrL^d{\otimes}\nu\in\mathrm{dom}(\calE)$.
	By definition of the metric slope
	\[|\pa\calE|_{\calW_{2,X,\nu}}(\mu)=\limsup_{\wh\mu\to\mu\text{ in }\calW_{2,X,\nu}}\frac{\big(\calE(\mu)-\calE(\wh\mu)\big)_+}{\mathcal W_{2,X,\nu}(\mu,\wh\mu)}.\]
	To obtain the asserted lower bound, we consider variations $\wh\mu$ induced by the Neumann heat semigroup $S_s$ on $\Omega$. Let $u_s(\cdot,y)\coloneqq S_su(\cdot,y)$,	$\mu_s\coloneqq u_s\scrL^d{\otimes}\nu$, $U_s:=\int_Y \kappa u_s\,\dd\nu$. The Neumann heat semigroup preserves non-negativity and the mass in each fibre, so that $\mu_s\in\calP_{2,\nu}(X{\times}Y)$. Moreover, as $s\downarrow0$, $u_s\to u$ in $L^2(Y;H^1(\Omega))$, $U_s\to U$ in $L^4(\Om)$, and
	$\mu_s\to\mu$ in $\mathcal W_{2,X,\nu}$.
	Also, $u_s\in L^2(Y;H^2(\Omega))$ and $\pa_nu_s(\cdot,y)=0$ on $\pa\Om$ for $s>0$.
	Without loss of generality, we may assume $\mu_s\not=\mu$ for a sequence $s\downarrow0$, since otherwise $u\equiv\textrm{const.},$ so that the assertion is immediate.
	
	Observing that $A$ and $S_s$ commute, we compute, for $s>0$, 
	\[-\frac{\dd}{\dd s}\calE(\mu_s)=\int_\Om(A\Delta_x u_s,\Delta_x u_s)_{L^2(Y)}\,\dd x +\int_\Om {\sf f}''(U_s)|\nabla_x U_s|^2\, \dd x+\int_\Om({\sf B}\nabla_x u_s,\nabla_x u_s)_{L^2(Y,\nu)}\,\dd x.\] so that, by the coercivity of $A$, 
    $\nu(Y)=1$, and the bound $M\ge\calE(\mu)\gtrsim\|\nabla_x u\|^2_{L^2(\Om\times Y)} \ge \|\nabla_x u_s\|^2_{L^2(\Om\times Y)}$,
	\begin{align}\label{eq:energy.diss.heat}
	\calE(\mu)-\calE(\mu_s)	\ge c\int_0^s\big(\|\Delta_xu_r\|_{L^2(\Om\times Y)}^2+\|\nabla_x(U_r^2)\|_{L^2(\Om)}^2\big)\,\dd r-CsM.%
	\end{align}
	For the Wasserstein distance, we assert the upper bound
	\begin{align}\label{eq:slope.heat.metric}
		\mathcal{W}_{2,X,\nu}^2(\mu,\mu_s)\le C_X\,	s\int_0^s\int_Y\int_\Om\frac{|\nabla_x u_r|^2}{u_r}\,\dd x\dd\nu\dd r.
	\end{align}
	In~\cite{GHL_preprint}, the underlying metric is defined through the dynamic Benamou--Brenier-type formulation. Hence, in that case, estimate~\eqref{eq:slope.heat.metric} is an immediate consequence of the fact 
	that the density--flux couple $(u_r,-\nabla_x u_r)$ is a solution to the continuity equation in $\Om$ with no-flux boundary condition. 
	In the present application, we need to cope with the possible
	non-convexity of $X=\ol\Om$ and the intrinsic distance.
	For $\nu$-a.e.\ $y$, extend $u_r(\cdot,y)\scrL^d$ and $-\nabla_x u_r(\cdot,y)\scrL^d$ by zero
	to $\R^d$. Owing to the no-flux boundary condition, the extended couple is a solution to the continuity equation on $\R^d$. 
	At this point, we conclude~\eqref{eq:slope.heat.metric} by invoking 
	the Benamou--Brenier formula on $\mathbb{R}^d$ with Euclidean cost together with the bound $|x-x'|\le d_X(x,x')\le \sqrt{C_{X}}|x-x'|$ for all $x,x'\in X$ for some $C_X\in[1,\infty)$, which follows from the smoothness of the bounded domain $\Om$.
To proceed, we rely on the boundedness of $\Om$ to use H\"older's inequality, then invoke the crucial square root inequality~\eqref{eq:square-root.est}, see~\cite[(3.4)]{GST_2009}, and elliptic regularity for the Neumann problem, which extends to non-convex domains $\Om$, to deduce after integrating over $Y$
\begin{align*}
		\int_Y\int_\Omega\frac{|\nabla_x u_r|^2}{u_r}\,\dd x\dd\nu&=4\int_Y\int_\Omega|\nabla_x \sqrt{u_r}|^2\,\dd x\dd\nu
\lesssim_{|\Om|} \bigg(\int_Y\int_\Omega|\nabla_x \sqrt{u_r}|^4\,\dd x\dd\nu\bigg)^{\frac12} \\&\lesssim_d\bigg(\int_Y\int_\Omega|\nabla^2_x u_r|^2\,\dd x\dd\nu\bigg)^{\frac12}\lesssim_\Om \bigg(\int_Y\int_\Omega|\Delta_x u_r|^2\,\dd x\dd\nu\bigg)^{\frac12}= \|\Delta_x u_r\|_{L^2(\Omega\times Y)}.
	\end{align*}
Now, let
\[a_s\coloneqq\frac1s\int_0^s\big(\|\Delta_xu_r\|_{L^2(\Om\times Y)}^2+\|\nabla_x(U_r^2)\|_{L^2(\Om)}^2\big)\,\dd r.%
\]
Combining the estimate on the Fisher information with~\eqref{eq:slope.heat.metric}
and applying Cauchy--Schwarz in time, we obtain
\[\frac1{s^2}\calW_{2,X,\nu}^2(\mu,\mu_s)\lesssim a_s^{1/2},\qquad\frac1s\big(\calE(\mu)-\calE(\mu_s)\big)_+\ge\big(ca_s-CM%
\big)_+,\]
where the second inequality follows from~\eqref{eq:energy.diss.heat}.
The definition of the metric slope now yields
\begin{align}\label{eq:slope.lower.bound.s}
|\pa\calE|_{\calW_{2,X,\nu}}(\mu)\gtrsim\limsup_{s\downarrow0}\frac{\big(ca_s-CM%
	\big)_+}{a_s^{1/4}}.
	\end{align}
 If $a_s\le 2CM/c$ for all $0<s\ll1$ the asserted bound~\eqref{eq:slope.coercivity} follows easily.
It therefore suffices to consider $s\in(0,1]$ with $a_{s}\ge 2CM/c$. Along such $s$, we infer,
invoking~\eqref{eq:slope.lower.bound.s},  that
\begin{align*}
L:=\limsup_{s\downarrow0}a_{s}^{3/4} \lesssim|\pa\calE|_{\calW_{2,X,\nu}}(\mu)<\infty.
\end{align*}
We now choose a sequence $s_k\downarrow0$ with $\lim_{k\to\infty}a_{s_k}^{3/4}=L$, and then pick $r_k\in(0,s_k)$ such that 
\[\|\Delta_xu_{r_k}\|_{L^2(\Om\times Y)}^2+\|\nabla_x(U_{r_k}^2)\|_{L^2(\Om)}^2\le a_{s_k}+\frac1k\quad\text{ for all }k\in \mathbb{N}.\]
The energy bound and the fibrewise mass constraints imply $\|u\|_{L^2(\Om\times Y)}^2+\|U\|_{L^4(\Om)}^4\lesssim_M1$.
These bounds persist under the heat flow by its $L^p$-contractivity.
Neumann elliptic regularity thus gives boundedness of $(u_{r_k})$ in $L^2(Y;H^2(\Om))$, while $(U_{r_k}^2)$ is bounded in $H^1(\Om)$.
Since $u_{r_k}\to u$ in $L^2(Y;H^1(\Om))$, and $U_{r_k}^2\to U^2$ in $L^2(\Om)$, we conclude that $u\in L^2(Y;H^2(\Om))$, $U^2\in H^1(\Om)$, and $u_{r_k}\weakly u$ in $L^2(Y;H^2(\Om))$.
Finally, since $\pa_nu_{r_k}(\cdot,y)=0$ in the trace sense on $\pa\Om$ for $\nu$-a.e.\ $y$, the weak continuity of the normal spatial trace operator on $L^2(Y;H^1(\Om;\mathbb{R}^d))$ yields  $\pa_nu(\cdot,y)=0$ on $\pa\Om$ for $\nu$-a.e.\ $y$.
By weak lower semi-continuity, $\|\Delta_xu\|_{L^2(\Om\times Y)}^2+\|\nabla_x(U^2)\|_{L^2(\Om)}^2\le L^{4/3}$.
Using once more elliptic regularity, we deduce %
\[\int_Y\|u(\cdot,y)\|_{H^2(\Om)}^2\,\dd\nu(y)+\|U^2\|_{H^1(\Om)}^2\lesssim_M 1+|\pa\calE|_{\calW_{2,X,\nu}}^{4/3}(\mu),\]
which proves~\eqref{eq:slope.coercivity}.
\end{proof}

\begin{remark}[An intrinsic approach]Using the superposition principle, the inequality~\eqref{eq:slope.heat.metric} can be obtained with $C_X=1$.
To this end, instead of resorting to the Benamou--Brenier formula on $\mathbb{R}^d$, define, using the notation of the above proof, $b_r(\cdot,y):=-\frac{\nabla_x u_r(\cdot,y)}{u_r(\cdot,y)}$ on $\{u_r(\cdot,y)>0\}$ and $b_r(\cdot,y)=0$ on $\{u_r(\cdot,y)=0\}$ and note that, for $\nu$-a.e.\ $y$
 \[\int_0^s\int_{\R^d}|b_r(\cdot,y)|\,u_r(\cdot,y)\,\dd x\dd r=\int_0^s\int_\Om|\nabla_x u_r(\cdot,y)|\,\dd x\dd r<\infty.\]
		Hence, the superposition principle~\cite[Theorem~3.2]{ambrosio2008transport} applies and yields $\eta^y\in\calP(C([0,s];\R^d))$, concentrated on absolutely continuous curves, such that $(e_r)_\#\eta^y=u_r(\cdot,y)\scrL^d$, and $\dot\gamma_r=b_r(\gamma_r,y)$
	for $\eta^y$-a.e.\ $\gamma$ and a.e.\ $r\in(0,s)$.
	Since $(e_r)_\#\eta^y$ is concentrated on $X$ for every $r$,
	restricting first to $r\in\mathbb{Q}\cap[0,s]$ and then using continuity of $\gamma\in C([0,s];\mathbb{R}^d)$ yields $\gamma_r\in X$ for every $r\in[0,s]$ for $\eta^y$-a.e.\ $\gamma$. Therefore, $d_X(\gamma_0,\gamma_s)\le\int_0^s|\dot\gamma_r|\,\dd r$.
Moreover, $\pi_s^y\coloneqq (e_0,e_s)_\#\eta^y\in\Gamma\big(u_0(\cdot,y)\scrL^d,u_s(\cdot,y)\scrL^d\big)$. It follows that
	\begin{align*}
	\mathcal{W}_{2,X}^2&\big(u_0(\cdot,y)\scrL^d,u_s(\cdot,y)\scrL^d\big)\le\int d_X^2(\gamma_0,\gamma_s)\,\dd\eta^y(\gamma)
\\&	\le s\int_0^s\int|\dot\gamma_r|^2\,\dd\eta^y(\gamma)\dd r
=s\int_0^s\int|b_r(\gamma_r,y)|^2\,\dd\eta^y(\gamma)\dd r=s\int_0^s\int_\Om
		\frac{|\nabla_x u_r(\cdot,y)|^2}{u_r(\cdot,y)}\,\dd x\dd r.
	\end{align*}
	Integrating with respect to $\nu$ yields \eqref{eq:slope.heat.metric} with $C_X=1$.
\end{remark}

\subsection{Existence by variational approximation: Proof of Theorem~\ref{thm:ex.Ediss.nonconvex}}\label{ssec:ExistNonconvex}
In this section, we establish Theorem~\ref{thm:ex.Ediss.nonconvex}. For simplicity, we only present the proof of Theorem~\ref{thm:ex.Ediss.nonconvex} for $\vartheta>0$.

As in Section~\ref{sec:diagE}, we use the Moreau--Yosida resolvent $\sfJ_h$, cf.~\eqref{eq:MoreauYosida}, now associated with $\mathcal{W}_{2,X,\nu}$ obtained from the intrinsic Wasserstein distance~$\mathcal{W}_{2,X}$, and define the fibred JKO scheme and interpolants analogously as in Definition~\ref{def:JKP_interpolations}.
Concerning compactness, the main new ingredient is the slope-coercivity estimate in Proposition~\ref{prop:coercivity.slope}, which will provide us with $H^2$ bounds for both
interpolants without resorting to the standard flow interchange argument based on the entropy, which is not available for non-convex domains.
The remaining point that requires some care is the derivation of the continuity equation using optimal dynamic plans for the intrinsic distance. 

\paragraph{Compactness}
The narrow compactness and energy-convergence arguments of Section~\ref{ssec:comp-conv} remain valid with $\calW_{2,\nu}$ replaced by $\calW_{2,X,\nu}$. Thus, after passage to a subsequence, there exists
\[\mu=u\scrL^d{\otimes}\nu\in\AC^2_{\loc}([0,\infty);\calP_{2,\nu}(X{\times}Y))\]
such that $\mu_\tau$ and $\wt\mu_\tau$ converge narrowly to $\mu$, uniformly on compact time intervals.

The energy convergence argument of Proposition~\ref{prop:energy-convergence} yields a limiting non-increasing energy function $E(t)$.
The abstract De Giorgi interpolation argument (cf.~\cite[Lemma~3.2.2]{AGS}) yields, for a.e.\ $t>0$,
\begin{align}\label{eq:discr.edi.metric.nc}
	\calE(\mu_\tau(t))+\frac12\int_0^t	\frac{1}{\tau^2}
	\calW_{2,X,\nu}^2(\mu_\tau^{k(s)},\mu_\tau^{k(s)-1})\,\dd s
	+\frac12\int_0^t\frac{1}{h_\tau(s)^2}\calW_{2,X,\nu}^2
	\bigl(\wt\mu_\tau(s),\mu_\tau^{k(s)-1}\bigr)\,\dd s\le \calE(\mu_0),
\end{align}
where $k(s)$ is determined by $s\in(t_{k(s)-1},t_{k(s)}]$ and $h_\tau(s)\coloneqq s-t_{k(s)-1}$.
Since $\mu_\tau^k\in\sfJ_\tau(\mu_\tau^{k-1})$ and $\wt\mu_\tau(t)\in\sfJ_{h_\tau(t)}(\mu_\tau^{k-1})$ for $t\in(t_{k-1},t_k]$,  the slope estimate at a Moreau--Yosida
minimiser~\cite[Lemma~3.1.3]{AGS} implies that
\begin{equation}\label{eq:discr.slopes.nc}
\begin{split}
|\pa\calE|_{\calW_{2,X,\nu}}^2(\mu_\tau(t))&\le \frac1{\tau^2}\calW_{2,X,\nu}^2(\mu_\tau^k,\mu_\tau^{k-1}),
	\\|\pa\calE|_{\calW_{2,X,\nu}}^2(\wt\mu_\tau(t))&\le \frac1{h_\tau(t)^2}\calW_{2,X,\nu}^2(\wt\mu_\tau(t),\mu_\tau^{k-1}).
\end{split}
\end{equation}
Since $\calE(\mu_\tau(t)),\calE(\wt\mu_\tau(t))\le\calE(\mu_0)$, Proposition~\ref{prop:coercivity.slope}, together with~\eqref{eq:discr.slopes.nc},~\eqref{eq:discr.edi.metric.nc} and H\"older's inequality, yields, for every $T<\infty$,
\begin{align*}
	\int_0^T\int_Y\big(\|u_\tau\|_{H^2(\Om)}^2+\|\wt u_\tau\|_{H^2(\Om)}^2\big)\dd\nu\dd t&+\int_0^T\big(\|U_\tau^2\|^2_{H^1(\Om)}+\|\wt U_\tau^2\|^2_{H^1(\Om)}\big)\,\dd t
\\&\lesssim\int_0^T\big(1+|\pa\calE|_{\calW_{2,X,\nu}}^{4/3}(\mu_\tau)+|\pa\calE|_{\calW_{2,X,\nu}}^{4/3}(\wt\mu_\tau)\big)\,\dd t	\lesssim_T 1.
\end{align*}
Consequently, after passage to a further subsequence,
\[u_\tau,\wt u_\tau\rightharpoonup u\quad\text{in }L^2((0,T)\times Y;H^2(\Om)).\]
Similarly as in  Section~\ref{sec:diagE}, cf.~Lemma~\ref{l:GMM_convergence_strong}~\ref{l-item:GMM_convergence_strong}, we obtain compactness for quantities that are averaged in $Y$.  Combined with interpolation this yields,
for every $\psi\in L^2(Y,\nu)$,
\[\langle u_\tau,\psi\rangle_Y,\,	\langle\wt u_\tau,\psi\rangle_Y\to\langle u,\psi\rangle_Y\quad\text{strongly in }L^\infty(0,T;L^2(\Om))\cap L^2(0,T;H^1(\Om)).\]
Using also the $L^2_{t}H^1_x\cap L^\infty_{t}L^2_x$ bounds for $U_\tau^2,\wt U_\tau^2$, this further implies that 
\[U_\tau,\wt U_\tau\to U\quad\text{strongly in }L^4((0,T)\times\Om).\]
These convergence properties easily yield convergence of the lower-order energies
\begin{align}\label{eq:E0.conv.nc}
	\calE_0(\mu_\tau),\calE_0(\wt\mu_\tau)\to\calE_0(\mu)\quad\text{in }L^1(0,T),
	\end{align}
where for the nonlocal term one argues as in %
Lemma~\ref{l:ApproxRepul}, cf.~\eqref{eq:estimateRepulsionApprox}.

\paragraph{Energy-dissipation inequality and flux-defect control}
For $t\in(t_{k-1},t_k]$, let $y\mapsto \vp_\tau^k(\cdot,y)$ be a Borel selection of fibrewise Kantorovich potentials from $\mu_\tau^k$ to $\mu_\tau^{k-1}$ according to Lemma~\ref{l:Kant-pot.meas}.
Likewise,  let $\wt\vp_\tau$ be a Borel measurable fibrewise Kantorovich potential from $\wt\mu_\tau$ to $\mu_\tau^{k-1}$.

By optimality and the first-variation formula~\eqref{eq:differentiationW2.formula} for the intrinsic Wasserstein distance, a calculation analogous to Proposition~\ref{prop:E-L_2} yields the following Euler--Lagrange equations for tangential vector fields.

\begin{lemma}[Euler--Lagrange equations]\label{l:EL.nc}
Define
\[j_\tau(t,\cdot):=\frac1\tau u_\tau^k\nabla\vp_\tau^k,\qquad 
\wt\jmath_\tau(t,\cdot):=\frac1{h_\tau(t)}\wt u_\tau(t,\cdot)\nabla\wt\vp_\tau(t,\cdot),\quad t\in(t_{k-1},t_k].\]
	Then, for a.e.\ $t>0$ and every $\xi\in L^\infty(Y;C^2_{\tan}(\overline\Omega;\mathbb R^d))$, it holds
	\begin{align}\label{eq:EL.nc}
		\int_\Om\int_Y j_\tau\cdot\xi\,\dd\nu\dd x
		&=\int_\Om\int_Y\alpha(y)\bigg[
		\bigg(\frac12\operatorname{div}_x\xi\,I+\nabla_x\xi\bigg)\nabla_xu_\tau
		+u_\tau\nabla_x\operatorname{div}_x\xi
		\bigg]\cdot\nabla_xu_\tau\,\dd\nu\dd x \notag\\
		&\quad+\int_\Om\int_Y
		\big(\kappa(y){\sf f}'(U_\tau)+\mathsf B u_\tau\big)
		\big(\nabla_xu_\tau\cdot\xi+u_\tau\operatorname{div}_x\xi\big)\,\dd\nu\dd x,
	\end{align}
	where $U_\tau=\int_Y\kappa(y)u_\tau(\cdot,y)\,\dd\nu(y)$.
	The same identity holds with $(u_\tau,U_\tau,j_\tau)$ replaced by $(\wt u_\tau,\wt U_\tau,\wt\jmath_\tau)$, where $\wt U_\tau=\int_Y\kappa(y)\wt u_\tau(\cdot,y)\,\dd\nu(y)$.
\end{lemma}

The compactness and convergence properties established above are sufficient to pass to the limit in the  second term on the right-hand side of~\eqref{eq:EL.nc}. More precisely, we assert the following.
\begin{lemma}[Identification of lower-order part of flux]\label{l:flux-passage.lower}
	 Along a subsequence $\tau\downarrow0$, it holds that for every $T<\infty$ and $\xi\in L^\infty((0,T)\times Y;C^2_{\tan}(\overline\Omega;\mathbb R^d))$
\begin{multline}\label{eq:flux.lower-order.lim}
\int_0^T\int_\Om\int_Y\big(\kappa(y){\sf f}'(U_\tau)+\mathsf B u_\tau\big)
\big(\nabla_xu_\tau\cdot\xi+u_\tau\operatorname{div}_x\xi\big)\,\dd\nu\dd x\dd t\\ \longrightarrow \quad
\int_0^T\int_\Om\int_Y\big(\kappa(y){\sf f}'(U)+\mathsf B u\big)
\big(\nabla_xu\cdot\xi+u\operatorname{div}_x\xi\big)\,\dd\nu\dd x\dd t,\qquad 
\end{multline}
and the same holds with $u_\tau, U_\tau$ replaced by the De Giorgi interpolation $\wt u_\tau,\wt U_\tau$.
\end{lemma}

\begin{proof}[Proof of Lemma~\ref{l:flux-passage.lower}]
The passage to the limit in the nonlocal part involving ${\sf B}u_\tau$ (resp.\ ${\sf B}\wt u_\tau$) is analogous to that in Section~\ref{sec:diagE}. It hence remains to consider the potential part coming from the double-well potential, which we write, using integration by parts, as 
\begin{align*}
-\int_0^T\int_\Om\int_Y\kappa u_\tau\xi\,\dd\nu\cdot \nabla {\sf f}'(U_\tau)\dd x.
\end{align*}
We now define
\[M_\tau:=\int_Y\kappa u_\tau\xi\,\dd\nu,\qquad M:=\int_Y\kappa u\xi\,\dd\nu.\]
By the uniform bound on the energy in~\eqref{eq:discr.edi.metric.nc} and the slope coercivity~\eqref{eq:slope.coercivity}, $(U_\tau^2)$ is $\tau$-uniformly bounded in $L^\infty(0,T;L^2(\Omega))\cap L^2(0,T;H^1(\Omega))$, giving that $U_\tau$  is bounded in $L^{4+\delta}((0,T)\times\Om)$ for $\delta=8/d>0$, thanks to the space--time Gagliardo--Nirenberg inequality. 
Combined with the strong %
convergence of the weighted $Y$-averages (cf.~ Lemma~\ref{l:GMM_convergence_strong}~\ref{l-item:GMM_convergence_strong}) and the bound $|M_\tau|\le\|\xi\|_{L^\infty} U_\tau$, we infer
\begin{align}\label{eq:U4.conv.strong}
	U_\tau\to U,\quad M_\tau\to M\qquad\text{strongly in }L^4((0,T)\times\Om).
\end{align}
Consequently, $(U_\tau-1)M_\tau\to(U-1)M$ strongly in $L^2((0,T)\times\Om)$,
whereas $\nabla(U_\tau^2)\rightharpoonup\nabla(U^2)$ and
$\nabla U_\tau\rightharpoonup\nabla U$ weakly in $L^2((0,T)\times\Om)$.
Motivated by this, using $\nabla{\sf f}'(U_\tau)=\frac{3\vartheta}{2}(U_\tau-1)\nabla(U_\tau^2)+\frac\vartheta2\nabla U_\tau$, we rewrite
\begin{align*}
	-\int_0^T\!\int_\Om\int_Y\kappa  u_\tau\xi\cdot\nabla{\sf f}'(U_\tau)\,\dd\nu\dd x\dd t
	&=-\frac\vartheta2\int_0^T\!\int_\Om\big(3(U_\tau-1)M_\tau\cdot\nabla(U_\tau^2)+M_\tau\cdot\nabla U_\tau\big)\,\dd x\dd t=:I(U_\tau,M_\tau).
\end{align*}
In the limit $\tau\downarrow0$, the right-hand side converges to $I(U,M)$, which equals
$-\int_{(0,T)\times\Om}\int_Y \kappa u\xi\cdot\nabla{\sf f}'(U)\,\dd\nu\dd x\dd t.$ 
Along the same lines, it follows that $I(\wt U_\tau,\wt M_\tau)\to I(U,M)$.
Overall, we have thus shown~\eqref{eq:flux.lower-order.lim}.
\end{proof}

We are now in a position to derive the energy-dissipation inequality~\eqref{eq:EDI.diagE}.
Using~\eqref{eq:W2.KantorovichPot}, we rewrite~\eqref{eq:discr.edi.metric.nc} as 
\begin{align}\label{eq:discr.edi.j.nc}
	\calE(\mu_\tau(t)) +\frac12\int_0^t\int_\Om\int_Y\bigg(\frac{\abs{j_\tau}^2}{u_\tau}
	+ \frac{\abs{\wt\jmath_\tau}^2}{\wt u_\tau}\bigg)\,\dd\nu\dd x\dd s\le \calE(\mu_0).
\end{align}
As in Proposition~\ref{prop:disc_EDI}, cf.~\eqref{eq:flux_uniform-integrability}, this bound implies the equi-integrability of $j_\tau$, $\wt\jmath_\tau$ and hence weak $L^1$ convergence along a subsequence. From the strong convergence of the difference $u_\tau-\wt u_\tau$ in $L^2_{t,y}H^1_x$,  identity~\eqref{eq:EL.nc}, valid both for $j_\tau$ and $\wt\jmath_\tau$,  combined with the convergence of the lower-order part of the flux term (cf.~Lemma~\ref{l:flux-passage.lower}), it follows that 
the weak $L^1$ limits must coincide, i.e., there exists $j\in L^1_\loc([0,\infty);L^1(\Om\times Y))$ such that (along a subsequence)
\[j_\tau,\wt\jmath_\tau\rightharpoonup j\qquad\text{weakly in }L^1((0,T)\times \Om\times Y;\R^d)\quad\text{for all }T<\infty.\]
By the convergence of the energies to some non-increasing limit $E=E(t)$, which follows as in Section~\ref{sec:diagE},
and the joint weak lower semicontinuity of the dissipation terms, we can now pass to the limit in
\eqref{eq:discr.edi.j.nc}  to infer, for a.e.\ $t>0$,
\[E(t)+\int_0^t\int_\Om\int_Y\frac{|j|^2}{u}\,\dd\nu\dd x\dd s\le\calE(\mu_0),\]
where $E(t)\ge\calE(\mu(t))$.

To establish the flux-defect estimate~\eqref{eq:def.e-diss.flux}, define $z_\tau:= u_\tau-u$ and $E_{\rm def}(t):= E(t)-\calE(\mu(t))$. 
Using the strong convergence~\eqref{eq:E0.conv.nc} of the lower-order energies and of $\mathcal{E}(\mu_\tau)\to E$ in $L^1(0,T)$,
combined with weak convergence $\nabla z_\tau\weakly0$ in $L^2_{t,x,y}$, it follows that
\begin{align}\label{eq:gradient.defect.nc}
	\frac12\int_\Om\int_Y\alpha|\nabla z_\tau|^2\,\dd\nu\dd x=\mathcal{E}_1(\mu_\tau)-\mathcal{E}_1(\mu)
	-\int_\Om\int_Y\alpha\nabla z_\tau\cdot\nabla u\,\dd\nu\dd x
	\weakly E_{\rm def} \qquad\text{ in }L^1(0,T).
\end{align}
Subtracting the distributional formula for $j_u$ from the Euler--Lagrange equation for $j_\tau$, expanding the quadratic terms, and using the convergence of the lower-order terms (cf.~Lemma~\ref{l:flux-passage.lower}), we deduce
\begin{align*}
\bigg|\int_0^T\!\la j-j_u,\xi\ra\dd t\bigg|
\le\limsup_{\tau\downarrow0}\bigg|\int_0^T\!\!\int_\Om\int_Y\alpha\Big((\nabla z_\tau\otimes\nabla z_\tau):\nabla\xi+\frac12|\nabla z_\tau|^2\divv\xi+z_\tau\nabla z_\tau\cdot\nabla\divv\xi\Big)\,\dd\nu\dd x\dd t\bigg|.
\end{align*}
We now choose $\xi=\chi_I\zeta$ with  fixed but arbitrary $\zeta \in L^\infty(Y;C^2_{\tan}(\ol\Om))$ and  $I\subseteq[0,T]$ measurable.
Then, using also the fibrewise Poincar\'e inequality $\|z_\tau(t,\cdot,y)\|_{L^2(\Om)}\lesssim_\Om\|\nabla z_\tau(t,\cdot,y)\|_{L^2(\Om)}$ and~\eqref{eq:gradient.defect.nc}, we conclude
\begin{align}
	\bigg|\int_I\la j-j_u,\zeta\ra\,\dd t\bigg|\lesssim_\Om \|\zeta\|_{L^\infty(Y;C^2(\ol\Om))}\int_I  E_{\rm def}(t)\,\dd t.
\end{align}
 Since $I\subseteq[0,T]$ was arbitrary, localisation and an
approximation argument using the separability of the set $\{\zeta\in L^\infty(Y;C^2_{\tan}(\ol\Om)):\|\zeta\|_{L^\infty(Y;C^2(\ol\Om))}\le 1\}\subset L^1(Y;C^2(\ol\Om))$
yield~\eqref{eq:def.e-diss.flux} for all $t\in [0,T]\setminus N$ with a negligible set $N$ independent of $\zeta$.

\paragraph{Derivation of the continuity equation}
In the Euclidean setting of Section~\ref{ssec:proof.ex.diag}, the continuity equation was obtained directly from the transport maps between consecutive states by a Taylor expansion. In the present intrinsic setting, where velocities along geodesics need not be constant, it appears more transparent to work instead with the geodesic interpolation, which satisfies an exact continuity equation at the discrete level.
Let $\liftpi_\tau^k\in\calP(\operatorname{Geo}(\ol\Om)\times Y)$
be a fibrewise optimal dynamic plan from $\mu_\tau^k$ to $\mu_\tau^{k-1}$,  cf.~Lemma~\ref{l:dyn.plan.meas}.
In particular, $(e_0,\id_Y)_\#\liftpi_\tau^k=\mu_\tau^k$, and $(e_1,\id_Y)_\#\liftpi_\tau^k=\mu_\tau^{k-1}$.
We define the time-rescaled geodesic interpolant and its dynamic flux by
\[\rho_\tau(t)\coloneqq (e_{1-r},\id_Y)_\#\liftpi_\tau^k,\qquad
m_\tau(t)\coloneqq (e_{1-r},\id_Y)_\#\left(-\frac1\tau\dot\gamma(1-r)\,\liftpi_\tau^k\right),\qquad t\in((k{-}1)\tau,k\tau],\]
where $r\coloneqq \tau^{-1}(t-(k{-}1)\tau)\in(0,1]$. Equivalently,  for every 
$\xi\in L^\infty(Y;C(\ol\Om;\mathbb R^d))$
and a.e.\ $t\in((k{-}1)\tau,k\tau]$,
\begin{align}\label{eq:dyn.flux}
	\int_{\ol\Om\times Y}\xi\cdot\dd m_\tau(t):=-\frac1\tau\int_{\GEO(\ol\Om)\times Y}\xi(\gamma_{1-r},y)\cdot\dot\gamma(1-r)\,\dd\liftpi_\tau^k(\gamma,y).
\end{align}
By the chain rule along the geodesics, it follows that
\begin{align}\label{eq:CEdiscr.geod}
	\int_0^T\int_{\ol\Om\times Y}\pa_t\phi(t,\cdot)\,\dd\rho_\tau(t)\dd t
	+\int_0^T\int_{\ol\Om\times Y}\nabla\phi(t,\cdot)\cdot\dd m_\tau(t)\dd t
	+\int_{\ol\Om\times Y}\phi(0,\cdot)\,\dd\mu_0=0
\end{align} 
for all $\phi\in L^\infty(Y;C^1([0,T]\times\ol\Om))$ such that $\phi(T,\cdot,y)=0$ for $\nu$-a.e.\ $y$.
Furthermore, since $\liftpi_\tau^k$ is concentrated on $\GEO(\ol\Om)\times Y$ and induces an optimal coupling between $\mu_\tau^k$ and $\mu_\tau^{k-1}$, it is not difficult to see, using Jensen's inequality, that
\begin{align}\label{eq:mtau.rhotau}
\int_0^T\int_{\ol\Om\times Y}\bigg|\frac{\dd m_\tau}{\dd\rho_\tau}\bigg|^2\,\dd\rho_\tau\dd t\le\sum_k\frac1\tau\mathcal W_{2,X,\nu}^2(\mu_\tau^k,\mu_\tau^{k-1})\lesssim1.
\end{align}
The coupling induced by $(e_{1-r},e_0,\id_Y)_\#\liftpi_\tau^k$ also yields
for $t\in(t_{k-1},t_k]$ and $r=(t-t_{k-1})/\tau$,
\begin{align*}
	\mathcal W_{2,X,\nu}^2(\rho_\tau(t),\mu_\tau(t))
	&\le \int_{\GEO(\ol\Om)\times Y} d_X^2(\gamma_{1-r},\gamma_0)\,\dd\liftpi_\tau^k(\gamma,y)\\
	&=(1-r)^2\int_{\GEO(\ol\Om)\times Y} d_X^2(\gamma_1,\gamma_0)\,\dd\liftpi_\tau^k(\gamma,y)
	=(1-r)^2\mathcal W_{2,X,\nu}^2(\mu_\tau^k,\mu_\tau^{k-1}).
\end{align*}
Therefore,
\begin{align*}
	\int_0^T\mathcal W_{2,X,\nu}^2(\rho_\tau(t),\mu_\tau(t))\,\dd t
	\lesssim \tau\sum_k\mathcal W_{2,X,\nu}^2(\mu_\tau^k,\mu_\tau^{k-1})\to0.
\end{align*}
Consequently, along a subsequence, $\rho_\tau(t)\dd t\weakstar\mu_t\dd t$ in the sense of narrow convergence, and $m_\tau\dd t\weakstar m$ in $\mathcal{M}([0,T]{\times} X{\times }Y;\mathbb{R}^d)$. 
By lower semi-continuity of the action (cf.\ e.g.~\cite[Lemma~10.1]{ABS_2024}), we further have that 
$m\ll \mu_t\dd t= u\scrL^{1+d}\otimes\nu$. More specifically, $\dd m=j_m\dd(\scrL^{1+d}\otimes\nu)$ with 
\[ \int_0^T\int_{\ol\Om\times Y}\frac{|j_m|^2}{u}\,\dd\nu\dd x\dd t \le \liminf_{\tau\downarrow0} \int_0^T\int_{\ol\Om\times Y}
\bigg|\frac{\dd m_\tau}{\dd\rho_\tau}\bigg|^2 \,\dd\rho_\tau\dd t. \]
Since $\rho_\tau(t)\dd t$ and $\mu_t\dd t$ have the common $(t,y)$-marginal $\scrL^1\otimes\nu$, the narrow convergence
$\rho_\tau(t)\dd t\weakstar\mu_t\dd t$ implies, by approximation, that
\[\int_0^T\int_{X{\times}Y}\phi\,\dd\rho_\tau(t)\dd t\to \int_0^T\int_{X{\times} Y}\phi\,\dd\mu_t\dd t\qquad\text{for all }\phi\in L^1((0,T)\times Y;C(X)).\]
Taking also into account the bound~\eqref{eq:mtau.rhotau}, we further infer
\[ \bigg| \int_0^T\int_{X\times Y}\xi\cdot\dd m_\tau(t)\dd t \bigg|\lesssim\|\xi\|_{L^2((0,T)\times Y;C(X;\mathbb R^d))}\]
and an analogous estimate for $m$. Thus, similarly as for the mass distributions, we obtain convergence of the flux term in the sense
\[\int_0^T\int_{X{\times}Y}\xi\cdot\dd m_\tau(t)\dd t\to\int_{(0,T){\times}X{\times} Y}\xi\cdot\dd m\qquad\text{for all }\xi\in L^2((0,T)\times Y;C(X;\mathbb{R}^d)).\]
We may therefore pass to the limit in~\eqref{eq:CEdiscr.geod} for every $\phi\in L^\infty(Y;C^1([0,T]\times\ol\Om))$
with $\phi(T,\cdot,y)=0$ for $\nu$-a.e.\ $y$.
Localising in $y$ and using a countable $C^1$-dense family of scalar test functions, we conclude in particular that, for
$\nu$-a.e.\ $y\in Y$ and every $\phi\in C_c^1([0,T)\times\ol\Om)$,
\begin{align}\label{eq:CE.geod}
 \int_0^T\int_\Om\big(u(\cdot,y)\pa_t\phi+j_m(\cdot,y)\cdot\nabla_x\phi\big)\,\dd x\dd t +\int_\Om u_0\phi_{|t=0}\,\dd x=0.
\end{align}
With these observations we finally show the equality $j_m=j$.
For this, we make use of the lemma below.

\begin{lemma}[Error between dynamic and Euler--Lagrange flux]\label{l:err.EL.dyn}
For all $T<\infty$, every compact set $K\Subset\Om$, and every $\xi\in L^\infty((0,T)\times Y;C^1(\ol\Om;\mathbb R^d))$
satisfying $\supp\xi(t,\cdot,y)\subseteq K$ for $\dd t\otimes\nu$-a.e.\ $(t,y)$, we have
\begin{align}\label{eq:err.EL.dyn}
\bigg|\int_0^T\int_{X\times Y}\xi\cdot\dd m_\tau\dd t-\int_0^T\int_\Om\int_Y\xi\cdot j_\tau\,\dd\nu\dd x\dd t\bigg|
\lesssim_K\|\xi\|_{L^\infty((0,T)\times Y;C^1(\ol\Om))}\sum_k\mathcal W_{2,X,\nu}^2(\mu_\tau^k,\mu_\tau^{k-1}).
\end{align}
\end{lemma}

\begin{proof}
	Using $(e_0,\id_Y)_\#\liftpi_\tau^k=\mu_\tau^k$, we can write 
	\begin{align*}
		\int_{(k{-}1)\tau}^{k\tau}\int_\Om\int_Y\xi\cdot j_\tau^k\,\dd\nu\dd x\dd t
		=\frac1\tau\int_{(k{-}1)\tau}^{k\tau}\int_\Om\int_Y\xi\cdot \nabla\vp_\tau^k\,\dd
		\mu_\tau^k\dd t. %
	\end{align*}
	Consequently, changing variables $t=t_k(r)\coloneqq (k-1+r)\tau$ and invoking Lemma~\ref{l:ini.vel.geo} fibrewise,  we deduce
	\begin{align}\nonumber
		\err_k&\coloneqq \int_{(k{-}1)\tau}^{k\tau}\int_{X\times Y}\xi\cdot\dd m_\tau\dd t-\int_{(k{-}1)\tau}^{k\tau}\int_\Om\int_Y\xi\cdot j_\tau^k\,\dd\nu\dd x\dd t\\
		&=\int_0^1\int_{\GEO(X)\times Y}\big(\xi(t_k(r),\gamma_{1-r},y)\cdot(-\dot\gamma(1-r))-\xi(t_k(r),\gamma_0,y)\cdot(-\dot\gamma(0+))\big)\,\dd\liftpi_\tau^k(\gamma,y)\dd r.\label{eq:err.k.rewrite}
	\end{align}
Let $\delta:=\dist(K,\pa\Om)>0$.	
Since the extension of $\xi$ by zero in the spatial variable is globally Lipschitz, we have $|\xi(t,x,y)-\xi(t,x',y)|\lesssim_\xi|x-x'|$.
	For $\liftpi_\tau^k$-a.e.\ $(\gamma,y)$ and a.e.\ $r\in(0,1)$, the constant-speed property gives
	\[|\dot\gamma(1-r)|=\ell\coloneqq d_X(\gamma_0,\gamma_1),\qquad
	|\gamma_{1-r}-\gamma_0|\le d_X(\gamma_{1-r},\gamma_0)=(1-r)\ell.\]
	Hence,
	\begin{align*}
		\big|\xi(t_k(r),\gamma_{1-r},y)\cdot\dot\gamma(1-r)-\xi(t_k(r),\gamma_0,y)\cdot\dot\gamma(0+)\big|\lesssim_\xi (1-r)\ell^2+|\xi(t_k(r),\gamma_0,y)|\,|\dot\gamma(1-r)-\dot\gamma(0+)|.
	\end{align*}
	The first term on the right-hand side being controlled by $\ell^2$, we aim to obtain a similar bound for the second term. 
    To this end, first observe that if $\xi(t_k(r),\gamma_0,y)=0$, it vanishes. Otherwise, we have $\gamma_0\in K$.
	Thus, if $\ell<\delta$, the geodesic is contained in $B_\delta^{\textrm{Eucl}}(\gamma_0)\subset\Om$ and hence, by minimality, it is a Euclidean line segment, so that $\dot\gamma(1-r)=\dot\gamma(0+)$ for a.e.\ $r$. 
    Conversely, if $\ell\ge\delta$, then
	\[|\dot\gamma(1-r)-\dot\gamma(0+)|\le2\ell\le2\delta^{-1}\ell^2.\]
	Thus, the integrand in~\eqref{eq:err.k.rewrite} is bounded in modulus by $C_\xi\ell^2$. Integration in $r$ and optimality of $\liftpi_\tau^k$ yield
	\[|\err_k|\lesssim_\xi\int_{\GEO(X)\times Y}d_X^2(\gamma_0,\gamma_1)\,\dd\liftpi_\tau^k(\gamma,y)
	=\mathcal W_{2,X,\nu}^2(\mu_\tau^k,\mu_\tau^{k-1}).\]
	With this, a summation over $k$ proves~\eqref{eq:err.EL.dyn}.
\end{proof}

 Since $\sum_k\mathcal W_{2,X,\nu}^2(\mu_\tau^k,\mu_\tau^{k-1})
=\tau\sum_k\frac1\tau\mathcal W_{2,X,\nu}^2(\mu_\tau^k,\mu_\tau^{k-1})
\lesssim\tau\to0$,
Lemma~\ref{l:err.EL.dyn} and the weak convergence properties established above yield
\[\int_0^T\int_\Om\int_Y \xi\cdot(j_m-j)\,\dd\nu\dd x\dd t=0\]
for every $\xi$ in the test class of Lemma~\ref{l:err.EL.dyn}. It follows that $j_m=j$ a.e., so that~\eqref{eq:CE.geod} is the required continuity equation with flux $j$, including the initial condition and the no-flux boundary condition, for $\nu$-a.e.\ $y$.

All remaining assertions can be derived analogously as in the proof of Theorem~\ref{thm:existence}. This completes the construction of energy-dissipating solutions.

\paragraph{Weak solvability for finitely many types} 
If $\supp\nu$ is finite, the narrow convergence implies 
that \[\mathcal{W}_{2,X,\nu}(\mu_\tau(t),\mu(t))\to0\quad \text{ locally uniformly in }t\in[0,\infty).\]
 Interpolating between $L^2(Y;H^2(\Om))$ and $\mathcal{W}_{2,X,\nu}$ we then infer
\begin{align}
	u_\tau,\wt u_\tau\to u\quad \text{in }L^2((0,T)\times Y;H^1(\Om)),
\end{align}
which implies $E(t)=\mathcal{E}(\mu_t)$ for a.e.\ $t>0$ and $j=j_u$, and hence the weak solution property.

\section{Weak--strong uniqueness principles}\label{sec:uniq.wkstr}
In this section, we establish the proofs of Theorems~\ref{thm:wk-str.uniq.e-diss} and~\ref{thm:wk-str.uniq}. 
We present the proofs for $\vartheta>0$ and $\kappa\equiv1$. The modifications required for $\vartheta=0$ and for general bounded $\kappa\geq\kappa_*>0$ are omitted for brevity. 

\subsection{Energy-dissipating solutions on bounded domains}\label{ssec:wsu.diag.e-diss}
In this section we show Theorem~\ref{thm:wk-str.uniq.e-diss}. Let $X=\overline\Omega$, where $\Omega\subset\mathbb R^d$ is a bounded smooth domain.
Let $(u,j,E)$ be an energy-dissipating solution in the sense of Definition~\ref{def:e-diss.sol}, with initial data $\mu_0=u_0\scrL^d{\otimes}\nu$, enjoying the additional regularity~\eqref{eq:ediss-sol.reg}.
Further, let $v$ denote a local strong solution with the regularity specified in Theorem~\ref{thm:wk-str.uniq.e-diss}.
Note that conservation of mass and the initial condition $v(0)=u_0$ imply that $\int_X v(t,x,y)\,\dd x=1$ for all $t\in[0,T]$ and $\nu\text{-a.e.\ }y\in Y$.

We define $U:=\int_Yu\dd\nu, V:=\int_Y v\dd\nu$, recall the notations
\begin{align*}
	w_u=-\alpha\Delta u+{\sf f}'(U)+\mathsf Bu,\qquad w_v=-\alpha\Delta v+{\sf f}'(V)+\mathsf Bv,
\end{align*}
and abbreviate $z:=u-v, Z:=U-V$.

To compensate for the non-convexity of  $\mathcal{E}_0$, we augment the classical relative energy by a squared negative Sobolev norm. To this end, we fix an integer $N>1+d/2$, let $\mathcal V=H^N(\Omega)$ with dual $\mathcal{V}^*$, and define, for $\lambda\ge1$ to be determined, 
\[S_\rel:=E_\rel+ \frac\lambda2\|z\|_{L^2(Y;\mathcal V^*)}^2,\]
where 
\begin{align*}
E_\rel(t)&=E_{\dft}(t)+\mathcal{E}_\rel(u|v)= E(t)-\calE(v)-\la  w_v,u-v\ra,
\end{align*}
with the understanding that $E(0):=\mathcal{E}(\mu_0)$, $E_{\dft}(0):=0$, so that $E_\rel(0)=\mathcal{E}_{\rel}(u_0|v_0)$ denotes the classical relative energy. Here, we use the shorthand $\mathcal{E}(v)$ to denote $\mathcal{E}(v\scrL^d\otimes\nu)$ etc. Similarly, below, we write `$u\in\mathrm{dom}(\mathcal{E})$' if $\mu=u\scrL^d\otimes\nu\in \mathrm{dom}(\mathcal{E})$. Finally, observe that, since $\mathcal{V}\embeds L^\infty(\Om)$, it holds that $\|u\|_{L^2(Y;\mathcal{V}^*)}\lesssim \|u\|_{L^2(Y;L^1(\Om))}\lesssim 1$.

We start with two auxiliary results that are valid in every space dimension $d\ge1$.

\begin{lemma}[Coercivity of augmented relative energy]\label{l:aug-relen.coerc}
	There exists $\lambda<\infty$ such that for all $u\in\mathrm{dom}(\mathcal{E})$ it holds that 
	\begin{align}\label{eq:Sdef<Srel.e}
	\mathcal{E}_\rel(u|v)+ \frac\lambda2\|z\|_{L^2(Y;\mathcal V^*)}^2\gtrsim \|z\|_{L^2(Y;H^1(\Omega))}^2+\vartheta\|Z\|_{L^4(\Omega)}^4.
	\end{align}
\end{lemma}
\begin{proof}
We first recall that by Ehrling's lemma
for all $\ve>0$ there exists $C_\ve<\infty$ such that
\begin{align}\label{eq:interpol.Ehrling-type}
\|\tilde z\|_{L^2(\Omega)}\le \ve\|\nabla \tilde z\|_{L^2(\Om)}+C_\ve \|\tilde z\|_{\mathcal{V}^*}\qquad\text{ for all }\tilde z\in H^1(\Om).
\end{align}
We now estimate, using the lower bound~\eqref{eq:pot-rel.coerc} and $\|Z\|_{L^2(\Om)}\le\|z\|_{L^2(\Om\times Y)}$.
\begin{align*}
\mathcal{E}_\rel(u|v)&=\frac12\int_\Om\int_Y\alpha|\nabla z|^2\,\dd\nu\dd x
	+\int_\Om\big({\sf f}(U)-{\sf f}(V)-{\sf f}'(V)Z\big)\,\dd x
	+\frac12\int_\Om(\mathsf Bz,z)_{L^2(Y,\nu)}\,\dd x\\
	&\ge\frac{\alpha_*}{2}\|\nabla z\|_{L^2(\Om\times Y)}^2
	+\frac{\vartheta}{12}\|Z\|_{L^4(X)}^4
	-\bigg(\frac\vartheta8+\frac{\beta^*}{2}\bigg)\|z\|_{L^2(\Om\times Y)}^2.
\end{align*}
Invoking~\eqref{eq:interpol.Ehrling-type} fibrewise, choosing $\ve>0$ sufficiently small and then $\lambda$ large enough (only depending on fixed parameters), we deduce~\eqref{eq:Sdef<Srel.e}.
\end{proof}

We next derive a sharp inequality for the evolution of $E_\rel$ using the energy-dissipation property.

\begin{lemma}[Relative energy inequality]\label{l:EDI.rel}
	For a.e.\ $t\in(0,T)$, it holds
	\begin{align*}
	E_\rel(t) +	\int_0^t\int_\Om\int_Yu\left|\frac{j}{u}+\nabla w_v\right|^2\,\dd\nu\dd x\dd s&\le E_\rel(0) + \sum_{m=1}^4 \int_0^t\mathrm{err}_m\,\dd s
	\end{align*}
	where, still denoting $z=u-v$ and $Z=U-V$ and recalling $j_u$ from~\eqref{eq:flux.distrib}, we have
	\begin{subequations}
		\begin{align}
			\mathrm{err}_1&:= -\int_\Om\int_Y\alpha\divv\bigl(z\nabla w_v\bigr)\Delta z\,\dd\nu\dd x
			\label{eq:grad-err.v1} 
			\\ \label{eq:grad-err.v2}&\phantom{:}= \int_\Om\int_Y\alpha\Big(\nabla^2 w_v:\nabla z\otimes\nabla z+\frac12\Delta w_v|\nabla z|^2+z\nabla z\cdot\nabla\Delta w_v\Big)\,\dd\nu\dd x ,\\
		\nonumber	\mathrm{{err}}_2&:=\langle j-j_u,\nabla w_v\rangle,\\
			\nonumber	\mathrm{err}_3&:= -\int_\Om\partial_t V\big({\sf f}''(V)Z-{\sf f}'(U)+{\sf f}'(V)
			\big)\,\dd x,\\
			\mathrm{err}_4&:=-\int_\Om\int_Y z\nabla w_v\cdot\big({\sf f}''(U)\nabla U-{\sf f}''(V)\nabla V+\mathsf B\nabla z\big)\,\dd\nu\dd x.
			\nonumber
		\end{align}
	\end{subequations}
\end{lemma}

\begin{proof}
To keep the argument concise we only present the formal calculations, since the approximation arguments required to justify them are standard.
We write $E_\rel=E-E_{\mathrm{aff}}$ with the affine correction
	\[	E_{\mathrm{aff}}(t)\coloneqq\mathcal E(v(t))+\int_\Om\int_Y w_v(t)z(t)\,\dd\nu\dd x.\]
Using the continuity equation for $u$ and the chain rule for $\mathcal E(v)$, we compute
	\begin{align}\label{eq:Eaff.evol}
		 \frac{\dd}{\dd t}E_{\mathrm{aff}}=\la w_v,\pa_tv\ra+\la\pa_t w_v,z\ra+\la w_v,\pa_tu-\pa_tv\ra=\la\pa_t w_v,z\ra+\la j,\nabla w_v\ra.	
		 \end{align}
Inserting the formula for $w_v$, using the symmetry of $\mathsf B$, and 
	$\pa_t v=\divv(v\nabla w_v)$, %
we have
	\begin{align*}
		\la\pa_t w_v,z\ra&=\int_\Om\int_Y\big(\alpha\nabla\pa_t v\cdot\nabla z+(\mathsf B\pa_t v)z\big)\,\dd\nu\dd x
		+\int_\Om{\sf f}''(V)\pa_t V\,Z\,\dd x\\
		&=-\int_\Om\int_Y\alpha\Delta z\,\divv(v\nabla w_v)\,\dd\nu\dd x-\mathrm{err}_3\\
		&\quad-\int_\Om\int_Y v\nabla w_v\cdot\big({\sf f}''(U)\nabla U-{\sf f}''(V)\nabla V+\mathsf B\nabla z\big)\,\dd\nu\dd x,
	\end{align*}
where we integrated by parts using $\pa_nz=\pa_nw_v=0$.
	Splitting $v=u-z$ in the two integrals on the right-hand side and using the constitutive flux identity with $\xi=\nabla w_v$, together with the first formula for $\mathrm{err}_1$, the definition of $\mathrm{err}_4$, and the identity $-\alpha\Delta v+{\sf f}'(V)+\mathsf Bv=w_v$, we deduce
	\[\la\pa_t w_v,z\ra=\la j_u,\nabla w_v\ra+\int_\Om\int_Yu|\nabla w_v|^2\,\dd\nu\dd x	-\mathrm{err}_1-\mathrm{err}_3-\mathrm{err}_4.\]
	Consequently, by the definition of $\mathrm{err}_2$,
	\[\frac{\dd}{\dd t}E_{\mathrm{aff}}=2\la j,\nabla w_v\ra+\int_\Om\int_Yu|\nabla w_v|^2\,\dd\nu\dd x
	-\sum_{m=1}^4\mathrm{err}_m.\]
We now invoke the energy-dissipation inequality for $u$, which implies
	\[E_\rel(s)\Big|_{s=0}^{s=t}+\int_0^t\bigg(\int_\Om\int_Y\frac{|j|^2}{u}\,\dd\nu\dd x
	+\frac{\dd}{\dd s}E_{\mathrm{aff}}\bigg)\,\dd s\le0.\]
In view of 
	\[\int_\Om\int_Y\frac{|j|^2}{u}\,\dd\nu\dd x+\frac{\dd}{\dd t}E_{\mathrm{aff}}
	=\int_\Om\int_Yu\bigg|\frac ju+\nabla w_v\bigg|^2\,\dd\nu\dd x-\sum_{m=1}^4\mathrm{err}_m,\]
this yields the asserted relative energy inequality.
	
	The second form of $\mathrm{err}_1$ for $z=u-v$ is obtained analogously as~\eqref{eq:flux-tensor}, expanding the divergence and integrating by parts 
\begin{align*}
	-\int_\Om\divv(z \nabla w_v)\Delta z\,\dd x&
	=-\int_\Om(\nabla w_v\cdot\nabla z)\Delta z\,\dd x-\int_\Om z(\divv \nabla w_v)\Delta z\,\dd x\\
	&=\int_\Om\nabla(\nabla w_v\cdot\nabla z)\cdot\nabla z\,\dd x+\int_\Om\nabla\bigl(z\divv \nabla w_v\bigr)\cdot\nabla z\,\dd x\\
	&=\int_\Om\Big(\nabla^2 w_v:(\nabla z\otimes\nabla z)+\frac12(\divv \nabla w_v)|\nabla z|^2+z\nabla z\cdot\nabla\divv \nabla w_v	\Big)\,\dd x.
\end{align*}
Here, we also used the vanishing of the boundary terms due to $\partial_nz=\partial_nw_v=0$.

This yields the asserted relative energy inequality.
\end{proof}

\begin{proof}[Proof of Theorem~\ref{thm:wk-str.uniq.e-diss}.]	From now on we restrict to $d\le 5$, and  choose $\lambda>0$ large enough according to Lemma~\ref{l:aug-relen.coerc}. In view of $E_\rel=E_\dft+\mathcal{E}_\rel(u|v)$, Lemma~\ref{l:aug-relen.coerc} then yields
	\begin{align}\label{eq:Srel.coerc}
		S_\rel\gtrsim E_\dft+ \|z\|_{L^2(Y;H^1(\Omega))}^2+\|Z\|_{L^4(\Omega)}^4.
	\end{align}
Let $R:\mathcal V\to\mathcal V^*$ denote the Riesz isomorphism.
Observing the regularity $z\in H^1(0,T;L^2(Y;\mathcal{V}^*))$, and using~\eqref{eq:CE.diagE} and $\pa_tv=\divv(v\nabla w_v)$, we compute for a.e.\ $t\in(0,T)$
\begin{align*}
\frac12\frac{\dd}{\dd t}\|z\|^2_{L^2(Y;\mathcal{V}^*)}&=\int_Y\la\pa_tz,R^{-1}z\ra\,\dd\nu=\int_Y\int_\Om(j+v\nabla w_v)\cdot\nabla(R^{-1}z)\,\dd x\dd\nu
\\&=\int_Y\int_\Om(j+u\nabla w_v)\cdot\nabla(R^{-1}z)\,\dd x\dd\nu
-\int_Y\int_\Om z\nabla w_v\cdot\nabla(R^{-1}z)\,\dd x\dd\nu
\\&\le\delta \int_Y\int_\Om u\bigg|\frac{j}{u}+\nabla w_v\bigg|^2\,\dd x\dd\nu
+C_\delta\int_Y\int_\Om u|\nabla(R^{-1}z)|^2 \,\dd x\dd\nu+C(\|\nabla w_v\|_{L^\infty})\|z\|_{L^2_{x,y}}^2
\end{align*}
For the third term in the last line, we note that $\|z\|_{L^2_{x,y}}^2\lesssim S_\rel$ by~\eqref{eq:Srel.coerc}.
To handle the second term, we estimate, using $\int_\Om u(x,y)\,\dd x=1$ for $\nu$-a.e.\ $y$,
\begin{align*}
\int_Y\int_\Om u|\nabla(R^{-1}z)|^2 \,\dd x\dd\nu &\le\int_Y\bigg(\int_\Om u\,\dd x\|\nabla(R^{-1}z)\|^2_{L^\infty(\Om)}\bigg)\dd\nu
\\&=\int_Y\|\nabla(R^{-1}z)\|^2_{L^\infty(\Om)}\dd\nu\lesssim\int_Y\|R^{-1}z\|^2_{\mathcal{V}}\,\dd\nu=
\int_Y\|z\|^2_{\mathcal{V}^*}\,\dd\nu\lesssim S_\rel.
\end{align*}
Thus, in combination with Lemma~\ref{l:EDI.rel}, we infer for $\delta=\frac{1}{2\lambda}$,
	\begin{align*}
	S_\rel(t) +	\frac12\int_0^t\int_\Om\int_Yu\left|\frac{j}{u}+\nabla w_v\right|^2\,\dd\nu\dd x\dd s&\le S_\rel(0) + \sum_{m=1}^4 \int_0^t\mathrm{err}_m\,\dd s+C\int_0^tS_\rel(r)\,\dd r.
\end{align*}
For $\mathrm{err}_1$, we use the form~\eqref{eq:grad-err.v2} and the coercivity~\eqref{eq:Sdef<Srel.e} to infer
	\[|\mathrm{err}_1|\le C_v\|z\|_{L^2(Y;H^1(\Omega))}^2\le C_vS_\rel.\]
For $\mathrm{err}_2$, we have $|\mathrm{err}_2|\lesssim E_{\dft}(t)\lesssim S_\rel$ by property~\eqref{eq:def.e-diss.flux} and the coercivity~\eqref{eq:Srel.coerc}.

For $\mathrm{err}_3$, using~\eqref{eq:pot.exp} 
and the boundedness of $V$ and $\partial_tV$, we infer
\begin{align*}
	|\mathrm{err}_3|\lesssim\vartheta\int_\Om\big(|Z|^2+|Z|^3\big)\,\dd x\lesssim\vartheta \|Z\|_{L^2(\Om)}^2+\vartheta\|Z\|_{L^4(\Om)}^4\lesssim S_\rel,
\end{align*}
where we used~\eqref{eq:Sdef<Srel.e} in the last step.

For the potential contribution to $\mathrm{err}_4$, we rely on $d\le 5$. 
To slightly simplify the estimate, it is convenient to integrate by parts, using $\partial_nw_v=0$, and rewrite
\[\mathrm{err}_4=\int_\Om\int_Y\divv(z\nabla w_v)({\sf f}'(U)-{\sf f}'(V)+\mathsf Bz)\,\dd\nu\dd x.\]
By~\eqref{eq:pot.exp} we have the pointwise bound $|{\sf f}'(U)-{\sf f}'(V)|\lesssim_v\vartheta(1+U^2)|Z|$ so that, invoking
H\"older's inequality, the embeddings $H^1(\Omega)\hookrightarrow L^{10/3}(\Omega)$, $H^2(\Omega)\hookrightarrow L^{10}(\Omega)$, valid for $d\le5$, 
and the bound $\|U\|_{L^{10}(\Omega)}\lesssim \|u\|_{L^2(Y;L^{10}(\Omega))}$,
we obtain
	\begin{align*}
		\|{\sf f}'(U)-{\sf f}'(V)\|_{L^2(\Omega)}\lesssim\vartheta\big(1+\|U\|_{L^{10}(\Omega)}^2\big)\|Z\|_{L^{10/3}(\Omega)}\lesssim\vartheta\big(1+\|u\|_{L^2(Y;H^2(\Omega))}^2\big)\|z\|_{L^2(Y;H^1(\Omega))}.
	\end{align*}
Estimating further
\[\|\divv(z\nabla w_v)\|_{L^2(\Om\times Y)}\le C_v\|z\|_{L^2(Y;H^1(\Om))},\]
and $\|\mathsf Bz\|_{L^2(\Om\times Y)}\le\beta^*\|z\|_{L^2(\Om\times Y)}$, we conclude via Cauchy--Schwarz
\begin{align}\label{eq:err4.e.bound}
	|\mathrm{err}_4|\lesssim_v\big(1+\|u\|_{L^2(Y;H^2(\Omega))}^2\big)S_\rel.
\end{align}
	Here $\|u\|_{L^2(Y;H^2(\Omega))}^2\in L^1(0,T)$ by \eqref{eq:wsu.e.regularity}. 

Overall, we conclude, for a.e.\ $t\in(0,T)$,
\begin{align}\label{eq:stab.est.e}
	S_\rel(t)+\frac12\int_0^t\int_\Om\int_Yu\left|\frac{j}{u}+\nabla w_v\right|^2\dd\nu\dd x\dd s
\le S_\rel(0)+C_v\int_0^t\big(1+\|u\|_{L^2(Y;H^2(\Omega))}^2\big)S_\rel\,\dd s.
\end{align}
Due to $v(0,\cdot)=u_0$, we have $S_\rel(0)=\mathcal{E}_{\rel}(u_0|v(0,\cdot))+\frac\lambda2\|u_0-v(0,\cdot)\|_{L^2(Y;\mathcal{V}^*)}^2=0$. Hence, Gronwall's inequality yields $S_\rel=0$ a.e. By \eqref{eq:Srel.coerc}, this implies $u=v$ and $E_\dft=0$ a.e. Consequently, $j=j_u=j_v$, which completes the proof of  Theorem~\ref{thm:wk-str.uniq.e-diss}.
\end{proof}

\begin{remark}
	If 
    $\vartheta=0$, the stability estimate, with the potential contribution omitted from the coercivity estimates, remains true in every space dimension $d\in \mathbb{N}_+$. 
\end{remark}

\subsection{Dissipative solutions on \texorpdfstring{$X=\mathbb{R}^d$}{X=Rd}}\label{ssec:wsu.diag.diss}

Below, we let $v$ be as in Theorem~\ref{thm:wk-str.uniq}, abbreviate $V=\int_Yv\,\dd\nu$, $w_v=-\alpha\Delta v+{\sf f}'(V)+\mathsf Bv$, and write $z=u-v$, $Z=U-V$ for any dissipative solution $(u,j,E,H)$.

\paragraph{Generalised distance and coercivity}
We proceed similarly as in the previous subsection, but now take $S_\rel:=E_\rel+\lambda H_\rel$ as the augmented relative energy with $H_\rel$ denoting the relative entropy
\begin{align}
H_\rel(t):=H(t)-\int_X\int_Y(\log v+1)z\,\dd x\dd\nu-\mathcal{H}(v),\;t>0,  
\end{align}
where $\lambda\ge1$ will be fixed below. Here we again adopt the convention that $E(0):=\mathcal{E}(\mu_0), H(0):=\mathcal{H}(\mu_0)$, and we write $\mathcal{H}(v)$ for $\mathcal{H}(v\scrL^d\otimes\nu)$ etc.
We first establish an appropriate coercivity property of $S_\rel$. 
Using the coercivity~\eqref{eq:pot-rel.coerc} of the relative potential term,
the symmetry of $\mathsf B$, its operator norm bound $\|\mathsf B\|\le\beta^*$, and $\|Z\|_{L^2(X)}\le\|z\|_{L^2(X\times Y)}$, we have
\begin{align*}
	E_\rel&=E_\dft+\frac12\int_X\int_Y\alpha|\nabla z|^2\,\dd\nu\dd x
	+\int_X\big({\sf f}(U)-{\sf f}(V)-{\sf f}'(V)Z\big)\,\dd x
	+\frac12\int_X(\mathsf Bz,z)_{L^2(Y,\nu)}\,\dd x\\
	&\ge E_\dft+\frac{\alpha_*}{2}\|\nabla z\|_{L^2(X\times Y)}^2
	+\frac{\vartheta}{12}\|Z\|_{L^4(X)}^4
	-\bigg(\frac\vartheta8+\frac{\beta^*}{2}\bigg)\|z\|_{L^2(X\times Y)}^2.
\end{align*}
Invoking Gagliardo--Nirenberg and Young's inequality, applied fibrewise, and then the Csisz\'ar--Kullback--Pinsker inequality (observing that $u(\cdot,y)$ and $v(\cdot,y)$ have unit mass, which for $v(\cdot,y)$ follows from the equation), we can estimate for every $\varepsilon>0$,
\begin{align*}
\|z\|_{L^2(X\times Y)}^2&\le\varepsilon\|\nabla z\|_{L^2(X\times Y)}^2+C_\varepsilon\int_Y\|z(\cdot,y)\|_{L^1(X)}^2\,\dd\nu(y)
	\le\varepsilon\|\nabla z\|_{L^2(X\times Y)}^2+C_\varepsilon H_\rel.
\end{align*}
Choosing $\varepsilon$ sufficiently small and then $\lambda$ sufficiently large, we deduce
\begin{align}\label{eq:Sdef<Srel}
	E_\dft+H_\dft+\|z\|_{L^2(X\times Y)}^2+\|\nabla z\|_{L^2(X\times Y)}^2+\vartheta\|Z\|_{L^4(X)}^4\lesssim S_\rel.
\end{align}

\paragraph{Evolution of relative energy and entropy}

First, we note that the relative energy inequality in Lemma~\ref{l:EDI.rel} with $\Om$ replaced by $\mathbb{R}^d$ remains valid for dissipative solutions on $X=\mathbb{R}^d$ with analogous proof.
We next establish a corresponding inequality for the evolution of $H_\rel$.

\begin{lemma}[Relative entropy inequality]\label{lem:rel-entr}
	For a.e.\ $t\in(0,T)$, it holds
	\begin{align*}
		H_\rel(t)+\int_0^t\mathcal{D}_\rel(u;u-v)\,\dd s&\le H_\rel(0)+\sum_{m=1}^4\int_0^t\overline{\mathrm{err}}_m\,\dd s,
	\end{align*}
	where, for $z:=u-v$, $Z:=\int_Y(u-v)\,\dd\nu$,
	\begin{align*}
	\mathcal{D}_\rel(u;z)=\int_X\int_Y\alpha|\Delta z|^2\,\dd\nu\dd x
	+\int_X{\sf f}''_\cvx(U)|\nabla Z|^2\,\dd x
	\end{align*}
	and 
	\begin{align*}
		\overline{\mathrm{err}}_1&\coloneqq \int_X\int_Y\alpha\divv(z\nabla\log v)\Delta z\,\dd\nu\dd x,\\
		\overline{\mathrm{err}}_2&\coloneqq -\la j-j_u,\nabla\log v\ra,\\
	\overline{\mathrm{err}}_3&=-\int_X\big({\sf f}''_\ccv(U)|\nabla Z|^2+%
(\sff''(U)-\sff''(V))	\nabla V\cdot\nabla Z\big)\,\dd x
	-\int_X(\mathsf B\nabla z,\nabla z)_{L^2(Y)}\dd x,\\
\overline{\mathrm{err}}_4&:= \int_X\int_Y	z\nabla\log v\cdot\big({\sf f}''(U)\nabla U-{\sf f}''(V)\nabla V+\mathsf B\nabla z\big)\,\dd\nu\dd x.
	\end{align*}
\end{lemma}

\begin{proof}
We first determine the evolution of the affine correction
\[H_{\mathrm{aff}}:=\mathcal H(v)
+\int_X\int_Y\log vz\,\dd\nu\dd x,
\qquad H_\rel=H-H_{\mathrm{aff}}.\]
Using the continuity equation~\eqref{eq:CE.diagE} for $u$, the chain rule for $\mathcal H(v)$ and fibrewise mass conservation so that $\int_Xz\,\dd x=0$ for $\scrL\otimes\nu$-a.e.\ $(t,y)$, we compute similarly as in~\eqref{eq:Eaff.evol}, for a.e.\ $0<t<T$,
\begin{align*}
	\frac{\dd}{\dd t}H_{\mathrm{aff}}=\int_X\int_Yz\partial_t\log v\,\dd\nu\dd x+\la j,\nabla\log v\ra.
\end{align*}
Here, the choice $\phi=\log v$ in~\eqref{eq:CE.diagE} on any time interval $(0,T'), T'<T$, can be justified by approximating $\log v$ by  $\rho_R\log v$, where $0\le\rho_R\le 1$, $R\gg1$, is a smooth space-time cut-off satisfying $\rho_R(t,x)=1$ for $0\le t\le T{-}\frac1R$, $|x|\le R$, $\supp\rho_R\subset [0,T)\times B_{2R}^X(0)$, and $\|\nabla_x\rho_R\|_{L^\infty}\lesssim\frac1R$.

Using $\partial_t\log v=\Delta w_v+\nabla\log v\cdot\nabla w_v$ and integrating by parts, we obtain
\begin{align*}
	\int_X\int_Y z\partial_t\log v\,\dd\nu\dd x
	&=-\int_X\int_Y\alpha\Delta v
	\bigl(\Delta z-\divv(z\nabla\log v)\bigr)\,\dd\nu\dd x\\
	&\quad-\int_X\int_Y
	(\nabla z-z\nabla\log v)\cdot
	\bigl({\sf f}''(V)\nabla V+\mathsf B\nabla v\bigr)
	\,\dd\nu\dd x.
\end{align*}
On the other hand, the constitutive flux identity with $\xi=\nabla\log v$, combined with
$u\nabla\log v=\nabla v+z\nabla\log v$, yields
\begin{align*}
	\la j_u,\nabla\log v\ra&=-\int_X\int_Y\alpha\Delta u\big(\Delta v+\divv(z\nabla\log v)\big)\,\dd\nu\dd x\\
	&\quad-\int_X\int_Y(\nabla v+z\nabla\log v)\cdot\big({\sf f}''(U)\nabla U+\mathsf B\nabla u\big)	\,\dd\nu\dd x.
\end{align*}
Since $\la j,\nabla\log v\ra=\la j_u,\nabla\log v\ra-\overline{\mathrm{err}}_2$, adding these identities and collecting terms, using the symmetry of $\mathsf B$, we obtain
\begin{align*}
	\frac{\dd}{\dd t}H_{\mathrm{aff}}&=-\int_X\int_Y
	\alpha\big(2\Delta u\Delta v-|\Delta v|^2\big)
	\,\dd\nu\dd x-\int_X\big(
	{\sf f}''(U)\nabla U\cdot\nabla V
	+{\sf f}''(V)\nabla V\cdot\nabla Z\big)\,\dd x\\
	&\quad-\int_X(\mathsf B\nabla v,2\nabla u-\nabla v)_{L^2(Y,\nu)}\,\dd x	-\sum_{m\in\{1,2,4\}}\overline{\mathrm{err}}_m.
\end{align*}
Combining this with the entropy-dissipation inequality~\eqref{eq:entr.diss}, rearranging terms, and splitting
\begin{align*}
	\big({\sf f}''(U)\nabla U-{\sf f}''(V)\nabla V\big)\cdot\nabla Z
	&=\sff''(U)|\nabla Z|^2+(\sff''(U)-\sff''(V))\nabla V\cdot\nabla Z,
\end{align*}
we infer the asserted  relative entropy inequality.
\end{proof}

\paragraph{Weak--strong stability estimate for dissipative solutions}

With the above bounds, we are now in a position to prove the weak--strong uniqueness principle.
\begin{proof}[Proof of Theorem~\ref{thm:wk-str.uniq}.]
By the  coercivity~\eqref{eq:Sdef<Srel}, it suffices to show, for a.e.\ $t\in(0,T)$,
	\begin{align}\label{eq:stab.est}
		S_\rel(s)\Big|_{s=0}^{s=t}\lesssim_v\int_0^tS_\rel(s)\,\dd s.
	\end{align}	
	Indeed, once~\eqref{eq:stab.est} is established, Gronwall's inequality and the fact that $S_\rel(0)=\mathcal{E}_{\rel}(u_0|v(0,\cdot))+\lambda\mathcal{H}_{\rel}(u_0|v(0,\cdot))=0$ imply $S_\rel=0$ a.e.\ in $(0,T)$. By~\eqref{eq:Sdef<Srel}, this implies $u=v$ and $E_\dft=H_\dft=0$. The flux-defect estimate~\eqref{eq:mv-flux.control} then yields $j=j_u$.
	
	To establish~\eqref{eq:stab.est}, we add the relative energy inequality (cf.~Lemma~\ref{l:EDI.rel} with $\Om$ replaced by $\mathbb{R}^d$) and the inequality for $\lambda H_\rel$ from Lemma~\ref{lem:rel-entr}, to infer
	\begin{align}\label{eq:Srel_errors}
		S_\rel(s)\Big|_{s=0}^{s=t}
		&+\int_0^t\int_X\int_Y
		u\bigg|\frac ju+\nabla w_v\bigg|^2\,\dd\nu\dd x\dd s
		+\lambda\int_0^t\mathcal D_\rel(u;z)\,\dd s\le\sum_{m=1}^4\int_0^t
		\big(\mathrm{err}_m+\lambda\overline{\mathrm{err}}_m\big)\,\dd s.
	\end{align}
By the flux-defect estimate~\eqref{eq:mv-flux.control} and the coercivity~\eqref{eq:Sdef<Srel}, we have
	\[\bigg|\int_0^t\big(\mathrm{err}_2+\lambda\overline{\mathrm{err}}_2\big)\,\dd s\bigg|
	\lesssim_v\int_0^t(E_\dft+H_\dft)\,\dd s\lesssim\int_0^tS_\rel\,\dd s.\]
Next, we  estimate, using the bound $\|\divv(z\xi)\|_{L^2(X\times Y)}\lesssim_vS_\rel^{1/2}$ for $\xi\in\{\nabla w_v,\nabla\log v\}$, 
	\[|\mathrm{err}_1|+\lambda|\overline{\mathrm{err}}_1|\le\frac{\lambda}{2}\int_X\int_Y\alpha|\Delta z|^2\,\dd\nu\dd x+C_vS_\rel.\]	
	For the lower-order terms, we use the boundedness of $V$, $\nabla V$, and $\partial_tV$, which follows from the regularity hypotheses on the strong solution. Let
	\[\mathcal K\coloneqq\int_XU^2|\nabla Z|^2\,\dd x,\qquad q\coloneqq\int_Y|z|\,\dd\nu.	\]
	By the entropy estimate, $U\nabla U\in L^2((0,T)\times X)$, so $\mathcal K\in L^1(0,T)$. Moreover, by the non-negativity and $\nu(Y)=1$, it holds that $q\le U+V$ and $q^2\le\int_Y|z|^2\,\dd\nu$.
	Splitting into $\{U\le2\|V\|_\infty+1\}$ and its complement, where $Z=U-V>\frac{U}{2}$, we deduce
	\begin{align}\label{eq:wsu.weighted-q}
		(1+U^2)q^2\le C_v\bigg(\int_Y|z|^2\,\dd\nu+|Z|^4\bigg).
	\end{align}
To estimate the terms involving $\sff''(U)-\sff''(V)$, we use identity~\eqref{eq:pot.aux}.
We then obtain, for every $\delta>0$,
	\begin{align*}
		|\mathrm{err}_4|+\lambda|\overline{\mathrm{err}}_4|
		&\le C_v\int_X\bigg((1+U^2)q|\nabla Z|+(1+U)|Z|q\bigg)\,\dd x
		+C_v\beta^*\|z\|_{L^2(X\times Y)}\|\nabla z\|_{L^2(X\times Y)}\\
		&\le\delta\mathcal K+C_{\delta,v}S_\rel,
	\end{align*}
where the last step follows from Young's inequality,~\eqref{eq:wsu.weighted-q}, and~\eqref{eq:Sdef<Srel}.
	
	For $\mathrm{err}_3$, the expansion~\eqref{eq:pot.exp} yields
	\[|\mathrm{err}_3|\le C_v\int_X\big(|Z|^2+|Z|^4\big)\,\dd x\le C_vS_\rel.\]
	Finally, using ${\sf f}''_\ccv=-\vartheta$,  identity~\eqref{eq:pot.aux}, and 
	$\|\mathsf B\|\le\beta^*$, we infer, for every $\delta>0$,
	\begin{align*}
		\lambda|\overline{\mathrm{err}}_3|&\le C_v\bigg(\|\nabla z\|_{L^2(X\times Y)}^2
		+\int_X(1+U)|Z||\nabla Z|\,\dd x\bigg)\le\delta\mathcal K+C_{\delta,v}S_\rel.
	\end{align*}
	
We insert these estimates into~\eqref{eq:Srel_errors}, choose $\delta=\lambda\vartheta/4$ and use ${\sf f}''_\cvx(U)\ge\frac32 \vartheta U^2$ to infer $2\delta\mathcal K\le\frac{\lambda}{2}\int_X{\sf f}''_\cvx(U)|\nabla Z|^2\,\dd x$, to obtain, after absorption, the bound
\begin{align*}
	S_\rel(t)&+\int_0^t\int_X\int_Yu\bigg|\frac ju+\nabla w_v\bigg|^2\,\dd\nu\dd x\dd s
	+\frac{\lambda}{2}\int_0^t\mathcal D_\rel(u;z)\,\dd s\le S_\rel(0)+C_v\int_0^tS_\rel\,\dd s.
\end{align*}
This shows~\eqref{eq:stab.est} and thus concludes the proof.	
\end{proof}

\section*{Acknowledgements}
The authors thank Jos\'e A.\ Carrillo and Markus Mittnenzweig for several stimulating discussions on the modelling.
\appendix

\section{Appendix}

\subsection{A short review of fibred optimal transport}\label{ssec:Notation}
In this section, we collect some selected results from optimal transport theory that will be used throughout the paper. In particular, a central ingredient is the so-called 
\emph{fibred Wasserstein distance} introduced in \cite{peszek2023heterogeneous},
see also \cite{CHSW2025CWDA}, where the name \emph{conditional Wasserstein distance} is used.

\begin{theorem}[{Properties of $\calW_{2,\nu}$, \cite[Sec. 3]{peszek2023heterogeneous}}]Let $\nu \in \calP(Y)$ be any probability
measure.
\begin{enumerate}[label=(\roman*)]
\item $\calW_{2,\nu}$ is a well-defined, finite distance on $\calP_{2,\nu}(X{\times} Y)$. 
\item The fibred space $(\calP_{2,\nu}(X{\times}Y), \calW_{2,\nu})$ is a Polish space.
\item For all $\mu_1,\mu_2\in\calP_{2,\nu}(X{\times}Y)$, there exists a measurable family $\{\gamma^y\}_{y\in Y} \subset \calP(X{\times} X)$ such that $\gamma^y$ is an optimal transport plan for $\mu_1^y$ and $\mu_2^y$ for $\nu$-a.e.\ $y\in Y$.
\item For $\mu_n,\wt \mu_n \in \calP_{2,\nu}(X{\times}Y)$ 
with $\mu_n\weakstar \mu$ and $\wt\mu_n \weakstar \wt\mu$ narrowly in $\calP_{2,\nu}(X{\times}Y)$, we have that
\[
\liminf_{n\to \infty} \calW_{2,\nu}(\mu_n,\wt \mu_n) \geq \calW_{2,\nu}(\mu,\wt\mu).
\]
\end{enumerate}
\end{theorem}

Similar to the classical Wasserstein case in \cite[Sect. 8.4]{AGS}, 
we can characterize the tangent space 
in $\mu\in\calP_{2,\nu}(X{\times Y})$.
The latter is defined as the closure of smooth gradient fields with respect to $L^2(X{\times} Y;\mu)$
norm, viz.
\begin{equation}
\mathrm{Tan}_\mu(\calP_{2,\nu}(X{\times}Y))
\coloneqq  \overline{\big\{\nablax\varphi|_{X\times Y}\,|\,\varphi\in C^\infty_c(\R^d{\times} \R^{d'})\big\}}^{L^2((X{\times}Y);\mu)}.
\end{equation}
The following results are proven in \cite[Propositions 3.20 and 3.21]{peszek2023heterogeneous}
\begin{proposition}\label{prop:contEqn}
    Let $\nu\in\calP(Y)$.
    \begin{enumerate}[label=(\roman*) ]
        \item  For $\mu\in\calP_{2,\nu}(X{\times}Y)$, it holds 
        \[
        \mathrm{Tan}_\mu(\calP_{2,\nu}(X{\times}Y)) = \big\{ v \in L^2(X{\times} Y;\mu)^d\,|\, v(\cdot ,y) \in\mathrm{Tan}_{\mu^y} (\calP_2(X)) \text{ for } \nu\text{-a.e.\ }y\in Y\big\}.
        \]
        \item For $\mu\in \mathrm{AC}^p(0,T;\calP_{2,\nu}(X{\times}Y))$, $1\leq p\leq \infty$, $T>0$, there exists a Borel family of
        vector fields $v_t \in \mathrm{Tan}_{\mu_t}(\calP_{2,\nu}(X{\times}Y))$ with $\int_0^T\|v_t \|^p_{L^2(X{\times} Y;\mu_t)} \dd t < \infty$ 
        such that $(\mu_t , v_t )$ satisfies the $x$-continuity equation
\begin{equation}\label{eq:contEqn}
\partial_t \mu_t +\divv_x  (v_t\mu_t)=0,\quad t\in(0,T),
\end{equation}
in the distributional sense. Moreover, we have that
$\|v_t\|_{L^2(X{\times} Y;\mu_t)}\leq |\mu'|_{\calW_{2,\nu}}(t)$ for a.e.\ $t\in [0,T]$, where $|\mu'|_{\calW_{2,\nu}}$ denotes the metric derivative of $\mu_t$ with respect  to $\calW_{2,\nu}$.
\item If $\mu_t$ is a Borel family of probability measures in $\calP_{2,\nu}(X{\times}Y)$, $v_t\in L^2(X{\times} Y;\mu_t)^d$ is a Borel family of velocity fields with $\int_0^T\|v_t\|^p_{L^2(X{\times} Y;\mu_t)}\dd t<\infty$, 
and both satisfy \eqref{eq:contEqn}, then $\mu\in \mathrm{AC}^p
(0, T ;\calP_{2,\nu}(X\times Y ))$ and 
$|\mu'|_{\calW_{2,\nu}} (t) \leq \|v_t \|_{L^2(X{\times} Y;\mu_t)}$  
for a.e.\ $t \in (0, T )$.

    \end{enumerate}
\end{proposition}

The differentiability of $\calW_{2,\nu}$ along absolutely continuous curves follows from the classical differentiability 
result for the Wasserstein distance $\calW_2$ on $\calP_2(X)$
and is proven in \cite[Proposition 3.22]{peszek2023heterogeneous}.

\begin{proposition}\label{prop:time-derivative_metric}
    Let  $\mu \in \mathrm{AC}(0,T;\calP_{2,\nu}(X{\times}Y))$,  $\sigma\in\calP_{2,\nu}(X{\times}Y)$, 
    and $v_t\in\mathrm{Tan}_{\mu_t}(\calP_{2,\nu}(X{\times}Y))$
 be the vector field tangent to $\mu$ characterized by Proposition~\ref{prop:contEqn}(ii).  Then
\[
\frac{\mathrm{d}}{\mathrm{d}t}\frac{1}{2}\calW_{2,\nu}^2(\mu_t,\sigma)
= \int_{(X{\times}Y)^2} (x-x')\cdot v_t(x,y)\,\dd\gamma_t(x,x',y,y')
\]
for a.e.\  $t\in[0,T]$ and any $\gamma_t\in\Gamma_{o,\nu}(\mu_t,\sigma)$, an optimal $\nu$-admissible transport plan for $\calW_{2,\nu}$.    
\end{proposition}

\subsection{Interpolation inequalities}

\begin{lemma}\label{l:interpolation}
	\begin{enumerate}[label={\normalfont(\roman*)}] 
		\item 
			Let $1< q\le2+\frac2d$ and abbreviate 		$\theta\coloneqq \frac{2d}{d+2}\frac{q-1}{q}\in(0,1)$, so that $\frac1q=\theta\Big(\frac12-\frac1d\Big)+(1-\theta)$. 
			
			Then,  for all $p\in[1,\infty)$ with $\theta p\le r$, where $r\in[1,\infty]$, and $z\in L^\infty((0,T)\times Y;L^1(X))$
		\begin{align}\label{eq:interpol.H1-M}
			\|z\|_{L^p(0,T;L^q(X{\times}Y))}\lesssim_{T,\theta,q} M_z^{1-\theta}\|\nablax z\|_{L^r(0,T;L^2(X{\times}Y))}^{\theta},
		\end{align}
		where $M_z=\esssup_{t,y}\int_X|z|\,\dd x<\infty$. 
		\item Let $\mu_1,\mu_0:[0,T]\to\calP_{2,\nu}(X{\times}Y)$ be such that, for a.e.\ $t\in[0,T]$, 
		\[
        \mu_i(t)=u_i(t)\scrL^d{\otimes}\nu,\qquad u_i\in L^2([0,T]\times Y;H^1(X)),\quad i=0,1.
        \]
		Then,
			\begin{align}\label{eq:interpol.H1-Hrel}
			\|u_0-u_1\|_{L^2([0,T]\times X\times Y)}^2\lesssim \|\nablax(u_0-u_1)\|_{L^2([0,T]\times X\times Y)}^{\frac{2d}{d+2}}\bigg(\int_0^T\mathcal{H}_{\rel}(u_0|u_1)\,\dd t\bigg)^{\frac{2}{d+2}}.
		\end{align}
	\end{enumerate}
\end{lemma}
\begin{proof}
	{\em Re~}(i):
By the Gagliardo--Nirenberg inequality, 
\begin{align}\label{eq:gns.X}
	\|z(t,\cdot,y)\|_{L^q(X)}^q\lesssim_q 	\|\nablax z(t,\cdot,y)\|_{L^2(X)}^{\theta q}\|z(t,\cdot,y)\|_{L^1(X)}^{(1-\theta)q}.
\end{align}
We integrate in $Y$, and apply H\"older's inequality for a term on the right-hand side to infer
\begin{align*}
	\|z(t,\cdot)\|_{L^q(X{\times}Y)}^q&\lesssim_q 	\int_Y\|\nablax z(t,\cdot,y)\|_{L^2(X)}^{\theta q}\underbrace{\|z(t,\cdot,y)\|_{L^1(X)}^{(1-\theta)q}.}_{\lesssim M^{(1-\theta)q}}
	\\&\lesssim\bigg(\int_Y\|\nablax z(t,\cdot,y)\|_{L^2(X)}^2\bigg)^{\theta q/2} M^{(1-\theta)q}.
\end{align*}
where in the second step we used $\theta q\le 2$ and applied H\"older's inequality and $\nu(Y)=1$.

The assertion is obtained by taking the $L^p_t$ norm and applying once more H\"older's inequality, observing that $\theta p\le r$. 

{\em Re~}(ii): We apply~\eqref{eq:gns.X} with $q=2$ and thus $\theta=\frac{d}{d+2}$ to the difference $z=u-v$, and estimate for $\mathcal{L}\otimes\nu$-a.e.\ in $[0,T]\times Y$, using Csiszár--Kullback--Pinsker inequality,
\[\|z(t,\cdot, y)\|_{L^1(X)}^2\lesssim \int_X h(u|v)\,\dd x.\]
The conclusion then follows upon integration over $[0,T]\times Y$ and use of the H\"older's inequality.
\end{proof}

\subsection{Narrow convergence and negative Sobolev spaces}
The following theorem shows that uniform narrow convergence of finite signed measures
with uniformly bounded total variation implies uniform convergence in sufficiently
negative Sobolev norms.

\begin{theorem}\label{thm:negSobolev}
Let $\{\sigma_n(t)\}_{n\in\bbN}\subset \mathcal{M}(\R^d)$ be measurable curves of finite signed measures. Assume $\sup_{n\in\bbN}\sup_{t\in[0,T]}
|\sigma_n(t)|(\R^d)<\infty$, and that 
$\sigma_n(t)\rightharpoonup \sigma(t)$ narrowly, uniformly in $t\in[0,T]$. Then, for every $m>d/2$, we also have that
$\sup_{t\in[0,T]}
\|\sigma_n(t)-\sigma(t)\|_{H^{-m}(\R^d)}
\to0$.
\end{theorem}

\begin{proof}
First note that in the case $m >d/2$, we have the continuous embedding $H^m(\R^d)\subset C_b(\R^d)$, 
thus $\mathcal{M}(\R^d)\subset H^{-m}(\R^d)$. 
Let us consider a sequence of measurable curves $t\mapsto \sigma_n(t)\in\mathcal{M}(\R^d)$ 
with uniformly bounded total variation norm and narrowly converging, uniformly in $t\in[0,T]$, i.e.,
\[
\sup_{t\in[0,T]}
\left|
\int_{\R^d}\varphi\,\dd\sigma_n(t)
-
\int_{\R^d}\varphi\,\dd\sigma(t)
\right|
\to0,\quad \text{for every }\varphi\in C_b(\R^d).
\]
By assumption, we have that
\[
M:=\sup_n\sup_{t\in[0,T]}|\sigma_n(t)|(\mathbb R^d)<\infty, \quad\text{and}\quad
|\omega_n(t)|(\mathbb R^d)\leq2M,
\text{ where }
\omega_n(t):=\sigma_n(t)-\sigma(t).
\]
We argue via Fourier transformation: For $x,\zeta\in\R^d$, define
\[
\widehat{\omega_n(t)}(\zeta)
:=
\int_{\mathbb R^d}\mathrm{e}^{-ix\cdot\zeta}\,\dd\omega_n(t)(x).
\]
For each fixed $\zeta\in \R^d$, both $\cos(x\cdot\zeta)$ and $\sin(x\cdot \zeta)$ belong to $C_b(\R^d)$, thus, the hypothesis gives
\begin{equation}
\label{eq:FourierModeConv}
f_n(\zeta)
:=
\sup_{t\in[0,T]}
|\widehat{\omega_n(t)}(\zeta)|^2
\to 0\quad\text{and}\quad
0\leq f_n(\zeta)\leq4M^2
\qquad\text{for every }\zeta.
\end{equation}
The functions $f_n$ are measurable: for every fixed $t$, the Fourier transform of the finite measure $\omega_n(t)$ is continuous in $\zeta$, and a supremum of continuous functions is lower semicontinuous.
For $m>d/2$, the Fourier characterization of the negative Sobolev norm gives, with a normalization constant $c_d>0$
\[
\|\omega_n(t)\|_{H^{-m}(\mathbb R^d)}^2
=
c_d\int_{\R^d}
(1+|\zeta|^2)^{-m}
|\widehat{\omega_n(t)}(\zeta)|^2\,\dd\zeta.
\]
Consequently, we arrive at
\[
\sup_{t\in[0,T]}
\|\omega_n(t)\|_{H^{-m}(\R^d)}^2
\leq
c_d\int_{\mathbb R^d}
(1+|\zeta|^2)^{-m}f_n(\zeta)\,\dd\zeta.
\]
Since $
(1+|\zeta|^2)^{-m}f_n(\zeta)
\leq
4M^2(1+|\zeta|^2)^{-m}$, 
and the right-hand side is integrable precisely because $m>d/2$, the Dominated Convergence theorem yields $\sup_{t\in[0,T]}
\|\omega_n(t)\|_{H^{-m}(\R^d)}\to 0$.
\end{proof}

\printbibliography[heading=bibintoc]

\end{document}

%% file: header.tex
\usepackage{latexsym,amssymb}
\usepackage{amsmath,amsthm,amsxtra,amsbsy,accents}
\usepackage{mathtools}
\usepackage{mathrsfs}
\usepackage{bm}
\usepackage{hyperref}

\DeclareMathAlphabet{\mathup}{OT1}{\familydefault}{m}{n}

\newcommand{\eps}{\varepsilon}

\DeclarePairedDelimiter{\abs}{\lvert}{\rvert}
\DeclarePairedDelimiter{\norm}{\lVert}{\rVert}
\DeclarePairedDelimiter{\bra}{(}{)}
\DeclarePairedDelimiter{\pra}{[}{]}

\DeclarePairedDelimiter{\skp}{\langle}{\rangle}

\makeatletter
\newcommand{\customlabel}[2]{%
   \protected@write \@auxout {}{\string \newlabel {#1}{{#2}{\thepage}{#2}{#1}{}} }%
   \hypertarget{#1}{}
}
\makeatother

\numberwithin{figure}{section}

\def\calA{{\mathcal A}} \def\calB{{\mathcal B}} 
\def\calD{{\mathcal D}} \def\calE{{\mathcal E}} \def\calF{{\mathcal F}}
 \def\calH{{\mathcal H}} 
  \def\calL{{\mathcal L}}
\def\calM{{\mathcal M}} \def\calN{{\mathcal N}} 
\def\calP{{\mathcal P}}  \def\calR{{\mathcal R}}
  
 \def\calW{{\mathcal W}} 
 
\def\rmD{{\mathrm D}}  
  
  \def\rmL{{\mathrm L}}

  \def\sff{{\mathsf f}}

\def\sfy{{\mathsf y}}

 \def\sfB{{\mathsf B}} 
  \def\sfF{{\mathsf F}}
  
\def\sfJ{{\mathsf J}}

  \def\scrL{{\mathscr  L}}

\def\scrY{{\mathscr  Y}}

 \def\bbN{{\mathbb N}} 
  \def\bbR{{\mathbb R}}

\newtheorem{theorem}{Theorem}[section]
\newtheorem{lemma}[theorem]{Lemma}
\newtheorem{corollary}[theorem]{Corollary}
\newtheorem{definition}[theorem]{Definition}

\newtheorem{proposition}[theorem]{Proposition}
\newtheorem{remark}[theorem]{Remark}

\makeatletter
\newcommand{\oset}[3][0ex]{%
  \mathrel{\mathop{#3}\limits^{
    \vbox to#1{\kern-2\ex@
    \hbox{$\scriptstyle#2$}\vss}}}}
\makeatother

\DeclareMathOperator*{\esssup}{ess\,sup}

\DeclareMathOperator*{\argmin}{arg\,min}

\newcommand{\dist}{\mathrm{dist}}

\newcommand{\AC}{\mathrm{AC}}

\DeclareMathOperator\supp{supp}

\newcommand{\nablax}{\nabla} 